\documentclass[a4paper,DIV=12]{scrartcl}

\usepackage[utf8]{inputenc}
\usepackage[T1]{fontenc}
\usepackage{lmodern}
\usepackage[svgnames,dvipsnames,rgb]{xcolor} 
\usepackage{amsfonts}
\usepackage{amsmath}
\usepackage{amssymb}
\usepackage{amsthm}
\usepackage{graphicx}
\usepackage{hyperref}
\definecolor{colR}{HTML}{CC6677}
\definecolor{colB}{HTML}{6699CC}
\colorlet{colG}{DarkSeaGreen}
\colorlet{colY}{Gold}
\hypersetup{linktocpage=true,colorlinks=true,linkcolor=colB!80!black,citecolor=DarkSeaGreen!80!black,urlcolor=colB!80!black}
\usepackage{float}
\usepackage{stackrel}
\usepackage{algorithm,algorithmicx,algpseudocode}

\usepackage{tikz,pgf}
\usetikzlibrary{external}
\tikzsetexternalprefix{tikz/}
\usetikzlibrary{calc}
\usetikzlibrary{patterns}
\usetikzlibrary{backgrounds}
\usetikzlibrary{arrows.meta}
\usetikzlibrary{decorations.pathmorphing}
\usetikzlibrary{decorations.pathreplacing}
\usetikzlibrary{fadings}

\usepackage[nameinlink]{cleveref}
\usepackage{subcaption}

\newcommand{\union}[1]{\textsc{union}($#1$)}

\newtheorem{theorem}{Theorem}[section]
\newtheorem{lemma}[theorem]{Lemma}
\newtheorem{proposition}[theorem]{Proposition}
\newtheorem{corollary}[theorem]{Corollary}
\theoremstyle{definition}\newtheorem{definition}[theorem]{Definition}
\theoremstyle{definition}\newtheorem{remark}[theorem]{Remark}

\newtheoremstyle{mytheorem}{}{}{\itshape}{}{\bfseries}{}{ }{\thmname{#1}.\thmnote{ \textup{(#3)}}}
\theoremstyle{mytheorem}\newtheorem{definition*}[theorem]{Definition}

\newcommand{\Trel}{\triangle}
\newcommand{\Prel}{\square}
\newcommand{\TPrel}{\sim}
\newcommand{\constAngle}{\measuredangle_\text{const}\,}
\newcommand{\infConstAngle}{\measuredangle_\text{inf}\,}
\DeclareRobustCommand{\asc}[1]{\tikz[scale=0.25,baseline={(0,0)}]{\draw[](0, 0)--(1, 0)--(1, 1)--(0, 1)--cycle;\draw[](0, 1)--(1, 0);}(#1)}

\newcommand{\Pframework}{P-framework}
\newcommand{\Pframeworks}{P-frameworks}
\newcommand{\TPframework}{TP-framework}
\newcommand{\TPframeworks}{TP-frameworks}
\newcommand{\APclass}{angle-preserving class}
\newcommand{\APclasses}{angle-preserving classes}

\newcommand{\flexVec}[2]{\varphi_{#1,#2}}
\newcommand{\realizVec}[2]{\rho_{#1,#2}}
\newcommand{\piHalfRotation}{R}
\newcommand{\piHalfRotationMatrix}{\begin{pmatrix}
    0 & -1 \\
    1 & 0
\end{pmatrix}}

\newcommand{\Rot}[1]{\Theta\left(#1\right)}

\newcommand{\oriented}[2]{(#1,#2)}
\newcommand{\Goriented}{\vec{G}}
\newcommand{\circulation}[1]{\vec{#1}}

\newcommand{\eqwithreference}[1]{\underset{\text{\scriptsize\mbox{#1}}}{=}}

\newcommand{\cartProd}{\mathbin{\square}}

\newcommand{\sst}[2]{\left\lbrace #1\,\textbf{\upshape\textbar}\,#2 \right\rbrace} 
\newcommand{\transpose}{\intercal}

\newcommand{\blue}{\text{blue}}
\newcommand{\red}{\text{red}}

\newcommand{\NN}{\mathbb{N}}
\newcommand{\RR}{\mathbb{R}}

\newcommand{\QQ}{\mathbb{Q}}

\newcommand{\ttt}{\mathbf{t}}
\newcommand{\zerovec}{\mathbf{0}}

\newcommand{\Cn}{\mathcal{C}_n}

\newcommand{\CnAPclass}{$\Cn$-\APclass{}}
\newcommand{\CnAPclasses}{$\Cn$-\APclasses{}}
\newcommand{\CnSymmetric}{$\Cn$-symmetric}
\newcommand{\RotCn}{\Theta_{\omega}}
\newcommand{\CnRibbonPart}[1]{\rho_{#1}}

\colorlet{col1}{DarkSeaGreen}
\colorlet{col2}{colR}
\colorlet{col3}{colY!90!black}
\colorlet{col4}{colB}
\colorlet{col5}{BurlyWood!90!black}
\colorlet{col6}{MediumPurple!90!black}
\colorlet{col7}{black}
\colorlet{ncol}{DarkSeaGreen}
\colorlet{tauL}{col4}
\colorlet{tauR1}{col2}
\colorlet{tauR2}{col5}
\colorlet{colbg}{white}
\colorlet{colfg}{black}
\colorlet{colgraphv}{colfg!75!colbg}
\colorlet{colgraphe}{colfg!65!colbg}
\colorlet{colline}{colfg!50!colbg}

\tikzstyle{gvertex}=[circle, draw=colbg, fill=colgraphv, inner sep=0pt, minimum size=4pt]
\tikzstyle{smallgvertex}=[circle, draw=colbg, fill=colgraphv, inner sep=0pt, minimum size=2pt]

\tikzstyle{fvertex}=[circle,inner sep=0pt,minimum size=2.5pt,fill=colbg!80!colfg,draw=colgraphv,line width=1pt,outer sep=1pt]
\tikzstyle{smallfvertex}=[circle,inner sep=0pt,minimum size=1.5pt,fill=colbg!80!colfg,draw=colgraphv,line width=0.5pt,outer sep=0.5pt]

\tikzstyle{edge}=[line width=1.5pt,colgraphe]
\tikzstyle{smalledge}=[line width=0.75pt,colgraphe]
\tikzstyle{dedge}=[edge,-{Latex[width=3.45pt,length=5pt]}] 
\tikzstyle{nonedge}=[edge,dotted]
\tikzstyle{nonribbon}=[edge,dashed]
\tikzstyle{redge}=[edge,colR]
\tikzstyle{bedge}=[edge,colB]

\tikzstyle{hedge2}=[line width=2.5pt]
\tikzstyle{ledge}=[edge,opacity=0.25]
\tikzstyle{lgvertex}=[gvertex,opacity=0.25]
\tikzstyle{lfvertex}=[fvertex,opacity=0.25]

\tikzstyle{taur}=[line width=1.5pt,col3,-latex]
\tikzstyle{taurZero}=[circle,thick,draw=col3, fill=col3, inner sep=0pt, minimum size=4.5pt]

\tikzstyle{axes}=[colline,-latex]

\tikzstyle{brace}=[line width=1pt,colgraphe]
\tikzstyle{face}=[opacity=0.3]

\pgfdeclaredecoration{simple line}{start}
{
  \state{start}[width = +0pt,
                next state=step]{
    \pgfpathmoveto{\pgfpoint{0pt}{0pt}}
  }
  \state{step}[auto end on length    = 3pt,
               auto corner on length = 3pt,
               width=+1pt]
  {
    \pgfpathlineto{\pgfpoint{1pt}{0pt}}
  }
  \state{final}
  {}
}

\title{\vspace{-1cm}Flexibility and rigidity of frameworks \\consisting of triangles and parallelograms}
\author{Georg Grasegger\thanks{Johann Radon Institute for Computational and Applied Mathematics (RICAM), Austrian Academy of Sciences} \and
Jan Legerský\thanks{Department of Applied Mathematics, Faculty of Information Technology, Czech Technical University in Prague}}
\date{\vspace{-0.75cm}}
\renewcommand{\thefootnote}{\fnsymbol{footnote}}
\begin{document}

\maketitle
\footnotetext{This research was funded in whole, or in part, by the Austrian Science Fund (FWF) I6233. For the purpose of open access, the authors have applied a CC BY public copyright license to any Author Accepted Manuscript version arising from this submission.}
\footnotetext{The research was partially funded by the Czech Science Foundation (GAČR) project 22-04381L.}
\renewcommand{\thefootnote}{\arabic{footnote}}
\tikzexternaldisable
\begin{abstract}
A framework, which is a (possibly infinite) graph with a realization of its vertices in the plane,
is called flexible if it can be continuously deformed while preserving the edge lengths.
We focus on flexibility of frameworks in which 4-cycles form parallelograms.
For the class of frameworks considered in this paper (allowing triangles),
we prove that the following are equivalent: flexibility, infinitesimal flexibility,
the existence of at least two classes of an equivalence relation based on 3- and 4-cycles
and being a non-trivial subgraph of the Cartesian product of graphs.
We study the algorithmic aspects and the rotationally symmetric version of the problem.
The results are illustrated on frameworks obtained from tessellations by regular polygons.
\end{abstract}

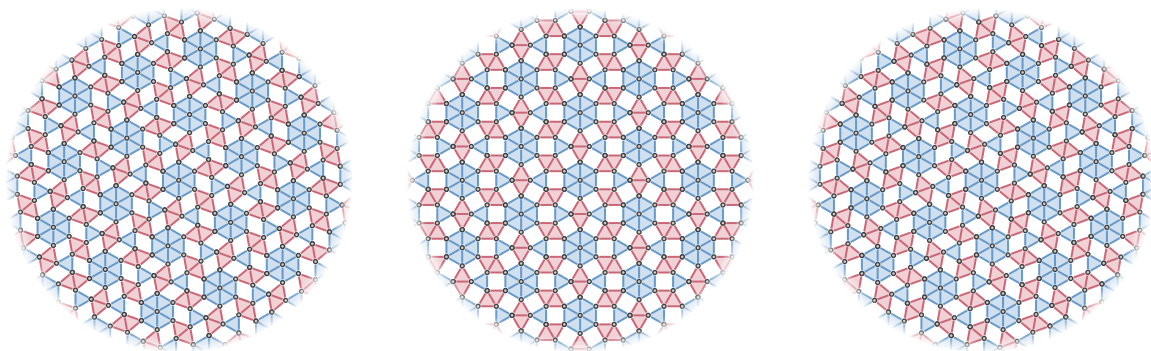
\begin{figure}[ht]
    \centering
    \tikzfading[name=fade out,inner color=transparent!100,outer color=transparent!0]
    \tikzexternalenable
    \foreach \w in {-20,0,20}
    {
        \tikzsetnextfilename{333333-33434-33434-intro\w}
        \begin{tikzpicture}[scale=0.24]
            \begin{pgfonlayer}{background}
                \clip circle[radius=9.5cm];
            \end{pgfonlayer}
            \begin{scope}
            \clip circle[radius=9.5cm];
            \begin{scope}
            \foreach \y [evaluate=\y as \xs using {Mod(\y,2)/2-2},evaluate=\xs as \xse using \xs+3] in {-5,...,6}
            {
                \foreach \x in {\xs,...,\xse}
                {
                    \begin{scope}[shift={($\y*(90:1)+\y*0.5*(\w+60:1)+\y*0.5*(\w+120:1)+\x*2*(30:1)+\x*2*(-30:1)+\x*3*(\w:1)$)}]
                        \node[smallfvertex] (a0) at (0,0) {};
                        \foreach \r [evaluate=\r as \rr using 60*\r-30] in {1,2,...,6}
                        {
                            \node[smallfvertex] (a\r) at (\rr:1) {};
                            \draw[smalledge,colB]  (a0)edge(a\r);
                        }
                        \draw[smalledge,colB] (a1)edge(a2) (a2)edge(a3) (a3)edge(a4) (a4)edge(a5) (a5)edge(a6) (a6)edge(a1);
                        \node[smallfvertex] (b1) at ($(a1)+(\w:1)$) {};
                        \node[smallfvertex] (b2) at ($(a1)+(\w+60:1)$) {};
                        \node[smallfvertex] (b3) at ($(b1)+(b2)-(a1)$) {};
                        \draw[smalledge,colR] (a1)edge(b1) (a1)edge(b2) (b1)edge(b2) (b1)edge(b3) (b2)edge(b3);
                        \node[smallfvertex] (c1) at ($(a6)+(\w:1)$) {};
                        \node[smallfvertex] (c2) at ($(a6)+(\w-60:1)$) {};
                        \node[smallfvertex] (c3) at ($(c1)+(c2)-(a6)$) {};
                        \draw[smalledge,colR] (a6)edge(c1) (a6)edge(c2) (c1)edge(c2) (c1)edge(c3) (c2)edge(c3);
                        \node[smallfvertex] (d1) at ($(a5)+(\w-60:1)$) {};
                        \node[smallfvertex] (d2) at ($(a5)+(\w-120:1)$) {};
                        \node[smallfvertex] (d3) at ($(d1)+(d2)-(a5)$) {};
                        \draw[smalledge,colR] (a5)edge(d1) (a5)edge(d2) (d1)edge(d2) (d1)edge(d3) (d2)edge(d3);

                        \node[smallfvertex] (d2s) at ($(c1)+(30:1)$) {};
                        \draw[smalledge,colB] (b1)edge(c1) (b1)edge(d2s) (c1)edge(d2s);
                        \node[smallfvertex] (b2s) at ($(d1)+(-30:1)$) {};
                        \draw[smalledge,colB] (d1)edge(c2) (d1)edge(b2s) (c2)edge(b2s);
                        \begin{pgfonlayer}{background}
                            \fill[colB,face] (a1.center)--(a2.center)--(a3.center)--(a4.center)--(a5.center)--(a6.center)--cycle;
                            \fill[colB,face] (b1.center)--(c1.center)--(d2s.center)--cycle;
                            \fill[colB,face] (d1.center)--(c2.center)--(b2s.center)--cycle;
                            \foreach \bcd/\a in {b/1,c/6,d/5}
                            {
                                \fill[colR,face] (a\a.center)--(\bcd1.center)--(\bcd3.center)--(\bcd2.center)--cycle;
                            }
                        \end{pgfonlayer}

                    \end{scope}
                }
            }
            \end{scope}
            \end{scope}
            \fill[white,path fading=fuzzy ring 15 percent] circle[radius=10.2cm];
        \end{tikzpicture}
    }
    \tikzexternaldisable
\end{figure}

When we take a look at a scaffold from the front we see a rectangular grid with a certain amount of diagonal bars. Scaffolds are intended to be rigid. From the mathematical point of view such a construction can be considered as a graph where the bars of the scaffold form the edges and their joints are the vertices.

Such a graph together with a placement of the vertices in the plane or in space gives a framework.
A framework is called flexible if there is a non-trivial flex (a deformation of the placement preserving the distances between adjacent vertices that is not induced by a rigid motion). Otherwise it is called rigid.
In this paper, all placements are in the plane.

There is an amount of papers, see below, dealing with the question
where to put diagonal bars in a grid construction in order to make it rigid.
In this paper we go a step further and allow the grid to be more general.
In particular we allow non-rectangular parallelograms and triangles in the the underlying graph
that do not come from inserting a diagonal. We show when such a framework is flexible
and give an algorithm for finding a parametrization of a flex.
We apply the theory to rotationally symmetric frameworks and illustrate it on periodic tilings by regular polygons.

\minisec{Previous work}
Grid graphs can be realized as grids of squares.
Infinitesimal rigidity thereof has been studied by Bolker and Crapo~\cite{BolkerCrapo}
considering some of the squares being braced by adding diagonals (see \Cref{fig:grids} left).
Rigidity is related to the connectivity of another graph
which consists of vertices representing rows and columns of the grid and edges when two of them are connected by a brace.
This initial setting has been generalized in different ways.

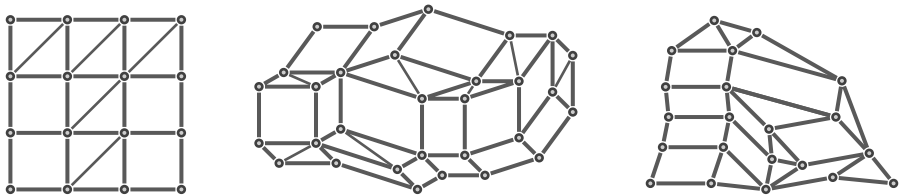
\begin{figure}[ht]
 \centering
 \begin{tikzpicture}[scale=0.75]
    \foreach \x [remember=\x as \xr (initially 1)] in {1,2,3,4}
    {
        \foreach \y [remember=\y as \yr (initially 1)] in {1,2,3,4}
        {
            \node[fvertex] (v\x\y) at (\x,\y) {};
            \ifnum\x>1
                \ifnum\y>1
                    \draw[edge] (v\x\yr)--(v\x\y);
                    \draw[edge] (v\x\y)--(v\xr\y);
                \else
                    \draw[edge] (v\x\y)--(v\xr\y);
                \fi
            \else
                \ifnum\y>1
                    \draw[edge] (v\x\yr)--(v\x\y);
                \fi
            \fi
        }
    }
    \draw[brace] (v13)--(v24);
    \draw[brace] (v23)--(v34);
    \draw[brace] (v33)--(v44);
    \draw[brace] (v22)--(v33);
    \draw[brace] (v21)--(v32);
 \end{tikzpicture}
 \qquad
 \begin{tikzpicture}[scale=0.75]
		\begin{scope}[rotate=-90]
			\node[fvertex] (0) at (0,0) {};
			\node[fvertex] (1) at (0,1) {};
			\node[fvertex,rotate around=0:(0)] (2) at ($(0)+(1.00,0)$) {};
			\node[fvertex] (3) at ($(2)+(1)-(0)$) {};
			\node[fvertex,rotate around=45:(2)] (4) at ($(2)+(0.50,0)$) {};
			\node[fvertex] (5) at ($(4)+(3)-(2)$) {};
			\node[fvertex,rotate around=120:(1)] (6) at ($(1)+(0.50,0)$) {};
			\node[fvertex] (7) at ($(6)+(3)-(1)$) {};
			\node[fvertex,rotate around=72:(7)] (8) at ($(7)+(1.50,0)$) {};
			\node[fvertex] (9) at ($(8)+(3)-(7)$) {};
			\node[fvertex] (10) at ($(5)+(9)-(3)$) {};
			\node[fvertex] (11) at ($(8)+(10)-(9)$) {};
			\node[fvertex] (12) at ($(6)+(8)-(7)$) {};
			\node[fvertex] (13) at ($(0)+(6)-(1)$) {};
			\node[fvertex,rotate around=144:(6)] (14) at ($(6)+(1.00,0)$) {};
			\node[fvertex] (15) at ($(14)+(13)-(6)$) {};
			\node[fvertex,rotate around=108:(6)] (16) at ($(6)+(1.00,0)$) {};
			\node[fvertex] (17) at ($(16)+(14)-(6)$) {};
			\node[fvertex] (18) at ($(12)+(16)-(6)$) {};
			\node[fvertex] (19) at ($(17)+(18)-(16)$) {};
			\node[fvertex,rotate around=90:(12)] (20) at ($(12)+(0.75,0)$) {};
			\node[fvertex] (21) at ($(20)+(8)-(12)$) {};
			\node[fvertex] (22) at ($(11)+(21)-(8)$) {};
			\node[fvertex] (23) at ($(18)+(20)-(12)$) {};
			\node[fvertex] (24) at ($(23)+(21)-(20)$) {};
			\node[fvertex] (25) at ($(19)+(23)-(18)$) {};
			\node[fvertex] (26) at ($(24)+(22)-(21)$) {};
			\node[fvertex] (27) at ($(25)+(24)-(23)$) {};
			\node[fvertex] (28) at ($(27)+(26)-(24)$) {};
			\node[fvertex] (29) at ($(25)+(28)-(27)$) {};
			\draw[edge] (0)--(1)
					(0)--(2) (2)--(3) (3)--(1)
					(2)--(4) (4)--(5) (5)--(3)
					(1)--(6) (6)--(7) (7)--(3)
					(7)--(8) (8)--(9) (9)--(3)
					(5)--(10) (10)--(9)
					(8)--(11) (11)--(10)
					(6)--(12) (12)--(8)
					(0)--(13) (13)--(6)
					(6)--(14) (14)--(15) (15)--(13)
					(6)--(16) (16)--(17) (17)--(14)
					(12)--(18) (18)--(16)
					(17)--(19) (19)--(18)
					(12)--(20) (20)--(21) (21)--(8)
					(11)--(22) (22)--(21)
					(18)--(23) (23)--(20)
					(23)--(24) (24)--(21)
					(19)--(25) (25)--(23)
					(24)--(26) (26)--(22)
					(25)--(27) (27)--(24)
					(27)--(28) (28)--(26)
					(25)--(29) (29)--(28)
					;
			\draw[brace] (27)edge(29) (3)edge(4) (1)edge(13) (7)edge(9) (12)edge(16) (20)edge(18) (23)edge(19);
		\end{scope}
 \end{tikzpicture}
 \qquad
 \begin{tikzpicture}[scale=0.8]
    \node[fvertex] (1) at (0,0) {};
    \node[fvertex] (2) at (1,0) {};
    \node[fvertex] (3) at (1.9,-0.1) {};
    \node[fvertex] (4) at (3,0.1) {};
    \node[fvertex] (5) at (4,0) {};

    \node[fvertex] (6) at ($(1)+(0.2,0.6)$) {};
    \node[fvertex] (7) at ($(2)+(6)-(1)$) {};
    \node[fvertex] (8) at ($(3)+(0.1,0.5)$) {};
    \node[fvertex] (9) at ($(3)+(0.6,0.4)$) {};
    \node[fvertex] (10) at ($(4)+(9)-(3)$) {};

    \node[fvertex] (11) at ($(6)+(8)-(3)$) {};
    \node[fvertex] (12) at ($(7)+(8)-(3)$) {};
    \node[fvertex] (13) at ($(8)+(-0.05,0.5)$) {};
    \node[fvertex] (14) at ($(10)+(13)-(9)$) {};

    \node[fvertex] (15) at ($(11)+(13)-(8)$) {};
    \node[fvertex] (16) at ($(12)+(13)-(8)$) {};
    \node[fvertex] (17) at ($(14)+(0.1,0.6)$) {};

    \node[fvertex] (18) at ($(15)+(17)-(14)$) {};
    \node[fvertex] (19) at ($(16)+(17)-(14)$) {};

    \node[fvertex] (20) at ($(19)+(0.4,0.3)$) {};
    \node[fvertex] (21) at ($(18)+(0.7,0.5)$) {};

    \draw[edge] (1)--(2) (2)--(3) (3)--(4) (4)--(5);
    \draw[edge] (1)--(6) (2)--(7) (3)--(7) (3)--(8) (3)--(9) (4)--(10) (5)--(10);
    \draw[edge] (6)--(7) (8)--(9) (9)--(10);
    \draw[edge] (6)--(11) (7)--(12) (8)--(12) (8)--(13) (9)--(13) (10)--(14);
    \draw[edge] (11)--(12) (13)--(14);
    \draw[edge] (11)--(15) (12)--(16) (13)--(16) (14)--(16);
    \draw[edge] (15)--(16) (16)--(14);
    \draw[edge] (15)--(18) (16)--(19) (10)--(17) (14)--(17);
    \draw[edge] (18)--(19) (19)--(17);
    \draw[edge] (19)--(20) (17)--(20);
    \draw[edge] (18)--(21) (19)--(21) (20)--(21);
 \end{tikzpicture}
 \caption{Different layouts of grids.}
 \label{fig:grids}
\end{figure}

Recently more general grids consisting of parallelograms attained interest (see \Cref{fig:grids} middle).
Such frameworks are for instance interesting in physics due to its relation with quasicrystals~\cite{Zhou2019}.
Rhombic tilings and their flexibility were found to be interesting for arts \cite{Wester}
and have later been formalized \cite{DuarteFrancis,GLbracing,power2021parallelogram}.
Again the connectivity of an auxiliary graph plays an important role.
Now the vertices represent so called ribbons which consist of parallel edges.
The connectivity was also related to NAC-colorings
which are colorings of the edges that determine whether a given graph admits a flexible placement in the plane \cite{GLbracing,flexibleLabelings}.
Frameworks, topologically equivalent to a disc, in which not all triangles come from bracing (see \Cref{fig:grids} right)
were mentioned on a poster (without proofs) by two undergraduate students Aiken and Gregov supervised by Whiteley~\cite{poster}.
They also state that their approach applies also for other simply connected topologies,
like the pages of a book,
which motivated us to formalize and generalize the notions and prove that the flexibility of frameworks with any simply connected topology
can be characterized combinatorially.

\minisec{Our contribution}
In this paper we define a class of frameworks that contains all frameworks mentioned above
as well as others (see \Cref{fig:newFramework}), called \emph{walk-independent frameworks}.
We consider both infinitesimal and continuous flexibility.

\begin{figure}[ht]
    \centering
    \begin{tikzpicture}
            \coordinate (h) at (0,1.3);
            \node[fvertex] (1) at (0,0) {};
            \node[fvertex] (2) at (1.3,0) {};
            \node[fvertex] (3) at (3.1,0.5) {};
            \node[fvertex] (4) at (0.58,1) {};
            \node[fvertex] (5) at ($(2)+(4)-(1)$) {};
            \node[fvertex] (6) at (2.66,1.12) {};
            \node[fvertex] (7) at ($(1)+(h)$) {};
            \node[fvertex] (8) at ($(2)+(h)$) {};
            \node[fvertex] (9) at ($(3)+(h)$) {};
            \node[fvertex] (10) at ($(4)+(h)$) {};
            \node[fvertex] (11) at ($(5)+(h)$) {};
            \node[fvertex] (12) at ($(6)+(h)$) {};
            \node[fvertex] (13) at (1.2,1.74) {};
            \node[fvertex] (14) at ($(10)+(13)-(7)$) {};
            \draw[edge] (1)edge(2) (4)edge(5) (7)edge(8) (10)edge(11) (7)edge(13) (8)edge(13) (10)edge(14) (11)edge(14);
            \draw[edge] (1)edge(4) (2)edge(3) (2)edge(5) (3)edge(5) (3)edge(6) (5)edge(6) (7)edge(10) (8)edge(9) (9)edge(11) (9)edge(12) (11)edge(12) (13)edge(14);
            \draw[edge] (1)edge(7) (2)edge(8) (3)edge(9) (4)edge(10) (6)edge(12);
    \end{tikzpicture}
    \caption{A framework whose flexibility can be described by the presented results.}
    \label{fig:newFramework}
\end{figure}
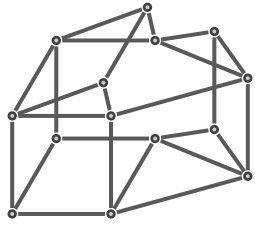

Instead of constructing an auxiliary graph whose connectivity would characterize flexibility of a given framework,
we partition the edges into equivalent classes where two
edges are in the same class whenever they are opposite in a parallelogram or
when they are part of the same triangle.
The classes have the property that any two edges in the same class keep their mutual angle during any flex,
hence, we call them \emph{\APclasses{}}.
Comparing to the previous cases, each row/column/ribbon is a subset of an \APclass{}
and two of them are connected by a path in the corresponding auxiliary graphs
if and only if they belong to the same \APclass{}. 
Namely, the disconnectedness of an auxiliary graph is equivalent to the number of \APclasses{} being at least two.

Our main result is that a walk-independent framework is flexible if and
only if it has at least two \APclasses{}.
We show how to construct a flex using \APclasses{}, give an algorithm to find the classes and analyze the algorithm's complexity.
The classes can be used also for determining which braces could be added to enforce rigidity:
two classes are merged when a diagonal of a 4-cycle intersecting both classes is added.
Similarly to the literature a relation to NAC-colorings can be established but
this paper can be read without knowing them.
The flexibility is also related to non-trivial subgraphs of the Cartesian product of graphs.
We define a class of graphs which can be seen simply connected
when considering 3- and 4-cycles as faces.
We show that every injective placement of such a graph with all 4-cycles being parallelograms is walk-independent,
hence our result applies.

The result is further extended to rotationally symmetric flexes which are related to respective symmetric versions of classes and colorings.
Finally we apply the theory to periodic tilings of the plane.

\minisec{Structure of the paper}
In the following paragraphs we define the important notions of rigidity and flexibility of graphs and frameworks.
The main results of the paper can be found in \Cref{sec:results}.
Generalizations to rotational symmetry of graphs and flexes are investigated in \Cref{sec:sym}
and the results are illustrated on periodic tessellations in \Cref{sec:tess}.

\minisec{Preliminaries}
In this section we collect necessary general definitions on rigidity and flexibility of graphs and frameworks. Further more specific definitions follow throughout the paper as needed.

Since all realizations of a disconnected graph are flexible,
we consider only connected graphs from now on.
\begin{definition*}
	Let $G=(V_G,E_G)$ be a simple undirected connected graph, possibly countably infinite.
	A map $\rho : V_G \rightarrow \RR^2$ such that $\rho(u) \neq \rho(v)$ for all edges $uv \in E_G$
	is a \emph{placement} or \emph{realization}.
	The pair $(G,\rho)$ is called a \emph{framework}.
	If $\rho$ is an injective realization such that each induced 4-cycle in $G$ forms a
	non-degenerate\footnote{Here by non-degenerate we mean that not all its vertices are collinear.}
	parallelogram in $\rho$, we call it a \emph{parallelogram placement}.
\end{definition*}

\begin{definition*}
	Two frameworks $(G,\rho)$ and $(G,\rho')$ are \emph{equivalent} if
	\begin{align*}
		\| \rho(u) - \rho(v) \| = \| \rho'(u) - \rho'(v)\|
	\end{align*}
	for all edges $uv \in E_G$.
	Two placements $\rho$ and $\rho'$ are \emph{congruent} if there exists a Euclidean
	isometry~$M$ of $\RR^2$ such that $M \rho'(v) = \rho(v)$ for all $v \in V_G$.
\end{definition*}

\begin{definition*}
	A \emph{flex} of a framework $(G,\rho)$ is a continuous path $t \mapsto \rho_t$,
	$t \in [0,\varepsilon)$ for $\varepsilon >0$,
	in the space of placements of $G$ such that $\rho_0= \rho$
	and each $(G,\rho_t)$ is equivalent to~$(G,\rho)$.
	The flex is called non-trivial if $\rho_t$ is non-congruent to $\rho$ for all $t \in (0,\varepsilon)$.

	We define a framework to be \emph{(continuously) flexible}
	if there is a non-trivial flex in $\RR^2$.
	Otherwise it is called \emph{rigid}.
\end{definition*}

\begin{definition*}
    An \emph{infinitesimal flex} of a framework  $(G,\rho)$ is a map $\varphi:V_G \rightarrow \RR^2$
    such that
    \[
        (\rho(u)-\rho(v))^\transpose\cdot(\varphi(u)-\varphi(v))=0
    \]
    for all edges $uv \in E_G$.
    It is called \emph{trivial} if it is induced by a rigid motion, namely, if
    there is a $2\times 2$ skew-symmetric matrix $A$ and $b\in\RR^2$ such that
    \[
        \varphi(v) = A \rho(v) + b
    \]
    for all $v\in V_G$.
    We say that $(G,\rho)$ is \emph{infinitesimally flexible}
    if there is a non-trivial infinitesimal flex, otherwise we call it \emph{infinitesimally rigid}.
\end{definition*}

\section{Results}\label{sec:results}
This section contains the main results of the paper on the relation between classes of edges, NAC-colorings and flexes.

In the first part of this section we define rigorously \emph{\APclasses{}}
and show that under certain flexibility conditions a graph needs to have more than one of them.
\emph{Walk-independence} is an important property of the frameworks we are considering.
We define this notion and prove that for walk-independent frameworks the existence of at least
two \APclasses{} yields a flex. We state the main theorem after recalling the Cartesian product of graphs
to which we relate \APclasses{} as well.
In \Cref{sec:comp} we analyze how to find the classes algorithmically.
In \Cref{sec:nac}, we show how the \APclasses{} can be connected to the NAC-colorings from \cite{flexibleLabelings}.
\Cref{sec:pframe} shows that the frameworks considered in \cite{GLbracing},
consisting of (braced) parallelograms only, are walk-independent.
In \Cref{sec:tpframe} we define (using certain simply connected simplicial complexes) a class of graphs for which all placements
with 4-cycles being parallelograms are walk-independent.

In \cite{GLbracing} so called ribbons defined sets of edges which keep their mutual angle during a possible motion. We extend this concept here to \APclasses.
\begin{definition}
    \label{def:APclass}
	Let $G$ be a connected graph. Consider the relation $\Trel$ on the set of edges, where
	two edges are in relation $\Trel$ if they are in a 3-cycle subgraph of $G$.
	Two edges are in relation $\Prel$ if they are opposite edges of a 4-cycle subgraph of $G$.
	An equivalence class of the reflexive-transitive closure $\TPrel$
	of the union $\Trel \cup \Prel$ is called an \emph{\APclass}.
    See \Cref{fig:eqclasses} for examples.
\end{definition}
The flexibility of a framework induces the existence of \APclasses{}.
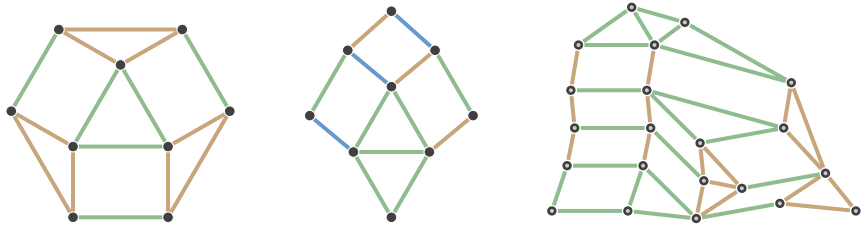
\begin{figure}[ht]
    \centering
    \begin{tikzpicture}[scale=1.25]
        \node[gvertex] (1) at (1,0) {};
        \node[gvertex] (2) at (60:1) {};
        \node[gvertex] (3) at (0,0) {};
        \node[gvertex] (4) at ($(1)+(30:0.75)$) {};
        \node[gvertex] (5) at ($(2)+(4)-(1)$) {};
        \node[gvertex] (6) at ($(2)+(150:0.75)$) {};
        \node[gvertex] (7) at ($(3)+(6)-(2)$) {};
        \node[gvertex] (8) at ($(3)+(270:0.75)$) {};
        \node[gvertex] (9) at ($(1)+(8)-(3)$) {};
        \draw[edge,col1] (1)edge(2) (2)edge(3) (3)edge(1) (4)edge(5) (6)edge(7) (8)edge(9);
        \draw[edge,col5] (1)edge(4) (1)edge(9) (2)edge(5) (2)edge(6) (3)edge(7) (3)edge(8) (5)edge(6) (7)edge(8) (9)edge(4);
    \end{tikzpicture}
    \qquad
    \begin{tikzpicture}[scale=1]
        \node[gvertex] (1) at (1,0) {};
        \node[gvertex] (2) at (60:1) {};
        \node[gvertex] (3) at (0,0) {};
        \node[gvertex] (4) at ($(1)+(40:0.75)$) {};
        \node[gvertex] (5) at ($(2)+(4)-(1)$) {};
        \node[gvertex] (6) at ($(2)+(0,1)$) {};
        \node[gvertex] (7) at ($(2)+(140:0.75)$) {};
        \node[gvertex] (8) at ($(3)+(7)-(2)$) {};
        \node[gvertex] (9) at (-60:1) {};
        \draw[edge,col1] (1)edge(2) (1)edge(9) (2)edge(3) (3)edge(1) (4)edge(5) (3)edge(9) (7)edge(8);
        \draw[edge,col5] (1)edge(4) (2)edge(5) (6)edge(7);
        \draw[edge,col4] (2)edge(7) (3)edge(8) (5)edge(6);
    \end{tikzpicture}
    \qquad
    \begin{tikzpicture}[scale=1]
        \node[fvertex] (1) at (0,0) {};
        \node[fvertex] (2) at (1,0) {};
        \node[fvertex] (3) at (1.9,-0.1) {};
        \node[fvertex] (4) at (3,0.1) {};
        \node[fvertex] (5) at (4,0) {};

        \node[fvertex] (6) at ($(1)+(0.2,0.6)$) {};
        \node[fvertex] (7) at ($(2)+(6)-(1)$) {};
        \node[fvertex] (8) at ($(3)+(0.1,0.5)$) {};
        \node[fvertex] (9) at ($(3)+(0.6,0.4)$) {};
        \node[fvertex] (10) at ($(4)+(9)-(3)$) {};

        \node[fvertex] (11) at ($(6)+(8)-(3)$) {};
        \node[fvertex] (12) at ($(7)+(8)-(3)$) {};
        \node[fvertex] (13) at ($(8)+(-0.05,0.5)$) {};
        \node[fvertex] (14) at ($(10)+(13)-(9)$) {};

        \node[fvertex] (15) at ($(11)+(13)-(8)$) {};
        \node[fvertex] (16) at ($(12)+(13)-(8)$) {};
        \node[fvertex] (17) at ($(14)+(0.1,0.6)$) {};

        \node[fvertex] (18) at ($(15)+(17)-(14)$) {};
        \node[fvertex] (19) at ($(16)+(17)-(14)$) {};

        \node[fvertex] (20) at ($(19)+(0.4,0.3)$) {};
        \node[fvertex] (21) at ($(18)+(0.7,0.5)$) {};

        \draw[edge,col1] (1)--(2) (6)--(7) (11)--(12) (15)--(16) (18)--(19) (18)--(21) (19)--(20) (19)--(21) (19)--(17) (17)--(20) (20)--(21)
                         (14)--(16) (13)--(16) (13)--(14) (9)--(10) (3)--(4);
        \draw[edge,col1] (2)--(3) (1)--(6) (2)--(7) (3)--(7) (8)--(12);
        \draw[edge,col5] (4)--(5);
        \draw[edge,col5] (3)--(8) (3)--(9) (4)--(10) (5)--(10);
        \draw[edge,col5] (8)--(9) ;
        \draw[edge,col5] (6)--(11) (7)--(12)  (8)--(13) (9)--(13) (10)--(14);
        \draw[edge,col5] (11)--(15) (12)--(16)  ;
        \draw[edge,col5] (15)--(18) (16)--(19) (10)--(17) (14)--(17);
    \end{tikzpicture}
    \caption{Three graphs and their \APclasses{} indicated by colors.}
    \label{fig:eqclasses}
\end{figure}

\begin{proposition}
	\label{prop:flexImpliesTwoAPclasses}
	If $(G,\rho)$ is a flexible framework such that all induced 4-cycles are
	non-degenerate parallelograms, then $G$ has at least two \APclasses{}.
\end{proposition}
\begin{proof}
	Let $\rho_t$ be a non-trivial flex of $(G,\rho)$ with $t\in [0, \varepsilon)$.
	We define the following relation on $E_G$:
	edges $uv$ and $wz$ of $G$ are in relation $\constAngle$ if there is $\varepsilon' \in (0, \varepsilon)$
	such that the angle between
	the lines $\rho_t(u)\rho_t(v)$ and $\rho_t(w)\rho_t(z)$ is the same for all $t\in [0, \varepsilon')$.
	Obviously, the relation is an equivalence.

	Since the parallelograms given by induced 4-cycles in the realization $\rho=\rho_0$ are non-degenerate,
	the opposite edges are parallel in $\rho_t$ for all $t$ in some interval $[0,\varepsilon')$.
	Hence, they belong to the same equivalence class of $\constAngle$.
	Clearly, all edges in a 3-cycle are in one equivalence class as well.
	Therefore, an \APclass{} is a subset of an equivalence class of $\constAngle$.

	Since the flex is non-trivial, there are at least two edges
	whose angle is not constant along the flex,
	hence, $\constAngle$ has at least two classes.
	Thus, there are at least two \APclasses{} in $G$.
\end{proof}

A similar statement holds for infinitesimal flexibility.
\begin{proposition}
	\label{prop:infFlexImpliesTwoAPclasses}
	If $(G,\rho)$ is an infinitesimally flexible framework such that all induced 4-cycles are
	non-degenerate parallelograms and all 3-cycles are non-degenerate triangles, then $G$ has at least two \APclasses{}.
\end{proposition}
\begin{proof}
    Suppose we have an infinitesimal flex $\varphi$ which is not induced by a rigid motion.
    We proceed in a similar manner as in \Cref{prop:flexImpliesTwoAPclasses},
    just we want to consider edges to be equivalent if the infinitesimal change of their angle is zero. Let
    \[
        \realizVec{u}{v} = \rho(u)-\rho(v) \text{ and } \flexVec{u}{v} = \varphi(u)-\varphi(v)
    \]
    for $u,v \in V_G$.
    To motivate the following, imagine for a moment that the infinitesimal flex
    $\varphi$ comes from a continuous flex $\rho_t$.
    The angle between $uv$ and $u'v'$ in the continuous flex
    is a function of $(\rho_t(u)-\rho_t(v))^\transpose(\rho_t(u')-\rho_t(v'))$.
    Setting the derivative at $t=0$ to zero gives
    \begin{equation}
        \label{eq:difAngle}
        \realizVec{u}{v}^\transpose\cdot\flexVec{u'}{v'} + \flexVec{u}{v}^\transpose\cdot\realizVec{u'}{v'}=0\,.
    \end{equation}
    The situation when $\realizVec{u}{v}$ and $\realizVec{u'}{v'}$ are linearly dependent
    has to be treated carefully, hence, let us assume for a moment that they are linearly independent (LI). Since $\varphi$ is an infinitesimal flex of $\rho$, we have $\realizVec{u}{v}^\transpose\cdot\flexVec{u}{v}=0$ and $\realizVec{u'}{v'}^\transpose\cdot\flexVec{u'}{v'}=0$
    or equivalently
    \[
        \flexVec{u}{v} = \alpha \piHalfRotation \realizVec{u}{v}
        \quad\text{ and }\quad
        \flexVec{u'}{v'} = \beta \piHalfRotation \realizVec{u'}{v'},
        \text{ where } \piHalfRotation=\piHalfRotationMatrix,
    \]
    for some $\alpha,\beta \in \RR$.
    It can be checked that \Cref{eq:difAngle} is satisfied if and only if $\alpha=\beta$.
    Therefore, we define the relation $\infConstAngle$ on the edge set $E_G$ as follows:
    \begin{align*}
        uv \infConstAngle u'v' \iff
            &\left(\realizVec{u}{v}, \realizVec{u'}{v'} \text{ are LI} \land
                \exists \alpha\in\RR:
                    \flexVec{u}{v} = \alpha \piHalfRotation \realizVec{u}{v}
                    \land
                    \flexVec{u'}{v'} = \alpha \piHalfRotation \realizVec{u'}{v'}
            \right)\\
            &\lor \left(
                \exists \beta\in\RR\setminus\{0\}:
                    \realizVec{u}{v} = \beta \realizVec{u'}{v'}
                    \land
                    \flexVec{u}{v} = \beta \flexVec{u'}{v'}
            \right).
    \end{align*}
    Notice that the second part of the right hand side deals with the situation when
    $\realizVec{u}{v}$ and $\realizVec{u'}{v'}$ are linearly dependent.
    Since adjacent vertices are mapped to distinct points, it makes sense to assume $\beta\neq 0$.
    The relation is well-defined in the sense that swapping $u$ and $v$
    does not change the logical value of the right hand side.
    The relation is clearly reflexive and symmetric.
    We show that it is transitive as well:
    let $uv, u'v', u''v''\in E_G$ be such that $uv \infConstAngle u'v'$ (this gives $\alpha$ xor $\beta$)
    and $u'v' \infConstAngle u''v''$ (this gives $\alpha'$ xor $\beta'$).
    We need to find corresponding $\alpha''$ xor $\beta''$
    to prove that $uv \infConstAngle u''v''$.
    To do so we distinguish four cases:
    \begin{labeling}{$\alpha,\alpha'$:}
        \item[$\alpha,\alpha'$:] We have $\alpha=\alpha'$.
            If $\realizVec{u}{v}$ and $\realizVec{u''}{v''}$
            are LI, we are done by setting $\alpha''=\alpha$. Otherwise,
            $\realizVec{u}{v} = \beta'' \realizVec{u''}{v''}$ and we have
            $\flexVec{u}{v} = \alpha \piHalfRotation\realizVec{u}{v}
                = \alpha' \beta'' \piHalfRotation\realizVec{u''}{v''}=\beta'' \flexVec{u''}{v''}$.
        \item[$\beta,\beta'$:] We get $\beta''=\beta\beta'$.
        \item[$\alpha,\beta'$:] Since $\realizVec{u''}{v''}=\frac{1}{\beta'}\realizVec{u'}{v'}$,
        we have that $\realizVec{u}{v}$ and $\realizVec{u''}{v''}$ are LI.
        Now $\flexVec{u''}{v''}=\frac{1}{\beta'}\flexVec{u'}{v'}
            =\alpha \piHalfRotation \frac{1}{\beta'}\realizVec{u'}{v'}=\alpha \piHalfRotation \realizVec{u''}{v''}$.
        Namely, we set $\alpha''=\alpha$.
        \item[$\beta,\alpha'$:] Follows from the previous and symmetry.
    \end{labeling}
    Next we show that every \APclass{} is contained in an equivalence class of $\infConstAngle$.
    Let $uv,uw$ be two incident edges in a 3-cycle, which is non-degenerate by assumption.
    Hence, $\realizVec{u}{v}, \realizVec{u}{w}$ are linearly independent.
    Let $\alpha$ be such that $\flexVec{u}{v}=\alpha \piHalfRotation\realizVec{u}{v}$.
    Then
    \begin{align*}
        \flexVec{u}{w}&=\flexVec{u}{v} + \flexVec{v}{w}
            =\alpha \piHalfRotation\realizVec{u}{v}  + \flexVec{v}{w}
            =\alpha \piHalfRotation\realizVec{u}{w} + \alpha \piHalfRotation\realizVec{w}{v}+ \flexVec{v}{w} = \alpha \piHalfRotation\realizVec{u}{w} + \gamma \flexVec{v}{w}
    \end{align*}
    for some $\gamma\in \RR$.
    If $\flexVec{v}{w} =(0,0)$, we have $uv\infConstAngle uw$.
    Otherwise multiplying the equation from the left by $\realizVec{u}{w}^\transpose$
    gives
    \begin{equation*}
        \underbrace{\realizVec{u}{w}^\transpose\cdot\flexVec{u}{w}}_{=0}= {\alpha \underbrace{\realizVec{u}{w}^\transpose\piHalfRotation\realizVec{u}{w}}_{=0}} + {\gamma \realizVec{u}{w}^\transpose\cdot\flexVec{v}{w}}\,.
    \end{equation*}
    Then $\gamma=0$ since $\realizVec{u}{w}^\transpose\cdot \flexVec{v}{w}\neq 0$ as $\realizVec{u}{v}, \realizVec{u}{w}$ are linearly independent and $\realizVec{v}{w}^\transpose\cdot \flexVec{v}{w}=0$.
    Hence, the edges are in the same equivalence class.

    Now, let $(u,v,v',u',u)$ be a 4-cycle.
    We have $\realizVec{u}{v}=\realizVec{u'}{v'}$ and $\realizVec{u}{u'}=\realizVec{v}{v'}$.
    Thus, $\flexVec{u}{v},\flexVec{u'}{v'}$, respectively $\flexVec{u}{u'},\flexVec{v}{v'}$, are linearly dependent and without loss of generality we can assume that $\flexVec{u}{v}=\gamma\flexVec{u'}{v'}$ and $\flexVec{u}{u'}=\varepsilon\flexVec{v}{v'}$.
    To prove that $uv\infConstAngle u'v'$ and $uu'\infConstAngle vv'$, we need to show that $\gamma$
    and $\varepsilon$ can be set to 1.
    We have
    \begin{align*}
        \flexVec{u}{v}&=\gamma(\flexVec{u'}{u}+\flexVec{u}{v'})=-\gamma\varepsilon\flexVec{v}{v'}+\gamma\flexVec{u}{v'}
            =-\gamma\varepsilon\flexVec{v}{v'}+\gamma(\flexVec{u}{v}+\flexVec{v}{v'})\\
        \iff (0,0)&=(\gamma-1)\flexVec{u}{v}+\gamma(1-\varepsilon)\flexVec{v}{v'}\,.
    \end{align*}
    If $\flexVec{u}{v},\flexVec{v}{v'}$ are non-zero, then they are linearly independent
    since they are orthogonal to $\realizVec{u}{v}$ and $\realizVec{v}{v'}$ and the 4-cycle is non-degenerate.
    Hence, $\gamma=\varepsilon=1$ which
    yields $uv\infConstAngle u'v'$ and $uu'\infConstAngle vv'$.
    If $\flexVec{u}{v},\flexVec{v}{v'}$ are both zero,
    then $\varphi(u)=\varphi(v)=\varphi(v')$. Hence, $\flexVec{u}{u'}=\flexVec{v}{u'}$,
    but this is orthogonal to both $\realizVec{u}{u'}$ and $\realizVec{v}{u'}$
    which are linearly independent by the assumption that the 4-cycles are non-degenerate.
    Therefore, $\flexVec{u}{u'}=0$ and the statement follows.
    If $\flexVec{u}{v}\neq (0,0)= \flexVec{v}{v'}$, then $\gamma=1$ and $\varepsilon$
    can be 1 as $\flexVec{u}{u'}=\varepsilon\flexVec{v}{v'}=(0,0)$.
    Finally, if $\flexVec{u}{v}= (0,0) \neq \flexVec{v}{v'}$, then
    $\varphi(u)=\varphi(v)$ which yields
    \[
        \flexVec{u'}{v'}=\flexVec{u'}{u}+\flexVec{v}{v'} = (-\varepsilon+1)\flexVec{v}{v'}\,.
    \]
    Multiplying the equation by $\realizVec{u'}{v'}^\transpose$ from the left implies $\varepsilon=1$,
    since $\flexVec{v}{v'}$ cannot be orthogonal to $\realizVec{u'}{v'}$ as this would imply that
    $\realizVec{u'}{v'}$ and $\realizVec{v}{v'}$ are linearly dependent which is not possible by the assumption.
    Then $\varphi(u')=\varphi(v')$ and $\gamma$ can be set to 1 as well.

    To conclude the statement, we just need to show that $\infConstAngle$
    has at least two equivalence classes.
    Since the infinitesimal flex $\varphi$ is not induced by a rigid motion,
    we can assume that $\varphi(\bar{u})=\varphi(\bar{v})=(0,0)$ and $\varphi(\bar{w})\neq (0,0)$
    for some edges $\bar{u}\bar{v},\bar{v}\bar{w}$.
    Namely, by applying a rigid motion to $\rho$, we can assume that
    $\rho(u)=(0,0)$ and $\rho(v)=(\lambda, 0)$ for an edge~$uv$, where ${\lambda=||\rho(u)-\rho(v)||\neq 0}$.
    Let $\varphi$ be a non-trivial infinitesimal flex of $\rho$.
    Since $\flexVec{u}{v}$ is orthogonal to $\realizVec{u}{v}$,
    we have $\flexVec{u}{v}=(0,y)$.
    By subtracting the trivial infinitesimal flex $-\frac{y}{\lambda}\piHalfRotation\rho(\cdot)+\varphi(u)$
    from~$\varphi$, we get a non-trivial infinitesimal flex $\varphi'$ such that $\varphi'(u)=\varphi'(v)=(0,0)$.
    Since the flex is non-trivial, there is a vertex $w$ such that $\varphi'(w)\neq (0,0)$.
    By connectivity of $G$, there is a path from $v$ to $w$
    and therefore there are two consecutive edges $\bar{u}\bar{v},\bar{v}\bar{w}$ with the required property.

    Finally, we check that $\bar{u}\bar{v}$ is not
    in relation with $\infConstAngle$ with $\bar{v}\bar{w}$:
    suppose first that $\realizVec{\bar{u}}{\bar{v}}$ and~$\realizVec{\bar{v}}{\bar{w}}$ are linearly independent.
    Then we have $(0,0)=\flexVec{\bar{u}}{\bar{v}}=\alpha\piHalfRotation\realizVec{\bar{u}}{\bar{v}}$.
    Since $\realizVec{\bar{u}}{\bar{v}}\neq (0,0)$ as $\bar{u}\bar{v}\in E_G$, it is necessary that $\alpha=0$.
    But the equality $\flexVec{\bar{v}}{\bar{w}}=\alpha\piHalfRotation\realizVec{\bar{v}}{\bar{w}}=(0,0)$
    does not hold as $\flexVec{\bar{v}}{\bar{w}}=-\varphi(\bar{w})\neq(0,0)$.
    On the other hand, if $\realizVec{\bar{u}}{\bar{v}}$ and $\realizVec{\bar{v}}{\bar{w}}$ are linearly dependent,
    i.e., $\realizVec{\bar{u}}{\bar{v}}=\beta\realizVec{\bar{v}}{\bar{w}}$ for some $\beta$
    which is non-zero since the vectors are non-zero by the fact that $\bar{u}\bar{v},\bar{v}\bar{w}\in E_G$,
    we have that $\beta \flexVec{\bar{v}}{\bar{w}}=-\beta \varphi(\bar{w})\neq (0,0)=\flexVec{\bar{u}}{\bar{v}}$.
\end{proof}

The following definition provides an important property for frameworks.
For a (closed) walk $W=(u_1,\ldots,u_k)$ and an \APclass{} $r$,
the notation $\sum_{\oriented{u}{v} \in r\cap W}$ means
we sum up over all edges $(u_i,u_{i+1})$ with $1\leq i <k$ such that $u_i u_{i+1}\in r$
in the following.
\begin{definition}
    We say that a framework $(G,\rho)$ is \emph{walk-independent},
    if $\rho$ is a parallelogram placement of $G$ and for every \APclass{} $r$
    \begin{equation*}
		\sum_{\oriented{u}{v} \in r\cap W} (\rho(v)-\rho(u)) = \sum_{\oriented{u}{v} \in r\cap W'} (\rho(v)-\rho(u))
	\end{equation*}
	for every $w_1,w_2\in V_G$  and walks $W,W'$ in $G$ from $w_1$ to $w_2$.
    See \Cref{fig:crossingprop-def} for an example.
    Equivalently, the sum is zero for every closed walk $C$, that is
    \begin{equation}
		\label{eq:zeroSumAPclass}
		\sum_{\oriented{u}{v} \in r\cap C} (\rho(v)-\rho(u)) = (0,0)\,.
	\end{equation}
\end{definition}
Note that this definition does depend on the chosen placement (see \Cref{fig:crossingprop-ex}).
\Cref{prop:zeroSumCheck} provides a set of cycles that are sufficient to be checked
instead of all closed walks for the finite case.

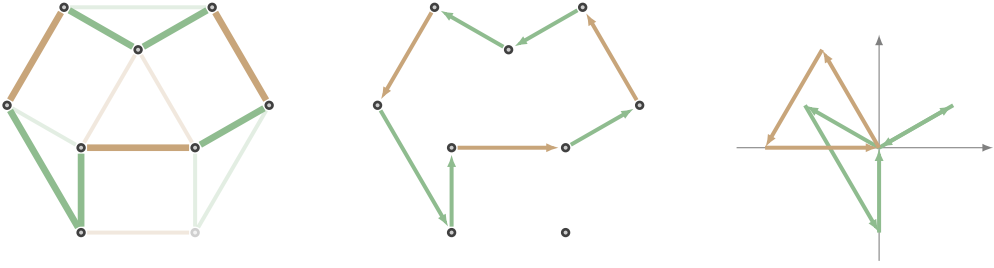
\begin{figure}[ht]
    \centering
    \begin{tikzpicture}[scale=1.5]
        \node[fvertex] (1) at (1,0) {};
        \node[fvertex] (2) at (60:1) {};
        \node[fvertex] (3) at (0,0) {};
        \node[fvertex] (4) at ($(1)+(30:0.75)$) {};
        \node[fvertex] (5) at ($(2)+(4)-(1)$) {};
        \node[fvertex] (6) at ($(2)+(150:0.75)$) {};
        \node[fvertex] (7) at ($(3)+(6)-(2)$) {};
        \node[fvertex] (8) at ($(3)+(270:0.75)$) {};
        \node[lfvertex] (9) at ($(1)+(8)-(3)$) {};

        \draw[ledge,col5] (1)edge(2) (2)edge(3) (3)edge(1);
        \draw[ledge,col5] (4)edge(5) (6)edge(7) (8)edge(9);
        \draw[ledge,col1] (5)edge(6) (7)edge(8) (9)edge(4);
        \draw[ledge,col1] (1)edge(4) (1)edge(9) (2)edge(5) (2)edge(6) (3)edge(7) (3)edge(8);
        \draw[hedge2,col1] (1)edge(4) (5)edge(2) (2)edge(6) (7)edge(8) (8)edge(3);
        \draw[hedge2,col5] (4)edge(5) (6)edge(7) (3)edge(1);
        \begin{scope}[xshift=3.25cm]
            \node[fvertex] (1) at (1,0) {};
            \node[fvertex] (2) at (60:1) {};
            \node[fvertex] (3) at (0,0) {};
            \node[fvertex] (4) at ($(1)+(30:0.75)$) {};
            \node[fvertex] (5) at ($(2)+(4)-(1)$) {};
            \node[fvertex] (6) at ($(2)+(150:0.75)$) {};
            \node[fvertex] (7) at ($(3)+(6)-(2)$) {};
            \node[fvertex] (8) at ($(3)+(270:0.75)$) {};
            \node[fvertex] (9) at ($(1)+(8)-(3)$) {};
            \draw[dedge,col1] (1)edge(4) (5)edge(2) (2)edge(6) (7)edge(8) (8)edge(3);
            \draw[dedge,col5] (4)edge(5) (6)edge(7) (3)edge(1);
        \end{scope}
        \begin{scope}[xshift=7cm]
            \draw[axes] (0,-1)--(0,1);
            \draw[axes] (-1.25,0)--(1,0);
            \draw[dedge,col1] (0,0)--++($(4)-(1)$);
            \draw[dedge,col1] ($(4)-(1)$)--++($(2)-(5)$);
            \draw[dedge,col1] ($(4)-(1)+(2)-(5)$)--++($(6)-(2)$);
            \draw[dedge,col1] ($(4)-(1)+(2)-(5)+(6)-(2)$)--++($(8)-(7)$);
            \draw[dedge,col1] ($(4)-(1)+(2)-(5)+(6)-(2)+(8)-(7)$)--++($ (3)-(8)$);
            \draw[dedge,col5] (0,0)--++($(5)-(4)$);
            \draw[dedge,col5] ($(5)-(4)$)--++($(7)-(6)$);
            \draw[dedge,col5] ($(5)-(4)+(7)-(6)$)--++($(1)-(3)$);
        \end{scope}
    \end{tikzpicture}
    \caption{A framework with a closed walk indicated in bold (left); the directed walk with colors of \APclasses{} (middle); the vectors of the edges form the walk for each \APclass{} (right). One can see that they do sum up to zero in this example.}
    \label{fig:crossingprop-def}
\end{figure}

\begin{figure}[ht]
    \centering
    \begin{tikzpicture}
        \begin{scope}
            \coordinate (h) at (0,1.3);
            \node[fvertex] (1) at (0,0) {};
            \node[fvertex] (2) at (1.3,0) {};
            \node[fvertex] (3) at (3.1,0.5) {};
            \node[fvertex] (4) at (0.58,1) {};
            \node[fvertex] (5) at ($(2)+(4)-(1)$) {};
            \node[fvertex] (6) at (2.66,1.12) {};
            \node[fvertex] (7) at ($(1)+(h)$) {};
            \node[fvertex] (8) at ($(2)+(h)$) {};
            \node[fvertex] (9) at ($(3)+(h)$) {};
            \node[fvertex] (10) at ($(4)+(h)$) {};
            \node[fvertex] (11) at ($(5)+(h)$) {};
            \node[fvertex] (12) at ($(6)+(h)$) {};
            \node[fvertex] (13) at (1.2,1.74) {};
            \node[fvertex] (14) at ($(10)+(13)-(7)$) {};
            \draw[edge,col5] (1)edge(2) (4)edge(5) (7)edge(8) (10)edge(11) (7)edge(13) (8)edge(13) (10)edge(14) (11)edge(14);
            \draw[edge,col1] (1)edge(4) (2)edge(3) (2)edge(5) (3)edge(5) (3)edge(6) (5)edge(6) (7)edge(10) (8)edge(9) (9)edge(11) (9)edge(12) (11)edge(12) (13)edge(14);
            \draw[edge,col6] (1)edge(7) (2)edge(8) (3)edge(9) (4)edge(10) (6)edge(12);
        \end{scope}
        \begin{scope}[xshift=4.5cm]
            \coordinate (h) at (0,1.3);
            \node[lfvertex] (1) at (0,0) {};
            \node[lfvertex] (2) at (1.3,0) {};
            \node[lfvertex] (3) at (3.1,0.5) {};
            \node[lfvertex] (4) at (0.58,1) {};
            \node[lfvertex] (5) at ($(2)+(4)-(1)$) {};
            \node[lfvertex] (6) at (2.66,1.12) {};
            \node[lfvertex] (7) at ($(1)+(h)$) {};
            \node[fvertex] (8) at ($(2)+(h)$) {};
            \node[fvertex] (9) at ($(3)+(h)$) {};
            \node[lfvertex] (10) at ($(4)+(h)$) {};
            \node[fvertex] (11) at (2.46,3.94) {};
            \node[lfvertex] (12) at ($(6)+(h)$) {};
            \node[fvertex] (13) at (1.2,1.74) {};
            \node[fvertex] (14) at ($(10)+(13)-(7)$) {};
            \draw[hedge2,col5] (11)edge(14) (13)edge(8);
            \draw[hedge2,col1] (8)edge(9) (9)edge(11) (14)edge(13);
            \draw[ledge,col5] (1)edge(2) (4)edge(5) (7)edge(8) (10)edge(11) (7)edge(13) (8)edge(13) (10)edge(14) (11)edge(14);
            \draw[ledge,col1] (1)edge(4) (2)edge(3) (2)edge(5) (3)edge(5) (3)edge(6) (5)edge(6) (7)edge(10) (8)edge(9) (9)edge(11) (9)edge(12) (11)edge(12) (13)edge(14);
            \draw[ledge,col6] (1)edge(7) (2)edge(8) (3)edge(9) (4)edge(10) (6)edge(12);
        \end{scope}
        \begin{scope}[xshift=8cm]
            \coordinate (h) at (0,1.3);
            \coordinate (1) at (0,0);
            \coordinate (2) at (1.3,0);
            \coordinate (3) at (3.1,0.5);
            \coordinate (4) at (0.58,1);
            \coordinate (5) at ($(2)+(4)-(1)$);
            \coordinate (6) at (2.66,1.12);
            \coordinate (7) at ($(1)+(h)$);
            \node[fvertex] (8) at ($(2)+(h)$) {};
            \node[fvertex] (9) at ($(3)+(h)$) {};
            \coordinate (10) at ($(4)+(h)$);
            \node[fvertex] (11) at (2.46,3.94) {};
            \coordinate (12) at ($(6)+(h)$);
            \node[fvertex] (13) at (1.2,1.74) {};
            \node[fvertex] (14) at ($(10)+(13)-(7)$) {};
            \draw[dedge,col1] (8)edge(9) (9)edge(11) (14)edge(13);
            \draw[dedge,col5] (11)edge(14) (13)edge(8);
        \end{scope}
    \end{tikzpicture}
    \caption{The left framework is walk-independent,
        whereas in the middle one, which has the same underlying graph and which is still a parallelogram placement, the walk-independence is violated by the indicated cycle. This can be seen on the right, where the cycle is colored according to the \APclasses{}.}
    \label{fig:crossingprop-ex}
\end{figure}
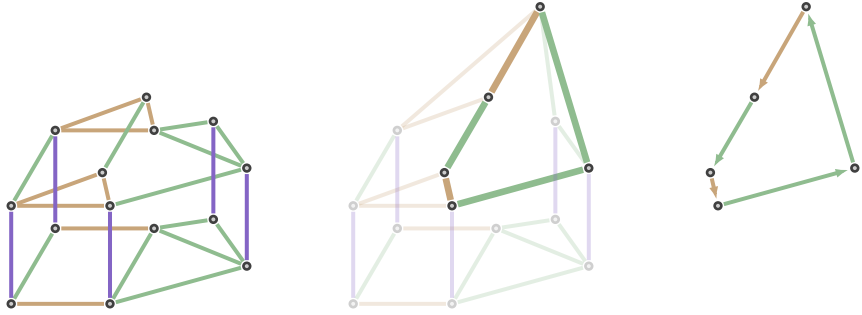

Walk-independence implies another interesting property.
\begin{lemma}
	\label{lem:APclassesAreEdgeCuts}
	Let $(G,\rho)$ be walk-independent.
	Every two distinct vertices in $G$ are separated by an \APclass{}.
	In particular, each \APclass{} is an edge cut.
\end{lemma}
\begin{proof}
	Let $\bar{u},\bar{v}\in V_G, \bar{u}\neq\bar{v}$ and $W$ be a path from $\bar{u}$ to $\bar{v}$.
	We have
	\begin{align*}
		\rho(\bar{v}) - \rho(\bar{u})
			= \sum_{\oriented{u}{v} \in W} (\rho(v)-\rho(u))
			= \sum_r\sum_{\oriented{u}{v} \in r\cap W} (\rho(v)-\rho(u))\,,
	\end{align*}
	where the outer sum is over all \APclasses{} $r$ occurring in $W$.
	Since $\rho(\bar{v}) \neq \rho(\bar{u})$ by injectivity,
	at least one of the inner sums has to be non-zero.
	The corresponding \APclass{} separates $\bar{u}$ and~$\bar{v}$,
	otherwise the cycle obtained by concatenating $W$ with a path from $\bar{v}$ to $\bar{u}$
	avoiding $r$ would contradict \Cref{eq:zeroSumAPclass}.

	In particular, an \APclass{} separates the end vertices of any of its edges,
	namely, it is an edge cut.
\end{proof}

Now we go the other way round and show how \APclasses{} affect flexibility.
\begin{proposition}
	\label{prop:twoAPclassesImplyFlex}
	Let $(G,\rho)$ be walk-independent such that $G$ has $\ell\in \NN \cup \{\infty\}$ \APclasses{}.
	If $\ell\geq 2$, then the framework is flexible.
	In particular, if $\ell \in \NN$ and  $\ell\geq 2$, there are $\ell-1$ independent ways how it can flex.
\end{proposition}
\begin{proof}
	Let $\sst{r_i}{i \in I}$ be the \APclasses{} of $G$, where $I=\{0, \ldots, \ell-1\}$ or $\NN$
	depending whether there are finitely or infinitely many of them.
	We pick a vertex $\bar{u}$. We can suppose that $\rho(\bar{u})=(0,0)$.

	For each $i\in I$,
	we define a map $\rho_i:V_G\rightarrow\RR^2$ as follows:
	let $W$ be any walk in $G$ from $\bar{u}$ to $v$ and
	\begin{equation*}
		\rho_i (v) = \sum_{\oriented{w_1}{w_2}\in W \cap r_i} (\rho(w_2)-\rho(w_1))\,.
	\end{equation*}
	By the walk-independence, it is well-defined,
	namely, the sum is independent of the choice of~$W$.
	Moreover, $\rho_i(v)\neq (0,0)$ only for finitely many $i$
	since the same walk $W$ can be used for all $i$ and as it is finite, it meets only finitely many \APclasses{}.
	We observe that if $uv\in r_j$ and $i\neq j$, then $\rho_i (u) = \rho_i (v)$.

	We define a map $\ttt\mapsto \rho_\ttt$,
	where $\ttt=(t_i)_{i\in I}$ with $t_0=0$ and $t_i\in [0,2\pi)$ for $i\neq 0$, by
	\begin{equation*}
		\rho_\ttt(v) =  \sum_{i\in I} \Rot{t_i} \rho_i(v)\,,
	\end{equation*}
	where $\Rot{t_i}$ is the clockwise rotation matrix by $t_i$ radians.
	There are only finitely many non-zero terms in the sum by the remark above.

	First, we show that setting $\ttt=\zerovec=(0)_{i\in I}$ gives $\rho$.
	Let $v\in V_G$ and $W$ be a path from $\bar{u}$ to $v$. Now
	\begin{align*}
		\rho_\mathbf{0}(v) &=\sum_{i\in I} \Rot{0} \rho_i(v)
				= \sum_{i\in I} \rho_i(v)
				= \sum_{i\in I} \sum_{\oriented{w_1}{w_2}\in W \cap r_i} (\rho(w_2)-\rho(w_1))\\
			& = \sum_{\oriented{w_1}{w_2}\in W} (\rho(w_2)-\rho(w_1))
				= \rho(v) - \rho(\bar{u}) = \rho(v)\,.
	\end{align*}
	We check that the edge lengths induced by each realization in the image are the same.
	If $uv$ is an edge belonging to an \APclass{} $r_j$,
	then for every $\ttt$ we have
	\begin{align*}
		\|\rho_\ttt(u) - \rho_\ttt(v)\|
			&= \left\|\sum_{i\in I} \Rot{t_i}(\rho_i(u) - \rho_i(v))\right\|
			 = \left\|\Rot{t_j}(\rho_j(u) - \rho_j(v))\right\| \\
			&= \|\rho_j(u) - \rho_j(v)\|\,,
	\end{align*}
	which is independent of $\ttt$. Notice that the
	second equality is due to the observation made after defining $\rho_i$.

	Hence, for instance $t \mapsto \rho_{(0,t,0,\ldots,0)}$, or respectively $t \mapsto \rho_{(0,t,0,\ldots)}$ if $\ell=\infty$, gives a flex of $(G,\rho)$
	which is non-trivial since the edges in $r_1$ change the angle
	with the edges in other \APclasses{}.
\end{proof}

In \Cref{fig:ex:flex} we see an example of a graph with three \APclasses{} and how it can flex.
We recall the Cartesian product of graphs in order to state the main theorem.

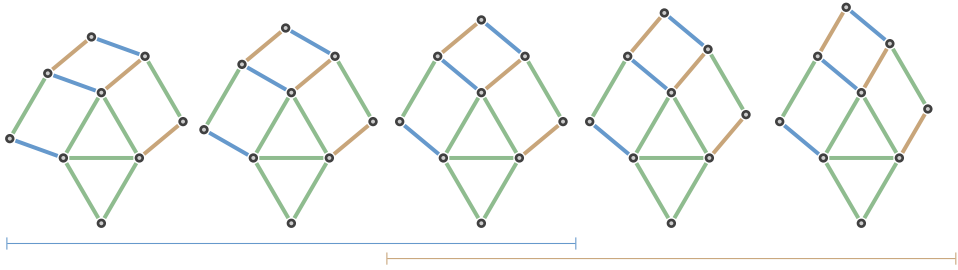
\begin{figure}[ht]
    \centering
    \begin{tikzpicture}[scale=1]
        \foreach \a/\b [count=\i from 0] in {0/20,0/10,0/0,10/0,20/0}
        {
            \begin{scope}[xshift=2.5*\i cm]
                \node[fvertex] (1) at (1,0) {};
                \node[fvertex] (2) at (60:1) {};
                \node[fvertex] (3) at (0,0) {};
                \node[fvertex] (4) at ($(1)+(40+\a:0.75)$) {};
                \node[fvertex] (5) at ($(2)+(4)-(1)$) {};
                \node[fvertex] (7) at ($(2)+(140+\b:0.75)$) {};
                \node[fvertex] (8) at ($(3)+(7)-(2)$) {};
                \node[fvertex] (6) at ($(5)+(7)-(2)$) {};
                \node[fvertex] (9) at (-60:1) {};
                \draw[edge,col1] (1)edge(2) (1)edge(9) (2)edge(3) (3)edge(1) (4)edge(5) (3)edge(9) (7)edge(8);
                \draw[edge,col5] (1)edge(4) (2)edge(5) (6)edge(7);
                \draw[edge,col4] (2)edge(7) (3)edge(8) (5)edge(6);
            \end{scope}
        }
        \draw[col4,|-|] ($(60:1)-(1.25,2)$)--++(7.5,0);
        \draw[col5,|-|] ($(60:1)-(1.25,2.2)+(5,0)$)--++(7.5,0);
    \end{tikzpicture}
    \caption{A framework with two independent ways of flexing. In the first half the \APclass{}    \protect\tikz{\protect\draw[edge,col4](0,0)--(0.5,0)} on the left causes the flex. In the second part the \APclass{} \protect\tikz{\protect\draw[edge,col5](0,0)--(0.5,0)} on the right does.}
    \label{fig:ex:flex}
\end{figure}

\begin{definition}
	The \emph{Cartesian product} of graphs $Q_1, \dots, Q_\ell$ is the graph 
	$Q_1\cartProd \dots\cartProd Q_\ell$ whose vertex set is $V_{Q_1}\times \dots\times V_{Q_\ell}$
	and $(u_1,\ldots,u_\ell)$ is adjacent to $(v_1,\ldots,v_\ell)$ if there is $1\leq j\leq \ell$
	such that $u_jv_j\in E_{Q_j}$ and $u_i=v_i$ for all $i\neq j$.
	
	We say that a graph $G$ is a \emph{non-trivial subgraph} of $H = Q_1\cartProd \dots\cartProd Q_\ell$
	if there is	an injective graph homomorphism $h:G\rightarrow H$
    such that $|\pi_i(h(G))|\geq 2$ for all $1\leq i\leq \ell$,
    where  $\pi_i:V_H \rightarrow V_{Q_i}$
	is the projection given by $\pi_i(v_1,\ldots,v_\ell)=v_i$.
\end{definition}

Now we prove the main result of the paper.
We remark that further equivalent conditions on flexibility are given by \Cref{cor:FlexIffCartNAC} in \Cref{sec:nac}.

\begin{theorem}
    \label{thm:main}
    Let $(G,\rho)$ be walk-independent.
    The following statements are equivalent:
    \begin{enumerate}
        \item\label{it:flex} $(G,\rho)$ is flexible,
        \item\label{it:twoAPclasses} $G$ has at least two \APclasses{}.
    \end{enumerate}
    If in addition all 3-cycles in $G$ form non-degenerate triangles in $\rho$,
    then the above statements are equivalent to $(G,\rho)$ being infinitesimally flexible.
    
    If $G$ has $\ell$ \APclasses{}, where $2\leq \ell\in \NN$, then
    $G$ is a non-trivial subgraph of the Cartesian product of $\ell$ graphs. 
\end{theorem}
\begin{proof}
    \ref{it:flex} is equivalent to \ref{it:twoAPclasses} by \Cref{prop:flexImpliesTwoAPclasses}
    and \Cref{prop:twoAPclassesImplyFlex}.
    If the triangles are non-degenerate, then the existence of a non-trivial infinitesimal flex
    implies \ref{it:twoAPclasses} by \Cref{prop:infFlexImpliesTwoAPclasses}.
    On the other hand, differentiation of the flex constructed in \Cref{prop:twoAPclassesImplyFlex}
    gives an infinitesimal flex, namely, we fix a vertex $\bar{u}$ and an \APclass{} $r$ and for $w\in V_G$ we set
    \[
        \varphi(w)=\piHalfRotationMatrix\cdot \sum_{\oriented{w_1}{w_2}\in W \cap r} (\rho(w_2)-\rho(w_1))\,,
    \]
    where $W$ is a walk from $\bar{u}$ to $w$.
    If $uv\in r$, then
    \[
        \varphi(u)-\varphi(v)
            =\piHalfRotationMatrix \left(\rho(u)-\rho(v)\right)\neq(0,0)\,,
    \]
    which is orthogonal to $\rho(u)-\rho(v)$.
    For $uv\notin r$ (such edges exist by the assumption that there are at least two \APclasses{}),
    we have $\varphi(u)-\varphi(v)=(0,0)$.
    Hence, the infinitesimal flex $\varphi$ is not induced by a rigid motion.
    
    Suppose that $G$ has \APclasses{} $r_1\ldots, r_\ell$ with $\ell\geq 2$.
    We follow the central concept of~\cite{Klavzar2005}.
    Let us consider \emph{quotient graphs} $Q_i$, where the vertices of $Q_i$
    are the connected components of the graph $G\setminus r_i$
    and two connected components are adjacent if and only if they are linked by an edge from $r_i$.
    Let $h_i:V_G\rightarrow V_{Q_i}$ be the map assigning to vertex $v$ the connected component of $G\setminus r_i$
    to which $v$ belongs.
    We define $h:G\rightarrow Q_1\cartProd \dots\cartProd Q_\ell$ by $h(v)=(h_1(v),\ldots,h_\ell(v))$.
    This is a graph homomorphism: if $uv\in E_G\cap r_i$,
    then $u$ and $v$ are in the same connected component of $G\setminus r_j$ for all $j\neq i$.
    Since $r_i$ is an edge cut by \Cref{lem:APclassesAreEdgeCuts},
    $h_i(u)\neq h_i(v)$ and $h_i(u)$ is adjacent to $h_i(v)$ in $Q_i$,
    namely, $h(u)h(v)$ is an edge of $Q_1\cartProd \dots\cartProd Q_\ell$.
    This also shows that $|\pi_i(h(G))|\geq 2$.
    We show that $h$ is injective: let $u,v$ be two distinct vertices of $G$.
    By \Cref{lem:APclassesAreEdgeCuts}, there is an \APclass{} $r_i$ separating $u$ and $v$,
    hence, $h(u)$ and $h(v)$ differ on the $i$-th component.
\end{proof}

\subsection{Computation}\label{sec:comp}
In this section we analyze how to find the \APclasses{} of a given graph computationally.
We give an algorithm and show its complexity.

Assume $|V_G|=n$, $|E_G|=m$ and the maximum  degree is $d$.
A disjoint-set data structure is an efficient way
to construct the set of equivalence classes, namely, \APclasses{}.
It allows to check if two elements belong to the same class and
to merge classes containing elements $e$ and $e'$ (\union{e,e'}).
\begin{algorithm}[ht]
	\caption{\textsc{\APclasses{} computation}}
	\label{alg:APclasses}
	\begin{algorithmic}[1]
		\Require $G$ given as adjacency lists
		\Ensure \APclasses{} of $G$
		\State Initialize a disjoint-set data structure $E$ with the singletons being the edges of $G$
		\State Initialize an empty hash table $T$ \Comment{to store $\{u,v,w\}$ such that $uw,vw\in E_G$}
		\State Initialize an empty hash table $S$ \Comment{to store $\{u,v\}\mapsto (w,\ell)$ such that $uw,vw\in E_G$, $\ell\in \NN$}
		\For{$w\in V_G$}
		    \For{$u,v$, $u\neq v$ in the adjacency list of $w$}
		        \If{$\{u,v,w\}\in T$}\label{line:b-begin}
		            \State \union{uv,vw,uw}
		        \Else\label{line:elseT}
		            \State Add $\{u,v,w\}$ to $T$
    		        \If{$\{u,v\}\notin S$}
    		            \State Add $\{u,v\}\mapsto (w,1)$ to $S$
    		        \Else
    		            \State Let $(w',l)$ be the value of $\{u,v\}$ in $S$
    		            \State \union{uw,vw'} and \union{vw,uw'}
    		            \If{$\ell=2$}\label{line:ifl}
    		                \State \union{uw,vw}
    		            \EndIf
    		            \State Set $\{u,v\}\mapsto (w',l+1)$ in $S$
    		        \EndIf
		        \EndIf
		        \If{$E$ has a single set}
		            \State\Return There is only a single \APclass{} $E_G$.
		        \EndIf\label{line:b-end}
		    \EndFor
		\EndFor
		\State\Return The sets in $E$
	\end{algorithmic}
\end{algorithm}

Notice that if $G$ has a parallelogram placement,
then it cannot have an induced complete bipartite subgraph on $2+s$
vertices with $s\geq 3$, since there is no injective realization such that
all 4-cycles in the complete bipartite subgraph, which are induced, form parallelograms.
Nevertheless, the presented algorithm works for any graph.

We show that \Cref{alg:APclasses} is correct:
considering a 3-cycle $(u_1,u_2,u_3,u_1)$,
we can assume that $u_1$ is processed first, causing $\{u_1,u_2,u_3\}$ to be added to $T$.
Then when $u_2$ or $u_3$ is processed, $\{u_1,u_2,u_3\}$ is in $T$ and
all three edges of the 3-cycle are forced to be in the same \APclass{}.

Let $(w_1,u,w_2,v,w_1)$ be a 4-cycle and $\{w_1,\ldots,w_s\}$
be all vertices that are adjacent to both $u$ and $v$.
In other words, there is a complete bipartite subgraph $H$ on $2+s$ vertices.
If $s=2$, the bipartite subgraph is actually a 4-cycle
and each pair of the opposite edges is in the same \APclass{}.
Otherwise, all edges of $H$ are in the same \APclass{}.
We show that the algorithm guarantees this.
We can assume that $uv\notin E_G$,
otherwise the edges of $H$ are in the same \APclass{}
thanks to being in 3-cycles sharing $uv$.
We can assume that $w_1,\ldots,w_s$ are processed by the outer loop
in the given order.
Since $\{u,w_i,v\}$ is not in $T$ (otherwise $uv\in E_G$),
the ``else'' branch (\Cref{line:elseT}) is entered for every $1\leq i\leq s$.
When $i=1$, $\{u,v\}\mapsto (w_1,1)$ is added to $S$.
When $i=2$, $uw_1$ and $vw_2$ are guaranteed to be
in the same \APclass{} $r$ and $vw_1$ and $uw_2$ in \APclass{} $r'$.
Since $l=1$, the inner ``if'' (\Cref{line:ifl}) is skipped.
When $i=3$, $uw_3$ is guaranteed to be in $r'$ and $vw_3$ in $r$.
Since $l=2$, the \APclasses{} $r$ and $r'$ are merged.
For $i\geq 4$, edges $uw_i,vw_i$ are guaranteed to be in $r=r'$.

We can detect induced complete bipartite subgraphs by running the algorithm
without the stopping condition and checking $uv\in E_G$
for all $\{u,v\}$ in $S$ that map to $(w,l)$ with $l\geq 3$.

Clearly, the algorithm does not force any other pair of edges
to be in the same \APclass{} then those that are supposed to be.
It can happen that all edges are already in the same \APclass{}
before processing all vertices,
hence it makes sense to keep track of the number of the classes in $E$ and stop if there is only a single one.

The number of times the block in the nested loops (\Crefrange{line:b-begin}{line:b-end}) is executed is at most
\begin{equation*}
    \sum_{w\in V_G} \frac{\deg w(\deg w-1)}{2}
    \leq \frac{d-1}{2}\sum_{w\in V_G}\deg w = (d-1)m\,.
\end{equation*}
Within the block, finding, inserting and retrieving from hash table
is used, all of these having constant amortized time.
Next, we use \union{\cdot}, which has amortized time given by the inverse Ackermann function $\alpha$ \cite{Tarjan1975}.
Listing the sets in $E$ costs $O(m\, \alpha(m))$ operations.
In total, the complexity is $O(dm\,\alpha(dm))$.

It is well known that a framework $(G,\rho)$ is infinitesimally rigid if and only if
its \emph{rigidity matrix} has the rank equal to $2|V_G|-3$.
We want to compare \Cref{alg:APclasses} with the computation of the rigidity matrix rank.
There is a randomized algorithm~\cite{Cheung2013}
allowing to compute the rank $r$ of an $m\times n$ matrix $A$, $n\leq m$, in $O(\ell+r^\omega)$ field operations,
where $\omega$ is the matrix multiplication exponent
and $\ell$ is the number of non-zero elements of $A$.
There is also a randomized algorithm~\cite{Saunders2004}
suitable for sparse matrices with complexity $O(n\ell)$.
In particular, the rank computation of
the rigidity matrix of $(G,\rho)$ with $n=|V_G|$, $m=|E_G|$
costs ${O(m+n^\omega)=O(n^\omega)}$, resp.\ $O(nm)$, operations
in the field of coordinates of $\rho$
since there are $4m$ non-zero entries.
To conclude, if the maximum degree $d$ is fixed,
the asymptotic behavior of \Cref{alg:APclasses} is better than the rank computation.
If the graph is dense, namely, $m$ is $\theta(n^2)$,
then the rank computation is asymptotically faster,
but one has to keep in mind that for the rank computation we count arithmetic operations in the field of realization coordinates whereas \Cref{alg:APclasses} works only with the input graph.

\subsection{Relation to NAC-colorings}\label{sec:nac}
NAC-colorings have been defined in \cite{flexibleLabelings} as a combinatorial property of graphs that have flexible frameworks.
We show here how \APclasses{} can be connected to NAC-colorings.
\begin{definition}
	Let~$G$ be a graph, possibly countably infinite.
	A coloring of edges $\delta\colon  E_G\rightarrow \{\text{\blue{}, \red{}}\}$
	is called a \emph{NAC-coloring},
	if it is surjective and for every cycle in~$G$,
	either all edges have the same color, or
	there are at least two edges in each color (see \Cref{fig:nac}).
\end{definition}
The main theorem for NAC-colorings shows their relation to flexes of a graph.
\begin{theorem}[\cite{DLinfinite,flexibleLabelings}]
	\label{thm:nacflexible}
	A connected non-trivial graph allows a flexible framework if and only if it has a NAC-coloring.
\end{theorem}
This result was generalized to periodic \cite{Dperiodic} and
rotational symmetry preserving flexes~\cite{DGLrotSymmetry,DLinfinite}.
Checking whether a general finite graph has a NAC-coloring is NP-complete~\cite{Garamvolgyi2022}.
Similarly to~\cite{GLbracing} we need a specific type of NAC-colorings in this paper.
\begin{definition}
	A NAC-coloring $\delta$ of a graph $G$ is called \emph{Cartesian}
	if no two distinct vertices are connected by a red and blue path simultaneously.
\end{definition}
	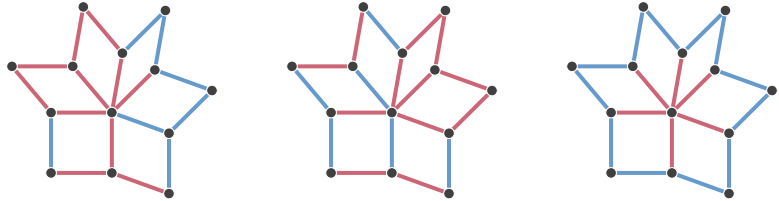
\begin{figure}[ht]
	  \centering
	  \begin{tikzpicture}[scale=0.8]
	    \node[gvertex] (a) at (0,0) {};
	    \node[gvertex] (b) at (1,0) {};
	    \node[gvertex,rotate around=40:(b)] (c) at ($(b)+(0,1)$) {};
	    \node[gvertex] (d) at ($(a)+(c)-(b)$) {};
	    \node[gvertex,rotate around=-10:(b)] (e) at ($(b)+(0,1)$) {};
	    \node[gvertex] (f) at ($(c)+(e)-(b)$) {};
	    \node[gvertex,rotate around=-45:(b)] (g) at ($(b)+(0,1)$) {};
	    \node[gvertex] (h) at ($(e)+(g)-(b)$) {};
	    \node[gvertex,rotate around=-110:(b)] (i) at ($(b)+(0,1)$) {};
	    \node[gvertex] (j) at ($(g)+(i)-(b)$) {};
	    \node[gvertex,rotate around=-180:(b)] (k) at ($(b)+(0,1)$) {};
	    \node[gvertex] (l) at ($(i)+(k)-(b)$) {};
	    \node[gvertex] (m) at ($(a)+(k)-(b)$) {};
	    \draw[redge] (a)--(b) (b)--(c) (c)--(d) (d)--(a) (b)--(e) (k)--(l) (b)--(k) (m)--(k) (e)--(f) (f)--(c) (b)--(g);
	    \draw[bedge] (g)--(h) (h)--(e) (b)--(i) (i)--(j) (j)--(g) (l)--(i) (a)--(m);
	  \end{tikzpicture}
	  \qquad
	  \begin{tikzpicture}[scale=0.8]
	    \node[gvertex] (a) at (0,0) {};
	    \node[gvertex] (b) at (1,0) {};
	    \node[gvertex,rotate around=40:(b)] (c) at ($(b)+(0,1)$) {};
	    \node[gvertex] (d) at ($(a)+(c)-(b)$) {};
	    \node[gvertex,rotate around=-10:(b)] (e) at ($(b)+(0,1)$) {};
	    \node[gvertex] (f) at ($(c)+(e)-(b)$) {};
	    \node[gvertex,rotate around=-45:(b)] (g) at ($(b)+(0,1)$) {};
	    \node[gvertex] (h) at ($(e)+(g)-(b)$) {};
	    \node[gvertex,rotate around=-110:(b)] (i) at ($(b)+(0,1)$) {};
	    \node[gvertex] (j) at ($(g)+(i)-(b)$) {};
	    \node[gvertex,rotate around=-180:(b)] (k) at ($(b)+(0,1)$) {};
	    \node[gvertex] (l) at ($(i)+(k)-(b)$) {};
	    \node[gvertex] (m) at ($(a)+(k)-(b)$) {};
	    \draw[redge] (a)--(b)  (c)--(d) (b)--(e) (f)--(c)
	    	(k)--(l) (b)--(g) (g)--(h) (h)--(e) (m)--(k) (b)--(i) (i)--(j) (j)--(g);
	    \draw[bedge] (b)--(c) (d)--(a)  (e)--(f) (b)--(k) (a)--(m) (l)--(i);
	  \end{tikzpicture}
	  \qquad
	  \begin{tikzpicture}[scale=0.8]
	    \node[gvertex] (a) at (0,0) {};
	    \node[gvertex] (b) at (1,0) {};
	    \node[gvertex,rotate around=40:(b)] (c) at ($(b)+(0,1)$) {};
	    \node[gvertex] (d) at ($(a)+(c)-(b)$) {};
	    \node[gvertex,rotate around=-10:(b)] (e) at ($(b)+(0,1)$) {};
	    \node[gvertex] (f) at ($(c)+(e)-(b)$) {};
	    \node[gvertex,rotate around=-45:(b)] (g) at ($(b)+(0,1)$) {};
	    \node[gvertex] (h) at ($(e)+(g)-(b)$) {};
	    \node[gvertex,rotate around=-110:(b)] (i) at ($(b)+(0,1)$) {};
	    \node[gvertex] (j) at ($(g)+(i)-(b)$) {};
	    \node[gvertex,rotate around=-180:(b)] (k) at ($(b)+(0,1)$) {};
	    \node[gvertex] (l) at ($(i)+(k)-(b)$) {};
	    \node[gvertex] (m) at ($(a)+(k)-(b)$) {};
	    \draw[redge] (a)--(b) (b)--(e) (b)--(c) (b)--(g)  (b)--(i) (b)--(k);
	    \draw[bedge]  (c)--(d) (d)--(a) (f)--(c) (k)--(l) (g)--(h)
	    	 (i)--(j) (j)--(g)(h)--(e) (m)--(k) (e)--(f) (a)--(m) (l)--(i);
	  \end{tikzpicture}
	  \caption{A coloring that is not a NAC-coloring (left), a Cartesian NAC-coloring (middle)
	  	and a NAC-coloring that is not Cartesian (right).}
	  \label{fig:nac}
	\end{figure}

The name \emph{Cartesian} is motivated by the relation to the Cartesian product of graphs,
see the discussion in \cite[Appendix~A]{GLbracing}.
Prior to the references mentioned in the discussion,
finite graphs that are non-trivial subgraphs of Cartesian products
were characterized in~\cite{Klavzar2005}
using edge colorings (called \emph{2-labelings with Condition~A} in the paper) by two colors satisfying
the condition stated in the following lemma, which proves that these are exactly Cartesian NAC-colorings.
\begin{lemma}
	\label{lem:cartNACIffCondA}
	Let $G$ be a connected graph.
	A surjective edge coloring of $G$ by red and blue is a Cartesian NAC-coloring
	if and only if for each induced non-monochromatic cycle the colors change
	at least three times while passing the cycle.  
\end{lemma}
\begin{proof}
	It is trivial that every Cartesian NAC-coloring satisfies the condition above:
	if the colors in a non-monochromatic induced cycle would change only twice,
	then the two vertices where the changes occur were connected by a red and a blue path simultaneously.
	
	In order to show that an edge coloring $\delta$  of $G$ satisfying the given condition is a Cartesian NAC-coloring,
	we follow the idea of the injectivity part of the proof of \cite[Theorem~2.2]{Klavzar2005}.
	We show that all non-monochromatic cycles have at least three changes of colors, not only induced ones.
	Suppose for contradiction that there is a cycle with only two color changes.
	Let $C$ be such a cycle of the minimal length and $P_\red$ and $P_\blue$ be the red and blue paths it consists of.
	By assumption, $C$ is not induced, namely, there is an edge $uv\in E_G$ such that $u,v$ are in $C$ but $uv$ is not in $C$.
	Suppose $u,v$ are both in $P_\red$, then $\delta(uv)$ cannot be red since this would contradict the minimality of $C$.
	But $\delta(uv)$ cannot be blue either since then the part of $P_\red$ from $u$ to $v$ together with the edge $uv$
	would create a cycle with two color changes shorter than $C$.
	Similarly, $u,v$ cannot be both in $P_\blue$.
	Notice that we have also covered the case where $u$ or $v$ is a vertex where a color change occurs.
	The only case left is when $u$ is a vertex with two incident blue edges of $C$ and $v$ with two such red edges.
	But this also gives a contradiction as independently of $\delta(uv)$ being red or blue,
	there would be a non-monochromatic cycle with two color changes shorter than $C$.
	Therefore, all non-monochromatic cycles have at least three color changes,
	hence, no two distinct vertices are connected by a red and blue path simultaneously.
	This also implies that there are at least two red and two blue edges in every non-monochromatic cycle.
	Hence, $\delta$ is a Cartesian NAC-coloring.
\end{proof}

After this detour to definitions from the literature we now relate Cartesian NAC-colorings to \APclasses.
\begin{lemma}
	\label{lem:monochromaticAPclassIffCartesian}
	Let $(G,\rho)$ be walk-independent.
	A surjective edge coloring $\delta$ of $G$ by red and blue is a Cartesian NAC-coloring
	if and only if every \APclass{} is monochromatic.
\end{lemma}
\begin{proof}
	$\implies:$ since $\delta$ is a NAC-coloring, 3-cycles are monochromatic.
	The opposite edges of a 4-cycle in a Cartesian NAC-coloring have the same color as well.
	Hence, edges in relation $\Trel$ or $\Prel$ have the same color,
	which gives that \APclasses{} are monochromatic.

	$\impliedby:$
	First, we show that $\delta$ is a NAC-coloring.
	Consider a cycle $C$.
	If $uv$ is an edge in $C$ which belongs to an \APclass{} $r$,
	there is another edge $e$ in $C$ belonging to $r$
	since $r$ separates $u$ and $v$ by \Cref{lem:APclassesAreEdgeCuts}.
	Therefore, if a color occurs in $C$, it occurs at least twice as $\delta(uv)=\delta(e)$
	by the assumption that \APclasses{} are monochromatic.

	It cannot happen that two vertices are connected by a red and a blue path simultaneously,
	since they are separated by an \APclass{} by \Cref{lem:APclassesAreEdgeCuts}.
	Hence, $\delta$ is a Cartesian NAC-coloring.
\end{proof}

Finally Cartesian NAC-colorings form a correspondence to flexible walk-independent frameworks.
\begin{corollary}
	\label{cor:FlexIffCartNAC}
	For a walk-independent framework $(G,\rho)$, the following are equivalent:
	\begin{enumerate}
	  \item $(G,\rho)$ is flexible,
	  \item $G$ has a Cartesian NAC-coloring,
	  \item $G$ is a non-trivial subgraph of the Cartesian product of graphs.
	\end{enumerate}
\end{corollary}
\begin{proof}
	By \Cref{prop:flexImpliesTwoAPclasses}, if $(G,\rho)$ is flexible,
	then $G$ has at least two \APclasses{}.
	Coloring the edges of some of the \APclasses{} by red and the rest by blue
	gives a Cartesian NAC-coloring by \Cref{lem:monochromaticAPclassIffCartesian}.
	On the other hand, if $G$ has a Cartesian NAC-coloring,
	there must be at least two \APclasses{} by \Cref{lem:monochromaticAPclassIffCartesian},
	hence $(G,\rho)$ is flexible by \Cref{prop:twoAPclassesImplyFlex}.
	
	The existence of a Cartesian NAC-coloring is equivalent to being
	a non-trivial subgraph of the Cartesian product of graphs
	by \Cref{lem:cartNACIffCondA} and \cite[Theorem~2.2]{Klavzar2005}.
\end{proof}
We remark that another characterization of a graph being a non-trivial subgraph
of the Cartesian product of graphs is given in~\cite{Klavzar2002}.
This characterization was used to prove that determining whether
a general graph is a non-trivial subgraph of the Cartesian product of graphs is NP-complete in~\cite{Hellmuth2013},
which is apparently not the case for graphs that admit a walk-independent framework by our results.

\subsection{\Pframeworks{}}\label{sec:pframe}
In this subsection, we show that the results generalize those on frameworks consisting of parallelograms only and those where diagonals of 4-cycles are added \cite{GLbracing}.

\begin{definition}
	Let $G$ be a connected graph. Recall the relation $\Prel$ on the set of edges, where
	two edges are in relation if they are opposite edges of a 4-cycle subgraph of $G$.
	An equivalence class of the reflexive-transitive closure of $\Prel$ is called a \emph{ribbon}.
    If every ribbon of~$G$ is an edge cut and
	$\rho$ is a parallelogram placement of $G$, we call the framework~$(G,\rho)$ a \emph{\Pframework{}}.
\end{definition}
Notice that there are no odd cycles in \Pframeworks{} \cite[Theorem~3.9]{GLbracing}. However, we can get triangles by adding diagonals to 4-cycles. We call this a bracing.
\begin{definition}
	A \emph{braced \Pframework{}} is a framework $(G,\rho)$ such that
	${G=(V_G,E_{G'}\cup E_d)}$ where $E_{G'}$ and $E_d$ are two non-empty disjoint sets where the edges in $E_d$ correspond to diagonals of some 4-cycles of \emph{the underlying unbraced subgraph} $G'=(V_G,E_{G'})$
	(these diagonals are also called \emph{braces})
	and $(G',\rho)$ is a \Pframework.
\end{definition}
Now we can relate these frameworks to walk-independence.
\begin{lemma}
    \label{lem:zeroSumPframework}
    A braced \Pframework{} $(G,\rho)$ is walk-independent.
\end{lemma}
\begin{proof}
    Let $r$ be an \APclass{} and $C$ be a closed walk in $G$.
    The sum in \Cref{eq:zeroSumAPclass} is preserved if we replace an edge $uv$ in~$C$
    which is a brace of the 4-cycle $(u,x,v,z,u)$ by $ux$ and~$xv$
    since they belong to the same \APclass{}, see \Cref{fig:replacingBrace}.
    Hence, we can assume that $C$ is in the underlying unbraced \Pframework{} $(G',\rho)$.
    Let $r_1,\ldots, r_k$ be all ribbons of $G'$ that occur in $r\cap C$.
    By \cite[Lemmas 3.2 and 3.7]{GLbracing} (or their infinite dimensional analogues),
    we have
	\begin{equation*}
		\sum_{\oriented{u}{v} \in r_i\cap C} (\rho(v)-\rho(u)) = (0,0)
	\end{equation*}
	for every ribbon $r_i$.
	Summing up the equation for $i\in \{1,\ldots,k\}$ gives the desired statement,
	since each ribbon $r_i$ is a subset of $r$ by construction of the corresponding equivalences.
    \begin{figure}[ht]
        \centering
         \begin{tikzpicture}
             \begin{scope}
                 \node[gvertex,label={below:$u$}] (a) at (0,0) {};
                 \node[gvertex,label={below right:$z$}] (b) at (1,0) {};
                 \node[gvertex,label={above:$v$}] (c) at (1,1.3) {};
                 \node[gvertex,label={above left:$x$}] (d) at (0,1.3) {};
                 \draw[dedge] ($(a)-(1.1,0.3)$)--(a);
                 \draw[dedge] (c)edge($(c)+(0.9,0.5)$);
                 \draw[dedge] (a)edge(c);
                 \draw[edge,colgraphe!10!white] (a)edge(d) (b)edge(c) (a)edge(b) (c)edge(d);
             \end{scope}
             \begin{scope}[xshift=5cm]
                 \node[gvertex,label={below:$u$}] (a) at (0,0) {};
                 \node[gvertex,label={below right:$z$}] (b) at (1,0) {};
                 \node[gvertex,label={above:$v$}] (c) at (1,1.3) {};
                 \node[gvertex,label={above left:$x$}] (d) at (0,1.3) {};
                 \draw[dedge] ($(a)-(1.1,0.3)$)--(a);
                 \draw[dedge] (c)edge($(c)+(0.9,0.5)$);
                 \draw[dedge,col1] (a)edge(d);
                 \draw[dedge,col3] (d)edge(c);
                 \draw[edge,colgraphe!10!white] (a)edge(c) (b)edge(c) (a)edge(b);
             \end{scope}
         \end{tikzpicture}
        \caption{Replacing a brace by edges of the underlying unbraced subgraph.}
        \label{fig:replacingBrace}
    \end{figure}
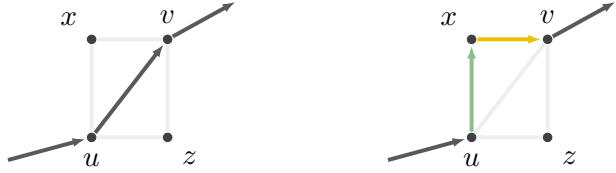
\end{proof}
Finally we get the following known result in terms of the notation in this paper.
\begin{corollary}
    \label{cor:Pframeworks}
    Let $(G,\rho)$ be a braced \Pframework{}. The following statements are equivalent:
    \begin{enumerate}
        \item\label{it:PfRigid} $(G,\rho)$ is flexible,
        \item\label{it:PfInfRigid} $(G,\rho)$ is infinitesimally flexible,
        \item\label{it:PfOneAPclass} $G$ has at least two \APclasses{},
        \item\label{it:PfNoNAC} $G$ has a Cartesian NAC-coloring.
    \end{enumerate}
    If $G$ has $\ell$ \APclasses{}, where $2\leq \ell\in \NN$, then
    $G$ is a non-trivial subgraph of the Cartesian product of $\ell$ graphs. 
\end{corollary}
\begin{proof}
    The statement follows from \Cref{thm:main} and \Cref{lem:zeroSumPframework} and
    the fact that a 3-cycle can occur only in a braced 4-cycle, whose vertices are not collinear by definition.
\end{proof}
The equivalence of \ref{it:PfRigid} and \ref{it:PfNoNAC} was proven for the finite
case in~\cite{GLbracing} and for the infinite case in~\cite{DLinfinite}.
For \Pframeworks{} obtained from parallelogram tilings the equivalence with \ref{it:PfInfRigid}
follows from \cite{power2021parallelogram}.
Since the number of connected components of the bracing graph as defined in~\cite{GLbracing}
is the same as the number of \APclasses{}, Theorem~1.1 of~\cite{GLbracing}
is implied by \Cref{cor:Pframeworks}.

\subsection{\TPframeworks{}}\label{sec:tpframe}
We define a class of graphs which ``consist of 3- and 4-cycles'' and whose parallelogram placements are walk-independent.
The corresponding frameworks are called \TPframeworks{}. For their definition we first need the notion of a simplicial complex.

\begin{definition}
	An \emph{abstract simplicial complex} $X$ is a (possibly infinite) collection of finite sets such that if $S$ is in $X$,
	then any non-empty subset of $S$ is in $X$.
	The union of all sets in $X$ is the \emph{vertex set} of $X$.
	Each element $S$ of $X$ is called a \emph{face}, the \emph{dimension} of $S$ is defined to be~$|S|-1$.
	The \emph{dimension} of $X$ is the supremum of the dimensions of the faces.
	The $k$-skeleton of $X$ is the abstract simplicial complex formed by all faces of $X$ of dimension at most $k$.
\end{definition}
In the following, we abuse the notation and identify singletons $\{v\}\in X$ with $v$ and we write $uv$ instead of $\{u,v\}$
for 1-dimensional faces of $X$ or edges of a graph.

\begin{definition}
	Let $u,v$ be vertices of an abstract simplicial complex $X$.
	A \emph{walk from~$u$ to~$v$} is a sequence of vertices $W=(u=u_0, u_1, \dots, u_n=v), n\geq 1$, such that
	$\{u_{i-1},u_i\}$ is in~$X$ for all $1\leq i\leq n$.
	Notice that we allow $u_{i-1}=u_i$.
	If $u=v$, the sequence $W$ is called a \emph{closed walk based at $u$}.

	We define a \emph{move} on a walk $(u_0, u_1, \dots, u_n)$ in $X$ by omitting vertex $u_i$
	for some $1\leq i\leq n-1$ such that $\{u_{i-1},u_i,u_{i+1}\}$ is a face of $X$,
	or the inverse operation.
	Two closed walks based at the same point are \emph{homotopic}
	if one can be obtained from the other by a sequence of moves or inverse moves.
\end{definition}

\begin{definition}
	An abstract simplicial complex $X$ is called \emph{simply connected}
	if its 1-skeleton is a connected graph
	and every closed walk based at $u$, where $u$ is a fixed vertex of~$X$, is homotopic to $(u,u)$.
	Notice that this is independent of the choice of $u$.
\end{definition}

We now need a specific simplicial complex.
\begin{definition}
	Let $G=(V_G,E_G)$ be a graph.
	We define a 2-dimensional abstract simplicial complex
	\begin{align*}
		\asc{G} = V_G &\cup E_G\\ &\cup \sst{\{u,v,w\}}{(u,v,w,u) \text{ is a 3-cycle in } G} \\
			&\cup \sst{\{u,w\}}{(u,v,w,z,u) \text{ is a 4-cycle in } G \text{ for some } v,z \in V_G}\\
			&\cup \sst{\{u,v,w\}}{(u,v,w,z,u) \text{ is an induced 4-cycle in } G \text{ for some } z \in V_G}\,.
	\end{align*}
	Let $\rho$ be a parallelogram placement of $G$
	such that every 3-cycle is a non-degenerate	triangle.
	If $\asc{G}$ is simply connected, the framework $(G, \rho)$ is called a \emph{\TPframework}.
\end{definition}

We remark that once we allow vertices to repeat in (closed) walks (in graph theoretical sense) in a graph $G$ (although there are no loops),
then there is an obvious correspondence with (closed) walks in $\asc{G}$ in the sense of the definition above.
Hence, we do not distinguish them from now on.

Clearly, there are \TPframeworks{} which are not \Pframeworks{}, but also the converse is true:
\Cref{fig:PnotTP} shows a \Pframework{} which is not a \TPframework{}.
The left framework in \Cref{fig:crossingprop-ex} is neither a \Pframework{} nor a \TPframework{},
but still is walk-independent.
\begin{figure}[ht]
    \centering
    \begin{tikzpicture}[scale=1.5]
        \coordinate (x) at (0.4,0.4);
        \coordinate (y) at (1,0);
        \coordinate (z) at (0,1);
        \foreach \i in {1,2,3}
        {
            \foreach \j in  {1,2,3}
            {
                \foreach \k in {0,1,2}
                {
                    \coordinate (c\i\j\k) at ($\i*(x)+\j*(y)+\k*(z)$);
                }
            }
        }
        \node[gvertex] (a110) at (c110) {};
        \node[gvertex] (a120) at (c120) {};
        \node[gvertex] (a130) at (c130) {};
        \node[gvertex] (a210) at (c210) {};
        \node[gvertex] (a220) at (c220) {};
        \node[gvertex] (a230) at (c230) {};
        \node[gvertex] (a320) at (c320) {};
        \node[gvertex] (a330) at (c330) {};
        \node[gvertex] (a111) at (c111) {};
        \node[gvertex] (a121) at (c121) {};
        \node[gvertex] (a131) at (c131) {};
        \node[gvertex] (a211) at (c211) {};
        \node[gvertex] (a221) at (c221) {};
        \node[gvertex] (a231) at (c231) {};
        \node[gvertex] (a321) at (c321) {};
        \node[gvertex] (a331) at (c331) {};
        \node[gvertex] (a112) at (c112) {};
        \node[gvertex] (a122) at (c122) {};
        \node[gvertex] (a212) at (c212) {};
        \node[gvertex] (a222) at (c222) {};
        \draw[edge,col1] (a212)edge(a222) (a211)edge(a221)  (a210)edge(a220);
        \draw[edge,col2] (a320)edge(a330) (a321)edge(a331) (a220)edge(a230);
        \draw[edge,col3] (a110)edge(a210) (a120)edge(a220) (a130)edge(a230) (a111)edge(a211) (a131)edge(a231) (a112)edge(a212) (a122)edge(a222);
        \draw[edge,col4] (a220)edge(a320) (a230)edge(a330) (a221)edge(a321) (a231)edge(a331);

        \foreach \a in {11,12,13,21,23,32,33}
        {
            \draw[edge,col5] (a\a0)edge(a\a1);
        }
        \foreach \a in {11,12,21,22}
        {
            \draw[edge,col6] (a\a1)edge(a\a2);
        }

        \draw[edge,col1] (a110)edge(a120) (a111)edge(a121) (a112)edge(a122);
        \draw[edge,col2] (a120)edge(a130) (a121)edge(a131) (a221)edge(a231);
    \end{tikzpicture}
    \qquad
    \begin{tikzpicture}[scale=1.5,face/.style={col5,opacity=0.5},hedge/.style={edge,dotted},pedge/.style={edge,dashed}]
        \coordinate (x) at (0.4,0.4);
        \coordinate (y) at (1,0);
        \coordinate (z) at (0,1);
        \foreach \i in {1,2,3}
        {
            \foreach \j in  {1,2,3}
            {
                \foreach \k in {0,1,2}
                {
                    \coordinate (c\i\j\k) at ($\i*(x)+\j*(y)+\k*(z)$);
                }
            }
        }

        \fill[face] (c320) rectangle (c331);
        \fill[face] (c220) -- (c230) -- (c330) -- (c320) -- (c220) -- cycle;
        \fill[face] (c211) rectangle (c222);
        \fill[face] (c110) -- (c120) -- (c220) -- (c210) -- (c110) -- cycle;
        \fill[face] (c120) -- (c130) -- (c230) -- (c220) -- (c120) -- cycle;
        \fill[face] (c110) -- (c210) -- (c211) -- (c111) -- (c110) -- cycle;
        \fill[face] (c111) -- (c211) -- (c212) -- (c112) -- (c111) -- cycle;
        \node[gvertex] (a210) at (c210) {};
        \node[gvertex] (a220) at (c220) {};
        \node[gvertex] (a320) at (c320) {};
        \node[gvertex] (a211) at (c211) {};

        \draw[pedge] (a320)edge(a321);
        \draw[hedge] (a220)edge(a320) (a320)edge(a330);

        \draw[pedge] (a211)edge(a221);
        \draw[edge] (a221)edge(a222);
        \draw[hedge] (a110)edge(a210) (a210)edge(a211) (a210)edge(a220);
        \draw[hedge] (a120)edge(a220) (a220)edge(a230);
        \draw[hedge] (a111)edge(a211) (a211)edge(a212);

        \fill[face] (c110) rectangle (c121);
        \fill[face] (c120) rectangle (c131);
        \fill[face] (c111) rectangle (c122);
        \draw[edge] (a110)edge(a120) (a120)edge(a121) (a121)edge(a111) (a111)edge(a110);
        \node[gvertex] (a111) at (c111) {};
        \node[gvertex] (a121) at (c121) {};
        \node[gvertex] (a110) at (c110) {};
        \draw[edge] (a120)edge(a130) (a121)edge(a131);
        \node[gvertex] (a120) at (c120) {};
        \fill[face] (c130) -- (c230) -- (c231) -- (c131) -- (c130) -- cycle;
        \node[gvertex] (a131) at (c131) {};
        \draw[edge] (a111)edge(a112) (a121)edge(a122);

        \fill[face] (c230) -- (c330) -- (c331) -- (c231) -- (c230) -- cycle;
        \draw[edge] (a130)edge(a230) (a230)edge(a231) (a231)edge(a131) (a131)edge(a130);
        \node[gvertex] (a130) at (c130) {};
        \node[gvertex] (a230) at (c230) {};
        \fill[face] (c221) -- (c231) -- (c331) -- (c321) -- (c221) -- cycle;
        \draw[edge] (a230)edge(a330) (a330)edge(a331) (a331)edge(a231);
        \node[gvertex] (a330) at (c330) {};

        \draw[edge] (a221)edge(a231) (a331)edge(a321) (a321)edge(a221);
        \node[gvertex] (a221) at (c221) {};
        \node[gvertex] (a231) at (c231) {};
        \node[gvertex] (a321) at (c321) {};
        \node[gvertex] (a331) at (c331) {};
        \fill[face] (c112) -- (c122) -- (c222) -- (c212) -- (c112) -- cycle;
        \draw[edge] (a112)edge(a122) (a122)edge(a222) (a222)edge(a212) (a212)edge(a112);
        \node[gvertex] (a112) at (c112) {};
        \node[gvertex] (a122) at (c122) {};
        \node[gvertex] (a212) at (c212) {};
        \node[gvertex] (a222) at (c222) {};
    \end{tikzpicture}
    \caption{An example of a \Pframework\ which is not a \TPframework.}
    \label{fig:PnotTP}
\end{figure}
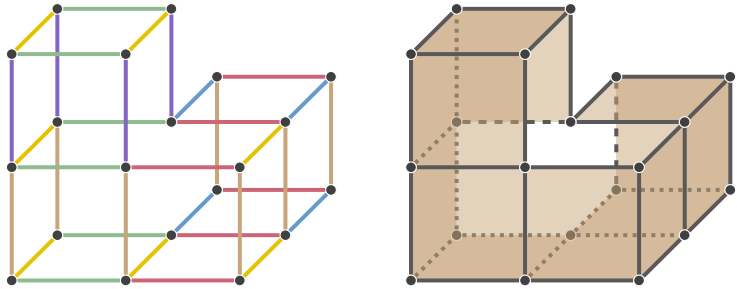

The same statement as \Cref{lem:zeroSumPframework} for \Pframeworks{} also holds for \TPframeworks.
\begin{lemma}
	\label{lem:zeroSumAPclass}
	A \TPframework{} $(G,\rho)$ is walk-independent.
\end{lemma}
\begin{proof}
	The idea of the proof is following: for an \APclass{} $r$,
	we aim to construct a map $T_r$ from the closed walks of $\asc{G}$ to $\RR^2$
	such that it evaluates to the left hand side of \Cref{eq:zeroSumAPclass} for a closed walk in~$G$.
	When we show that $T_r$ is invariant under the moves and equal to $(0,0)$
	for a trivial closed walk, namely, $(v,v)$ for $v\in V_G$,
	the statement follows since all closed walks are homotopic to a trivial one.
	In order to do so, we first define a map $\tau_r:V_G\times V_G \rightarrow \RR^2$ by
	\begin{equation*}
		\tau_r(u,v)=\begin{cases}
						\rho(v)-\rho(u) & \text{if } uv\in E_G\cap r\,,\\
						\rho(v)-\rho(u) & \text{if } uv\in \asc{G} \land uv\notin E_G \land (\exists z\in V_G: uz, zv\in r)\,,\\
						\rho(v)-\rho(z) & \text{if } uv\in \asc{G} \land uv\notin E_G \land (\exists z\in V_G: uz \in E_G \setminus r \land vz\in r)\,,\\
						(0,0) & \text{otherwise.}
					\end{cases}
	\end{equation*}
	We shall check that the map is well-defined:
	suppose that $uv\in \asc{G}$ and $uv\notin E_G$.
	Since $uv\in \asc{G}$, there is a 4-cycle $(u,z,v,z',u)$ in $G$. This 4-cycle is unique up to swapping $z$ and $z'$,
	in other words, there are only two common neighbors of $u$ and $v$.
	Indeed, if there was $z''\in V_G$ distinct from $z,z'$ and adjacent to both $u$ and $v$ in $G$,
	then $\rho(z')=\rho(z'')$ as both $(u,z,v,z',u)$ and $(u,z,v,z'',u)$ are parallelograms in $\rho$,
	which is a contradiction.
	Since the opposite edges in a 4-cycle are always in the same \APclass,
	either the second, third, or the fourth case in the definition of $\tau_r$ occurs, see \Cref{fig:taur}.
	Moreover, vertex $z$ in the third case is unique.

	\begin{figure}[ht]
		\centering
		\begin{tikzpicture}
			\node[fvertex,label={below left:$u$}] (u) at (0,0) {};
			\node[fvertex,label={below right:$z$}] (z) at (1,0) {};
			\node[fvertex,label={above right:$v$}] (v) at (1.5,1.5) {};
			\node[fvertex,label={above left:$z'\!$}] (zp) at (0.5,1.5) {};
			\draw[edge] (u)--(z);
			\draw[edge] (z)--(v);
			\draw[edge] (v)--(zp);
			\draw[edge] (zp)--(u);
			\draw[taur,dotted] (u)--(v);

			\begin{scope}[xshift=3cm]
				\node[fvertex,label={below left:$u$}] (u) at (0,0) {};
				\node[fvertex,label={below right:$z$}] (z) at (1,0) {};
				\node[fvertex,label={above right:$v$}] (v) at (1.5,1.5) {};
				\node[fvertex,label={above left:$z'\!$}] (zp) at (0.5,1.5) {};
				\draw[nonribbon] (u)--(z);
				\draw[nonribbon] (v)--(zp);
				\draw[edge] (zp)--(u);
				\draw[nonedge] (u)--(v);
				\draw[taur] (z) -- (v);
			\end{scope}

			\begin{scope}[xshift=6cm]
				\node[fvertex,label={below left:$u$}] (u) at (0,0) {};
				\node[fvertex,label={below right:$z$}] (z) at (1,0) {};
				\node[fvertex,label={above right:$v$}] (v) at (1.5,1.5) {};
				\node[fvertex,label={above left:$z'\!$}] (zp) at (0.5,1.5) {};
				\draw[edge] (u)--(z);
				\draw[nonribbon] (z)--(v);
				\draw[nonribbon] (zp)--(u);
				\draw[nonedge] (u)--(v);
				\draw[taur] (zp) -- (v);
			\end{scope}

			\begin{scope}[xshift=9cm]
				\node[fvertex,label={below left:$u$}] (u) at (0,0) {};
				\node[fvertex,label={below right:$z$}] (z) at (1,0) {};
				\node[fvertex,label={above right:$v$}] (v) at (1.5,1.5) {};
				\node[fvertex,label={above left:$z'\!$}] (zp) at (0.5,1.5) {};
				\draw[nonribbon] (u)--(z);
				\draw[nonribbon] (z)--(v);
				\draw[nonribbon] (v)--(zp);
				\draw[nonribbon] (zp)--(u);
				\draw[nonedge] (u)--(v);
				\node[taurZero] at (u) {};
			\end{scope}
		\end{tikzpicture}
 		\caption{The possible values (yellow) of $\tau_r(u,v)$ if $uv$ (dotted) belongs to $\asc{G}$ but not to $E_G$
 		depending whether the edges of the corresponding 4-cycle are in $r$ (solid) or not (dashed).}
		\label{fig:taur}
	\end{figure}
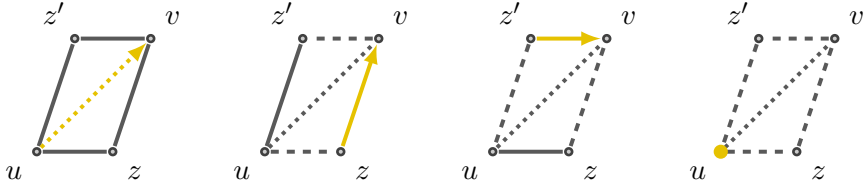

	Now we set
	\begin{equation*}
		T_r(C) = \sum_{\oriented{u}{v} \in C} \tau_r(u,v)
	\end{equation*}
	for a closed walk $C$ in $\asc{G}$.
	Since closed walks are always finite even if $G$ is infinite, the sum is well-defined.
	If $C$ is a closed walk in~$G$, then $T_r(C)$ equals the left hand side of \Cref{eq:zeroSumAPclass}.
	If $u\in V_G$, then $T_r((u,u))=\tau_r(u,u)=(0,0)$.
	Hence, the only fact left to be shown is 
	that $T_r(C_1)=T_r(C_2)$, where the closed walk $C_1$ is obtained from the closed walk~$C_2$ by a move.
	In particular, supposing that $u_1,u_2,u_3$ are three consecutive vertices in $C_2$
	and $u_2$ is omitted to get $C_1$ using that $\{u_1, u_2, u_3\}$ is a face of $\asc{G}$,
	we have to show that
	\begin{equation}
		\label{eq:TrInv}
		\tau_r(u_1,u_3)= \tau_r(u_1,u_2) + \tau_r(u_2, u_3)\,.
	\end{equation}

	If $u_1=u_2$ or $u_2=u_3$, then \Cref{eq:TrInv} holds since $\tau_r(u_2, u_2)=0$.
	We settle the case $u_1=u_3$ by showing that $\tau_r(u,v)=-\tau_r(v,u)$ for any $u,v\in V_G$:
	the only non-trivial case is when $uv\in \asc{G}\setminus E_G$,
	i.e., there is an induced 4-cycle $(u,z,v,z',u)$ in $G$,
	and $uz\notin r$ while $vz\in r$.
	This 4-cycle is a parallelogram in $\rho$, therefore, $\rho(v)-\rho(z)=-(\rho(u)-\rho(z'))=-\tau(v,u)$
	as  $vz'\notin r$ and $z'u\in r$.

	Now we focus on the case when $u_1, u_2, u_3$ are distinct.
	Since the 2-dimensional faces of $\asc{G}$ have at most one non-edge of $G$ as a subface,
	we have that at most one of the three edges among these vertices is not in $E_G$.
	If all three edges are in $E_G$, then either all of them or none is in $r$.
	In both cases, \Cref{eq:TrInv} holds.

	Assume $u_1u_3\notin E_G$.
	Let $(u_1,u_2,u_3,z,u_1)$ be the unique induced 4-cycle in $G$ having $u_1u_3$ as a diagonal.
	The following four subcases are summarized in \Cref{fig:taurNonedge13Proof}.
	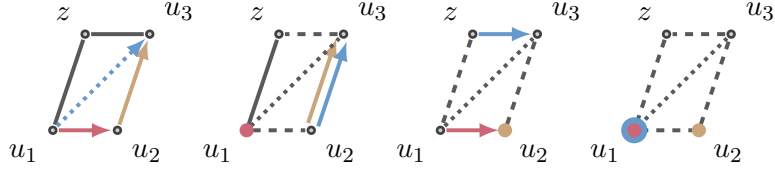
\begin{figure}[ht]
		\centering
		\begin{tikzpicture}[scale=0.85]
			\node[fvertex,label={below left:$u_1$}] (u1) at (0,0) {};
			\node[fvertex,label={below right:$u_2$}] (u2) at (1,0) {};
			\node[fvertex,label={above right:$u_3$}] (u3) at (1.5,1.5) {};
			\node[fvertex,label={above left:$z$}] (z) at (0.5,1.5) {};
			\draw[taur,tauR1] (u1)--(u2);
			\draw[taur,tauR2] (u2)--(u3);
			\draw[edge] (u3)--(z);
			\draw[edge] (z)--(u1);
			\draw[taur,dotted,tauL] (u1)--(u3);

			\begin{scope}[xshift=3cm]
				\node[fvertex,label={below left:$u_1$}] (u1) at (0,0) {};
				\node[fvertex,label={below right:$u_2$}] (u2) at (1,0) {};
				\node[fvertex,label={above right:$u_3$}] (u3) at (1.5,1.5) {};
				\node[fvertex,label={above left:$z$}] (z) at (0.5,1.5) {};
				\draw[nonribbon] (u1)--(u2);
				\draw[taur,tauR2,decorate,decoration={simple line, raise=2pt}] (u2)--(u3);
				\draw[nonribbon] (u3)--(z);
				\draw[edge] (z)--(u1);
				\draw[nonedge] (u1)--(u3);
				\draw[taur,tauL,decorate,decoration={simple line, raise=-2pt}] (u2) -- (u3);
				\node[taurZero,tauR1] at (u1) {};
			\end{scope}

			\begin{scope}[xshift=6cm]
				\node[fvertex,label={below left:$u_1$}] (u1) at (0,0) {};
				\node[fvertex,label={below right:$u_2$}] (u2) at (1,0) {};
				\node[fvertex,label={above right:$u_3$}] (u3) at (1.5,1.5) {};
				\node[fvertex,label={above left:$z$}] (z) at (0.5,1.5) {};
				\draw[taur,tauR1] (u1)--(u2);
				\draw[nonribbon] (u2)--(u3);
				\draw[nonribbon] (z)--(u1);
				\draw[nonedge] (u1)--(u3);
				\draw[taur,tauL] (z) -- (u3);
				\node[taurZero,tauR2] at (u2) {};
			\end{scope}

			\begin{scope}[xshift=9cm]
				\node[fvertex,label={below left:$u_1$}] (u1) at (0,0) {};
				\node[fvertex,label={below right:$u_2$}] (u2) at (1,0) {};
				\node[fvertex,label={above right:$u_3$}] (u3) at (1.5,1.5) {};
				\node[fvertex,label={above left:$z$}] (z) at (0.5,1.5) {};
				\draw[nonribbon] (u1)--(u2);
				\draw[nonribbon] (u2)--(u3);
				\draw[nonribbon] (u3)--(z);
				\draw[nonribbon] (z)--(u1);
				\draw[nonedge] (u1)--(u3);
				\node[taurZero,tauL, minimum size=9pt] at (u1) {};
				\node[taurZero,tauR1] at (u1) {};
				\node[taurZero,tauR2] at (u2) {};
			\end{scope}
		\end{tikzpicture}
 		\caption{The illustration of \Cref{eq:TrInv}, i.e.,
 		${\color{tauL}\underline{{\color{black}\tau_r(u_1,u_3)}}}= {\color{tauR1}\underline{{\color{black}\tau_r(u_1,u_2)}}} + {\color{tauR2}\underline{{\color{black}\tau_r(u_2, u_3)}}}$,
 		if $u_1u_3$ (dotted) belongs to $\asc{G}$ but not to $E_G$
 		depending whether the edges of the corresponding 4-cycle are in $r$ (solid) or not (dashed).}
		\label{fig:taurNonedge13Proof}
	\end{figure}

	If $u_1u_2, u_2 u_3 \in r$,	then
	 $\tau_r(u_1,u_3)= \rho(u_3)-\rho(u_1) = \rho(u_3)-\rho(u_2) + \rho(u_2) -\rho(u_1)	= \tau_r(u_2, u_3) + \tau_r(u_1,u_2)$.
	Hence, \Cref{eq:TrInv} holds.
	If $u_1u_2\notin r$ and ${u_2 u_3\in r}$,
	then $\tau_r(u_1,u_3)=\rho(u_3)-\rho(u_2) = \tau_r(u_2,u_3)$.
	Since $\tau_r(u_1,u_2)=(0,0)$, we have  \Cref{eq:TrInv}.
	If $u_1u_2\in r$ and ${u_2 u_3\notin r}$,
	then $\tau_r(u_1,u_3)=\rho(u_3)-\rho(z) = \rho(u_2) -\rho(u_1) = \tau_r(u_2,u_1)$.
	\Cref{eq:TrInv} holds since $\tau_r(u_2,u_3)=(0,0)$.
	If $u_1u_2, u_2 u_3 \notin r$, then $u_1z,zu_3\notin r$.
	Thus, $\tau_r(u_1,u_3)=(0,0)= \tau_r(u_1,u_2)= \tau_r(u_2, u_3)$,
	which gives again that \Cref{eq:TrInv} holds.

	The proofs when $u_1u_2\notin E_G$, resp.\ $u_2u_3\notin E_G$ are similar,
	see \Cref{fig:taurNonedge12Proof,fig:taurNonedge23Proof}.\qedhere
	\begin{figure}[ht]
		\centering
		\begin{tikzpicture}[scale=0.85]
			\node[fvertex,label={below left:$u_1$}] (u1) at (0,0) {};
			\node[fvertex,label={below right:$u_3$}] (u3) at (1,0) {};
			\node[fvertex,label={above right:$u_2$}] (u2) at (1.5,1.5) {};
			\node[fvertex,label={above left:$z$}] (z) at (0.5,1.5) {};
			\draw[taur,tauL] (u1)--(u3);
			\draw[taur,tauR2] (u2)--(u3);
			\draw[edge] (u2)--(z);
			\draw[edge] (z)--(u1);
			\draw[taur,dotted,tauR1] (u1)--(u2);

			\begin{scope}[xshift=3cm]
				\node[fvertex,label={below left:$u_1$}] (u1) at (0,0) {};
				\node[fvertex,label={below right:$u_3$}] (u3) at (1,0) {};
				\node[fvertex,label={above right:$u_2$}] (u2) at (1.5,1.5) {};
				\node[fvertex,label={above left:$z$}] (z) at (0.5,1.5) {};
				\draw[nonribbon] (u1)--(u3);
				\draw[taur,tauR2,decorate,decoration={simple line, raise=2pt}] (u2)--(u3);
				\draw[nonribbon] (u2)--(z);
				\draw[edge] (z)--(u1);
				\draw[nonedge] (u1)--(u2);
				\draw[taur,tauR1,decorate,decoration={simple line, raise=2pt}] (u3) -- (u2);
				\node[taurZero,tauL] at (u1) {};
			\end{scope}

			\begin{scope}[xshift=6cm]
				\node[fvertex,label={below left:$u_1$}] (u1) at (0,0) {};
				\node[fvertex,label={below right:$u_3$}] (u3) at (1,0) {};
				\node[fvertex,label={above right:$u_2$}] (u2) at (1.5,1.5) {};
				\node[fvertex,label={above left:$z$}] (z) at (0.5,1.5) {};
				\draw[taur,tauL] (u1)--(u3);
				\draw[nonribbon] (u3)--(u2);
				\draw[nonribbon] (z)--(u1);
				\draw[nonedge] (u1)--(u2);
				\draw[taur,tauR1] (z) -- (u2);
				\node[taurZero,tauR2] at (u2) {};
			\end{scope}

			\begin{scope}[xshift=9cm]
				\node[fvertex,label={below left:$u_1$}] (u1) at (0,0) {};
				\node[fvertex,label={below right:$u_3$}] (u3) at (1,0) {};
				\node[fvertex,label={above right:$u_2$}] (u2) at (1.5,1.5) {};
				\node[fvertex,label={above left:$z$}] (z) at (0.5,1.5) {};
				\draw[nonribbon] (u1)--(u3);
				\draw[nonribbon] (u3)--(u2);
				\draw[nonribbon] (u2)--(z);
				\draw[nonribbon] (z)--(u1);
				\draw[nonedge] (u1)--(u2);
				\node[taurZero,tauL, minimum size=9pt] at (u1) {};
				\node[taurZero,tauR1] at (u1) {};
				\node[taurZero,tauR2] at (u2) {};
			\end{scope}
		\end{tikzpicture}
 		\caption{The illustration of \Cref{eq:TrInv}, i.e.,
 		${\color{tauL}\underline{{\color{black}\tau_r(u_1,u_3)}}}= {\color{tauR1}\underline{{\color{black}\tau_r(u_1,u_2)}}} + {\color{tauR2}\underline{{\color{black}\tau_r(u_2, u_3)}}}$,
 		if $u_1u_2$ (dotted) belongs to $\asc{G}$ but not to $E_G$
 		depending whether the edges of the corresponding 4-cycle are in $r$ (solid) or not (dashed).}
		\label{fig:taurNonedge12Proof}
	\end{figure}
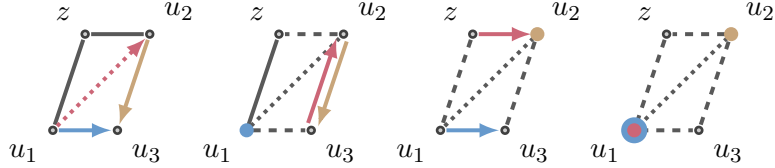

	\begin{figure}[ht]
		\centering
		\begin{tikzpicture}[scale=0.85]
			\node[fvertex,label={below left:$u_2$}] (u2) at (0,0) {};
			\node[fvertex,label={below right:$u_1$}] (u1) at (1,0) {};
			\node[fvertex,label={above right:$u_3$}] (u3) at (1.5,1.5) {};
			\node[fvertex,label={above left:$z$}] (z) at (0.5,1.5) {};
			\draw[taur,tauR1] (u1)--(u2);
			\draw[taur,tauL] (u1)--(u3);
			\draw[edge] (u3)--(z);
			\draw[edge] (z)--(u2);
			\draw[taur,dotted,tauR2] (u2)--(u3);

			\begin{scope}[xshift=3cm]
				\node[fvertex,label={below left:$u_2$}] (u2) at (0,0) {};
				\node[fvertex,label={below right:$u_1$}] (u1) at (1,0) {};
				\node[fvertex,label={above right:$u_3$}] (u3) at (1.5,1.5) {};
				\node[fvertex,label={above left:$z$}] (z) at (0.5,1.5) {};
				\draw[nonribbon] (u2)--(u1);
				\draw[taur,tauR2,decorate,decoration={simple line, raise=-2pt}] (u1)--(u3);
				\draw[nonribbon] (u3)--(z);
				\draw[edge] (z)--(u2);
				\draw[nonedge] (u2)--(u3);
				\draw[taur,tauL,decorate,decoration={simple line, raise=2pt}] (u1) -- (u3);
				\node[taurZero,tauR1] at (u1) {};
			\end{scope}

			\begin{scope}[xshift=6cm]
				\node[fvertex,label={below left:$u_2$}] (u2) at (0,0) {};
				\node[fvertex,label={below right:$u_1$}] (u1) at (1,0) {};
				\node[fvertex,label={above right:$u_3$}] (u3) at (1.5,1.5) {};
				\node[fvertex,label={above left:$z$}] (z) at (0.5,1.5) {};
				\draw[taur,tauR1] (u1)--(u2);
				\draw[nonribbon] (u1)--(u3);
				\draw[nonribbon] (z)--(u2);
				\draw[nonedge] (u2)--(u3);
				\draw[taur,tauR2] (z) -- (u3);
				\node[taurZero,tauL] at (u1) {};
			\end{scope}

			\begin{scope}[xshift=9cm]
				\node[fvertex,label={below left:$u_2$}] (u2) at (0,0) {};
				\node[fvertex,label={below right:$u_1$}] (u1) at (1,0) {};
				\node[fvertex,label={above right:$u_3$}] (u3) at (1.5,1.5) {};
				\node[fvertex,label={above left:$z$}] (z) at (0.5,1.5) {};
				\draw[nonribbon] (u2)--(u1);
				\draw[nonribbon] (u1)--(u3);
				\draw[nonribbon] (u3)--(z);
				\draw[nonribbon] (z)--(u2);
				\draw[nonedge] (u2)--(u3);
				\node[taurZero,tauL, minimum size=9pt] at (u1) {};
				\node[taurZero,tauR1] at (u1) {};
				\node[taurZero,tauR2] at (u2) {};
			\end{scope}
		\end{tikzpicture}
 		\caption{The illustration of \Cref{eq:TrInv}, i.e.,
 		${\color{tauL}\underline{{\color{black}\tau_r(u_1,u_3)}}}= {\color{tauR1}\underline{{\color{black}\tau_r(u_1,u_2)}}} + {\color{tauR2}\underline{{\color{black}\tau_r(u_2, u_3)}}}$,
 		if $u_2u_3$ (dotted) belongs to $\asc{G}$ but not to $E_G$
 		depending whether the edges of the corresponding 4-cycle are in $r$ (solid) or not (dashed).}
		\label{fig:taurNonedge23Proof}
	\end{figure}
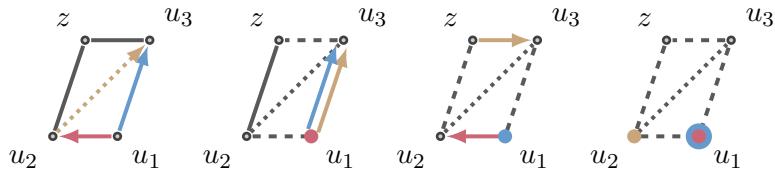
\end{proof}

Finally, the previous statements give equivalences of (infinitesimal) flexibility in terms of \APclasses{} and Cartesian NAC-colorings.
\begin{theorem}
    \label{thm:TPframeworks}
    Let $(G,\rho)$ be a \TPframework{}.
    The following statements are equivalent:
    \begin{enumerate}
        \item $(G,\rho)$ is flexible,
        \item $(G,\rho)$ is infinitesimally flexible,
        \item $G$ has at least two \APclasses{},
        \item $G$ has a Cartesian NAC-coloring.
    \end{enumerate}    
    If $G$ has $\ell$ \APclasses{}, where $2\leq \ell\in \NN$, then
    $G$ is a non-trivial subgraph of the Cartesian product of $\ell$ graphs. 
\end{theorem}
\begin{proof}
    The statement follows from \Cref{thm:main} and \Cref{lem:zeroSumAPclass}.
\end{proof}

\section{Rotationally symmetric \TPframeworks{}}\label{sec:sym}
This section is devoted to results on symmetric graphs, frameworks and flexes.
Symmetric flexibility of graphs has been considered in \cite{DGLrotSymmetry}.
We again relate the NAC-colorings (in this case symmetric ones) to symmetric \APclasses{}.

We start with the definition of the symmetric version from the notions of the previous sections.
\begin{definition}
    Let $G$ be a connected graph and $n\in \NN, n\geq 2$.
    Let the group $\Cn$ be a cyclic subgroup of order $n$ (generated by $\omega$) of the automorphism group of $G$.
    For $\gamma\in\Cn$, we define $\gamma v := \gamma(v)$ for $v\in V_G$ and
    $\gamma e := \gamma u \gamma v$ for $e=uv\in E_G$.
    We call a subset $S \subset V_G$ \emph{invariant} if $\gamma S=S$ for all $\gamma \in \Cn$,
    resp.\ \emph{partially invariant} if $\gamma S =S$ for some $\gamma \in \Cn, \gamma\neq 1$.
    A vertex $v\in V_G$ is called \emph{(partially) invariant}
    if $\{v\}$ is (partially) invariant.
    The graph $G$ is called \emph{\CnSymmetric}
    if the set of invariant vertices is an independent set in $G$
    and we have that if $v\in V_G$ is partially invariant,
    then $v$ is invariant.
\end{definition}

\begin{definition}
    A framework $(G,\rho)$ with $G$ being \CnSymmetric{}
    is called \emph{\CnSymmetric}
    if $\rho(\omega v) = \RotCn \rho(v)$ for each $v \in V_G$,
    where $\RotCn:=\Rot{\frac{2\pi}{n}}$ is the $\frac{2\pi}{n}$ rotation matrix.
    If there is a non-trivial flex $\rho_t$ of $(G,\rho)$
    such that each $(G,\rho_t)$ is \CnSymmetric,
    then $(G,\rho)$ is \emph{\CnSymmetric{} flexible},
    and \emph{\CnSymmetric{} rigid} otherwise.
\end{definition}

\begin{definition}
    An infinitesimal flex $\varphi$ of a \CnSymmetric{} framework $(G,\rho)$
    is \emph{\CnSymmetric{}} if for all $v\in V_G$:
    \[
        \varphi(\omega v) = \RotCn \varphi(v)\,.
    \]
    The framework $(G,\rho)$ is called \CnSymmetric{} infinitesimally flexible
    if it has a non-trivial \CnSymmetric{} infinitesimal flex.
\end{definition}

\begin{definition}
    Let $G$ be a \CnSymmetric{} graph.
    Consider the following relation on the edge set $E_G$:
    $e_1,e_2$ are in relation if and only if
    $e_1 (\Trel\cup \Prel) e_2$ (see \Cref{def:APclass})
    or $e_1=\omega^k e_2$ for some $k\in\NN_0$.
    An equivalence class of the transitive closure of the relation
    is called a \emph{\CnAPclass}.
\end{definition}

Again we can say something on the number of \CnAPclasses{} of a flexible framework.
\begin{proposition}
	\label{prop:CnFlexImpliesTwoCnAPclasses}
	If a $(G,\rho)$ is a \CnSymmetric{} flexible framework
	such that all induced 4-cycles are non-degenerate parallelograms,
	then $G$ has at least two \CnAPclasses{}.
\end{proposition}
\begin{proof}
    We proceed analogously to the proof of \Cref{prop:flexImpliesTwoAPclasses}
    with the following observation:
    the angle between the lines
    given by $e\in E_G$ and $\omega^k e$ is $\frac{2k\pi}{n}$
    in any \CnSymmetric{} realization, namely, it is constant along any \CnSymmetric{} flex.
\end{proof}

Similarly we can show this for infinitesimal flexibility.
\begin{proposition}
	\label{prop:CnInfFlexImpliesTwoCnAPclasses}
	If a $(G,\rho)$ is a \CnSymmetric{} infinitesimally flexible framework such that all induced 4-cycles are
	non-degenerate parallelograms and all 3-cycles are non-degenerate triangles, then $G$ has at least two \CnAPclasses{}.
\end{proposition}
\begin{proof}
    We proceed as in the proof of \Cref{prop:infFlexImpliesTwoAPclasses},
    we only need to prove extra that $e \infConstAngle\, \omega e$ for every $e \in E_G$.
    Recall that
    \begin{align*}
        uv \infConstAngle\, \omega u \omega v \iff
            &\left(\realizVec{u}{v}, \realizVec{\omega u}{\omega v} \text{ are LI} \land
                \exists \alpha\in\RR:
                    \flexVec{u}{v} = \alpha \piHalfRotation \realizVec{u}{v}
                    \land
                    \flexVec{\omega u}{\omega v} = \alpha \piHalfRotation \realizVec{\omega u}{\omega v}
            \right)\\
            &\lor \left(
                \exists \beta\in\RR\setminus\{0\}:
                    \realizVec{u}{v} = \beta \realizVec{\omega u}{\omega v}
                    \land
                    \flexVec{u}{v} = \beta \flexVec{\omega u}{\omega v}
            \right).
    \end{align*}
    By the definition of a \CnSymmetric{} realization/infinitesimal flex,
    we have
    \begin{align*}
        \realizVec{\omega u}{\omega v} = \RotCn \realizVec{u}{v} \quad\text{ and }\quad
        \flexVec{\omega u}{\omega v} = \RotCn \flexVec{u}{v}\,.
    \end{align*}
    If $\realizVec{u}{v}, \realizVec{\omega u}{\omega v}$ are linearly independent,
    $uv \infConstAngle\, \omega u\, \omega v$ follows by a direct computation.
    Otherwise, we have necessarily $\beta=-1$ and $n=2$ and the statement follows as well.
\end{proof}
\begin{remark}
    In the following we use the walk-independence just as defined before,
    namely, the sum in \Cref{eq:zeroSumAPclass} is over \APclasses{},
    not \CnAPclasses{}.
\end{remark}

Using the previous definitions and results we get a symmetric counterpart of \Cref{prop:twoAPclassesImplyFlex} on the relation between \CnAPclasses{} and flexes.
\begin{proposition}
	\label{prop:twoCnAPclassesImplyCnFlex}
	Let $(G,\rho)$ be a \CnSymmetric{} walk-independent framework such that $G$ has $\ell\in \NN \cup \{\infty\}$ \CnAPclasses{}.
	If $\ell\geq 2$, then the framework is \CnSymmetric{} (infinitesimally) flexible.
	In particular, if $\ell \in \NN$ and  $\ell\geq 2$, there are $\ell-1$ independent ways how it can continuously flex.
\end{proposition}
\begin{proof}
    First, we focus on continuous flexibility.
	Let $\sst{r_i}{i \in I}$ be the \CnAPclasses{} of $G$,
	where $I=\{0, \ldots, \ell-1\}$ or $\NN$
	depending whether there are finitely or infinitely many of them.
	We define a map $\ttt\mapsto \rho_\ttt$,
	where $\ttt=(t_i)_{i\in I}$ with $t_0=0$ and $t_i\in [0,2\pi)$ for $i\neq 0$, by
	\begin{equation*}
		\rho_\ttt(v) =  \sum_{i\in I} \Rot{t_i} \CnRibbonPart{i}(v)\,,
	\end{equation*}
	where
	\begin{equation*}
		\CnRibbonPart{i}(v) = \sum_{\oriented{w_1}{w_2}\in W \cap r_i} (\rho(w_2)-\rho(w_1))
	\end{equation*}
	with $W$ being any walk in $G$ from a fixed vertex $\bar{u}$ to $v$.
	By construction a \CnAPclass{} $r_i$ is a union of \APclasses{} $S_i=\{r_{i,1},\ldots,r_{i,k_i}\}$.
	Therefore, we check that~$\rho_\ttt(v)$ is defined as in the proof of
	\Cref{prop:twoAPclassesImplyFlex} just that \APclasses{} from the same \CnAPclass{}
	are rotated by the same angle:
	\begin{align*}
	    \rho_\ttt(v)
	        &= \sum_{i\in I} \Rot{t_i} \sum_{\oriented{w_1}{w_2}\in W \cap r_i} (\rho(w_2)-\rho(w_1))\\
	        &= \sum_{i\in I}\sum_{r\in S_i} \Rot{t_i} \sum_{\oriented{w_1}{w_2}\in W \cap r} (\rho(w_2)-\rho(w_1))\,.
	\end{align*}
	In particular, it is well-defined and the edge lengths are constant.
	In comparison to \Cref{prop:twoAPclassesImplyFlex}, we cannot assume that
	$\rho(\bar{u})$ is in the origin since $\rho$ is \CnSymmetric{}.
	Instead
	\begin{align*}
		\rho_{\zerovec}(v) = \rho(v) - \rho(\bar{u}).
	\end{align*}
	We adjust the map using the idea from \cite[Lemma 7.4]{DLinfinite}
	to make it \CnSymmetric{} and starting at $\rho$
    by translating the center of gravity of the orbit of $\bar{u}$ to the origin.
	Let
	\begin{equation*}
	    \tilde{\rho}_\ttt(v) :=
	        \rho_\ttt(v) - \frac{1}{n}\sum_{j=0}^{n-1}\rho_\ttt (\omega^j\bar{u})\,.
	\end{equation*}
	Note that if $\bar{u}$ is an invariant vertex, then $\rho(\bar{u})=(0,0)$
	and thus $\tilde{\rho}_\ttt = \rho_\ttt$.
	The map $\tilde{\rho}_\ttt$ is a (multi-parametric) flex of $(G,\rho)$:
	\begin{equation*}
	    \tilde{\rho}_{\zerovec}(v)
	    = \rho(v) - \rho(\bar{u}) - \frac{1}{n}\sum_{j=0}^{n-1}\left(\rho(\omega^j\bar{u}) - \rho(\bar{u})\right)
	    = \rho(v) - \underbrace{\frac{1}{n}\sum_{j=0}^{n-1}\RotCn^j\rho(\bar{u})}_{=(0,0) \text{ by symmetry}} = \rho(v)\,.
	\end{equation*}
	Let $W$ be walk from $\bar{u}$ to $\omega \bar{u}$.
	Concatenating $W,\omega W, \ldots, \omega^{j-1}W$ yields
	a walk $W_j$ from $\bar{u}$ to~$\omega^j \bar{u}$.
	Hence,
	\begin{align}\label{eq:rhoiomegajubar}
	    \CnRibbonPart{i}(\omega^j \bar{u})
	        &= \sum_{\oriented{w_1}{w_2}\in W_j \cap r_i} (\rho(w_2)-\rho(w_1))
	        = \sum_{k=0}^{j-1}\sum_{\oriented{w_1}{w_2}\in \omega^k W \cap r_i} (\rho(w_2)-\rho(w_1)) \nonumber\\
	        &= \sum_{k=0}^{j-1}\sum_{\oriented{w_1}{w_2}\in W \cap r_i} \left(\rho(\omega^k w_2)-\rho(\omega^k w_1)\right)\,.
	\end{align}
	Notice that in the last equality we use that $r_i$ is invariant under $\omega$.
	Now
	\begin{align}
	    \label{eq:auxCnIdentity}
	    \sum_{j=0}^{n-1}&\left(\CnRibbonPart{i}(\omega^j \bar{u}) - \RotCn\CnRibbonPart{i}(\omega^j \bar{u})\right) \nonumber \\
	    &\eqwithreference{\eqref{eq:rhoiomegajubar}} \sum_{j=0}^{n-1} \sum_{\oriented{w_1}{w_2}\in W \cap r_i}\sum_{k=0}^{j-1}
	         \left(\rho(\omega^k w_2)-\rho(\omega^k w_1)
	            - \rho(\omega^{k+1} w_2)+\rho(\omega^{k+1} w_1)\right) \nonumber \\
	    &= \sum_{\oriented{w_1}{w_2}\in W \cap r_i}\sum_{j=0}^{n-1}
	         \left(\rho(w_2)-\rho(w_1)
	            - \rho(\omega^{j-1+1} w_2)+\rho(\omega^{j-1+1} w_1)\right) \nonumber \\
	    &= \sum_{\oriented{w_1}{w_2}\in W \cap r_i} \Bigg(
	         n\left(\rho(w_2)-\rho(w_1)\right)
	         - \underbrace{\sum_{j=0}^{n-1}
	                \left(\rho(\omega^{j} w_2)-\rho(\omega^{j} w_1)\right)
	                }_{=(0,0) \text{ by symmetry}}\Bigg)
	   \eqwithreference{def.\ $\rho_i$} n \CnRibbonPart{i}(\omega \bar{u})\,.
	\end{align}
	We fix a vertex $v$.
	If $W$ is a walk from $\bar{u}$ to $\omega\bar{u}$
	and $W'$ a walk from $\bar{u}$ to $v$, then
	\begin{align}\label{eq:rhoiomegav}
	    \CnRibbonPart{i}(\omega v) &= \sum_{\oriented{w_1}{w_2}\in W \cap r_i} (\rho(w_2)-\rho(w_1))
	        + \sum_{\oriented{w_1}{w_2}\in \omega W' \cap r_i} (\rho(w_2)-\rho(w_1)) \nonumber\\
	    &= \CnRibbonPart{i}(\omega \bar{u})
	   + \sum_{\oriented{w_1}{w_2}\in W' \cap r_i} (\rho(\omega w_2)-\rho(\omega w_1))
	    = \CnRibbonPart{i}(\omega \bar{u}) + \RotCn \CnRibbonPart{i}(v)\,.
	\end{align}
	Finally,
	\begin{align*}
	    \tilde{\rho}_\ttt (\omega v)
	        &= \rho_\ttt(\omega v) - \frac{1}{n}\sum_{j=0}^{n-1}\rho_\ttt (\omega^j\bar{u})
	        =  \sum_{i\in I} \Rot{t_i} \left(\CnRibbonPart{i}(\omega v) - \frac{1}{n}\sum_{j=0}^{n-1} \CnRibbonPart{i}(\omega^j\bar{u})\right) \\
	        &\eqwithreference{\eqref{eq:rhoiomegav}} \sum_{i\in I} \Rot{t_i} \left(
	            \CnRibbonPart{i}(\omega \bar{u})
	            + \RotCn \CnRibbonPart{i}(v)
	            - \frac{1}{n}\sum_{j=0}^{n-1} \CnRibbonPart{i}(\omega^j\bar{u})
	            \right) \\
	        &\eqwithreference{\eqref{eq:auxCnIdentity}}
	            \sum_{i\in I} \Rot{t_i} \left(
	            \RotCn \CnRibbonPart{i}(v)
	            - \frac{1}{n}\RotCn\sum_{j=0}^{n-1} \CnRibbonPart{i}(\omega^j \bar{u})
	            \right) \\
	        & = \RotCn \sum_{i\in I} \Rot{t_i} \left(
	            \CnRibbonPart{i} (v)
	            - \frac{1}{n}\sum_{j=0}^{n-1} \CnRibbonPart{i}(\omega^j \bar{u})
	            \right) \\
	        & = \RotCn \left(
	            \rho_\ttt(v)
	            - \frac{1}{n}\sum_{j=0}^{n-1} \rho_\ttt(\omega^j \bar{u})
	            \right)
	        = \RotCn \tilde{\rho}_\ttt(v)\,.
	\end{align*}

	Hence, for instance $t \mapsto \rho_{(0,t,0,\ldots,0)}$, or respectively $t \mapsto \rho_{(0,t,0,\ldots)}$ in the case of infinitely many \APclasses,
	gives a \CnSymmetric{} flex $\tilde{\rho}_t$ of $(G,\rho)$
	which is non-trivial since the edges in $r_1$ change the angle
	with the edges in other \APclasses{}.
    The continuous flex $\tilde{\rho}_t$
    yields an infinitesimal one: for $v\in V_G$ we set
    \begin{align*}
        \varphi(v)&:=\frac{\partial \tilde{\rho}_t(v)}{\partial t}\Big|_{t=0}\,.
    \end{align*}
    We remark that this is equal to
    \[
        \piHalfRotationMatrix\cdot \left(\sum_{\oriented{w_1}{w_2}\in W \cap r_1} (\rho(w_2)-\rho(w_1))
        - \frac{1}{n}\sum_{j=0}^{n-1}\sum_{\oriented{w_1}{w_2}\in W_j \cap r_1} (\rho(w_2)-\rho(w_1))\right)\,,
    \]
    where $W$ is a walk from $\bar{u}$ to $v$
    and $W_j$ from $\bar{u}$ to $\omega^j\bar{u}$.
    If $uv\in r_1$, then
    \[
        \varphi(u)-\varphi(v)
            =\piHalfRotationMatrix \left(\rho(u)-\rho(v)\right)\neq(0,0)\,,
    \]
    which is orthogonal to $\rho(u)-\rho(v)$.
    For $uv\notin r_1$ (such edges exist by the assumption that there are at least two \APclasses{}),
    we have $\varphi(u)-\varphi(v)=(0,0)$.
    Hence, the infinitesimal flex $\varphi$ is not induced by a rigid motion.
    The infinitesimal flex is also \CnSymmetric{}:
    \begin{equation*}
        \varphi(\omega v)=\frac{\partial \tilde{\rho}_t(\omega v)}{\partial t}\Big|_{t=0}
            = \frac{\partial \RotCn \tilde{\rho}_t(v)}{\partial t}\Big|_{t=0}
            = \RotCn \frac{\partial \tilde{\rho}_t(v)}{\partial t}\Big|_{t=0}
            = \RotCn \varphi(v)\,.
    \end{equation*}
\end{proof}

In \Cref{fig:ex:symflex} we see an example of a graph with three \CnAPclasses{} and how it can flex.

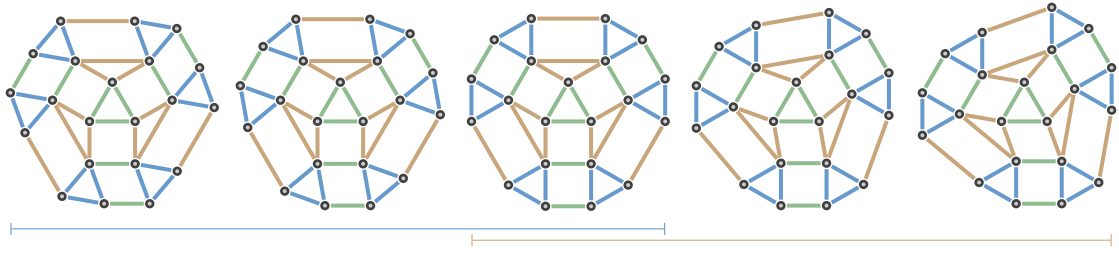
\begin{figure}[ht]
    \centering
    \begin{tikzpicture}[scale=0.75]
        \foreach \a/\b [count=\i from 0] in {0/20,0/10,0/0,10/0,20/0}
        {
            \begin{scope}[xshift=4*\i cm]
            \node[fvertex] (1) at (0.8,0) {};
            \node[fvertex] (2) at (60:0.8) {};
            \node[fvertex] (3) at (0,0) {};
            \node[fvertex] (4) at ($(1)+(30+\a:0.75)$) {};
            \node[fvertex] (5) at ($(2)+(4)-(1)$) {};
            \node[fvertex] (6) at ($(2)+(150+\a:0.75)$) {};
            \node[fvertex] (7) at ($(3)+(6)-(2)$) {};
            \node[fvertex] (8) at ($(3)+(270+\a:0.75)$) {};
            \node[fvertex] (9) at ($(1)+(8)-(3)$) {};
            \node[fvertex] (10) at ($(4)+(30+\b:0.75)$) {};
            \coordinate (r\i) at (10);
            \node[fvertex] (11) at ($(5)+(10)-(4)$) {};
            \node[fvertex] (12) at ($(5)+(90+\b:0.75)$) {};
            \node[fvertex] (13) at ($(6)+(12)-(5)$) {};
            \node[fvertex] (14) at ($(6)+(150+\b:0.75)$) {};
            \node[fvertex] (15) at ($(7)+(14)-(6)$) {};
            \coordinate (l\i) at (15);
            \node[fvertex] (16) at ($(7)+(210+\b:0.75)$) {};
            \node[fvertex] (17) at ($(8)+(16)-(7)$) {};
            \node[fvertex] (18) at ($(8)+(270+\b:0.75)$) {};
            \node[fvertex] (19) at ($(9)+(18)-(8)$) {};
            \node[fvertex] (20) at ($(9)+(330+\b:0.75)$) {};
            \node[fvertex] (21) at ($(4)+(20)-(9)$) {};

            \draw[edge,col1] (1)edge(2) (2)edge(3) (3)edge(1) (4)edge(5) (6)edge(7) (8)edge(9);
            \draw[edge,col5] (1)edge(4) (1)edge(9) (2)edge(5) (2)edge(6) (3)edge(7) (3)edge(8) (5)edge(6) (7)edge(8) (9)edge(4);
            \draw[edge,col1] (10)edge(11) (14)edge(15) (18)edge(19);
            \draw[edge,col5] (12)edge(13) (16)edge(17) (20)edge(21);
            \draw[edge,col4] (11)edge(12) (13)edge(14) (15)edge(16) (17)edge(18) (19)edge(20) (21)edge(10);
            \draw[edge,col4] (4)edge(10) (5)edge(11) (5)edge(12) (6)edge(13) (6)edge(14) (7)edge(15) (7)edge(16) (8)edge(17) (8)edge(18) (9)edge(19) (9)edge(20) (4)edge(21);
            \end{scope}
        }
        \draw[col4,|-|] ($(0,-1.9)!(r2)!(1,-1.9)$)--($(0,-1.9)!(l0)!(1,-1.9)$);
        \draw[col5,|-|] ($(0,-2.1)!(l2)!(1,-2.1)$)--($(0,-2.1)!(r4)!(1,-2.1)$);
    \end{tikzpicture}
    \caption{A symmetric framework with two independent ways of flexing. In the first half the \CnAPclass{}  \protect\tikz{\protect\draw[edge,col4](0,0)--(0.5,0)} causes the flex. In the second part the \CnAPclass{} \protect\tikz{\protect\draw[edge,col5](0,0)--(0.5,0)} does.}
    \label{fig:ex:symflex}
\end{figure}

NAC-colorings may preserve symmetry, which is what we need now.
\begin{definition}
    Let $G$ be a \CnSymmetric{} graph.
	A NAC-coloring $\delta$ of $G$ is called \CnSymmetric{},
	if $\delta(\omega e)=\delta(e)$ for all $e\in E_G$
	and there is no edge connecting any
	two partially invariant connected components of
    the subgraphs
	\begin{align*}
		G_\red^\delta = (V_G, \{e\in E_G \colon \delta(e) = \red\})\, \text{ and }
		G_\blue^\delta = (V_G, \{e\in E_G \colon \delta(e) = \blue\})\,.
	\end{align*}
\end{definition}
The definition above follows~\cite{DGLrotSymmetry},
where it is shown that a \CnSymmetric{} graph has a \CnSymmetric{} flexible realization
if and only if it has a \CnSymmetric{} NAC-coloring.
The condition on invariant components not being connected
might seem to be superfluous for the following lemma about \CnSymmetric{} walk-independent frameworks,
but it is crucial in the construction of a flexible realization
from a \CnSymmetric{} NAC-coloring for a general \CnSymmetric{} graph
in the result from \cite{DGLrotSymmetry} mentioned above.

Now a Cartesian \CnSymmetric{} NAC-coloring can be determined by monochromatic \CnAPclasses.
\begin{lemma}
    \label{lem:CnMonochromaticAPclassIffCartesian}
    Let $(G,\rho)$ be a \CnSymmetric{} walk-independent framework.
    An edge coloring of $G$ by \red{} and \blue{}
    is a Cartesian \CnSymmetric{} NAC-coloring of $G$ if and only if
    it is surjective and the \CnAPclasses{} of $G$ are monochromatic.
\end{lemma}
\begin{proof}
    Assume that $\delta$ is a Cartesian \CnSymmetric{} NAC-coloring of $G$.
    By \Cref{lem:monochromaticAPclassIffCartesian}, \APclasses{} are monochromatic.
    Since a \CnAPclass{} is the union of \APclasses{} in the same orbit,
    it is monochromatic by the invariance of $\delta$.

    On the other hand, assume \CnAPclasses{} are monochromatic.
    Then also \APclasses{} are monochromatic and
    hence $\delta$ is a Cartesian NAC-coloring by \Cref{lem:monochromaticAPclassIffCartesian}.
    The invariance of $\delta$ under the action of $\omega$
    is implied by the invariance of \CnAPclasses{}.
    We are only left to prove that there are no two partially invariant connected
    components of $G_\red^\delta$ or $G_\blue^\delta$ connected by an edge.
    The idea of the following comes from \cite[Lemma~1]{DGLrotSymmetry}.
    By surjectivity of $\delta$, there are at least two \CnAPclasses{}.
    Hence, we can construct a flex
    \begin{equation*}
	    \tilde{\rho}_t(v) =
	        \rho_t(v) - \frac{1}{n}\sum_{j=0}^{n-1}\rho_t (\omega^j\bar{u})
	\end{equation*}
	as in the proof of \Cref{prop:twoCnAPclassesImplyCnFlex},
	where for $\ttt=(t_i)_{i\in I}$
	we set $t_i=0$ for all \blue{} \CnAPclasses{} and
	$t_i=t$ for all \red{} \CnAPclasses{}.
	For $uv\in E_G$,
	we have (see also the proof of \Cref{prop:twoAPclassesImplyFlex})
	\begin{align}
	    \label{eq:redBlueCnFlex}
		\tilde{\rho}_t(u) - \tilde{\rho}_t(v) =
		\rho_t(u) - \rho_t(v) =
		    \begin{cases}
                \Rot{t}(\rho(u) - \rho(v)) & \text{if } \delta(uv)=\red{}  \\
                \rho(u) - \rho(v)  & \text{if } \delta(uv)=\blue{}\,.
		    \end{cases}
	\end{align}
	Suppose that $V_1, V_2$ are the vertex sets
	of two distinct partially invariant connected components of $G_\red^\delta$.
	We show that $v_1v_2\notin E_G$ for any $v_i\in V_i$.
	There are $1<k_i<n$ such that $\omega^{k_i} v_i\in V_i$
	and hence also \red{} paths $W_i=(v_i=u_0^i, \ldots, u_{s_i}^i=\omega^{k_i}v_i)$
	from $v_i$ to~$\omega^{k_i}v_i$.
	We consider
	\begin{align*}
	    (\underbrace{\RotCn^{k_i} - E}_{T_i})\tilde{\rho}_t(v_i)
	      &= \tilde{\rho}_t(\omega^{k_i}v_i) - \tilde{\rho}_t(v_i)
	      = \sum_{j=0}^{s_i-1} \left(\tilde{\rho}_t(u_{j+1}^i) - \tilde{\rho}_t(u_j^i)\right)\\
	      &\eqwithreference{\eqref{eq:redBlueCnFlex}} \sum_{j=0}^{s_i-1} \Rot{t}\left(\rho(u_{j+1}^i) - \rho(u_j^i)\right)
	      = \Rot{t}\left(\rho(\omega^{k_i}v_i) - \rho(v_i)\right)\,.
	\end{align*}
	Since $T_i$ is invertible and commutes with $\Rot{t}$,
	\begin{align*}
	    \tilde{\rho}_t(v_1)-\tilde{\rho}_t(v_2)
	        &= T_1^{-1}\Rot{t}\left(\rho(\omega^{k_1}v_1) - \rho(v_1)\right)
	         - T_2^{-1}\Rot{t}\left(\rho(\omega^{k_2}v_2) - \rho(v_2)\right) \\
	        &= \Rot{t}\left( T_1^{-1}\left(\rho(\omega^{k_1}v_1) - \rho(v_1)\right)
	         - T_2^{-1}\left(\rho(\omega^{k_2}v_2) - \rho(v_2)\right)\right)\,.
	\end{align*}
	If the vector rotated by $\Rot{t}$ is zero,
	then $\tilde{\rho}_t(v_1)=\tilde{\rho}_t(v_2)$ for all $t$,
	in particular $\rho_0(v_1)=\tilde{\rho}_0(v_1)=\tilde{\rho}_0(v_2)=\rho_0(v_2)$,
	which contradicts injectivity of $\rho$.
	Therefore, $\tilde{\rho}_t(v_1)-\tilde{\rho}_t(v_2)$ is non-constant.
	The edge $v_1v_2$ cannot exist,
	since it would be \blue{} which would contradict \Cref{eq:redBlueCnFlex}.
	In the case of two partially invariant connected components of $G_\blue^\delta$,
	by analogous computation we get that $\tilde{\rho}_t(v_1)-\tilde{\rho}_t(v_2)$
	is constant, contradicting \Cref{eq:redBlueCnFlex} again.
\end{proof}

This yields the symmetric counter part of \Cref{thm:TPframeworks}.
\begin{theorem}
    \label{thm:CnTPframeworks}
    Let $(G,\rho)$ be a \CnSymmetric{} walk-independent framework.
    The following statements are equivalent:
    \begin{enumerate}
        \item\label{it:CnFlex} $(G,\rho)$ is \CnSymmetric{} flexible,
        \item\label{it:CnTwoAPclasses} $G$ has at least two \CnAPclasses{},
        \item\label{it:CnCartNAC} $G$ has a Cartesian \CnSymmetric{} NAC-coloring.
    \end{enumerate}
    If in addition all 3-cycles in $G$ form non-degenerate triangles in $\rho$,
    then the above statements are equivalent to $(G,\rho)$
    being infinitesimally \CnSymmetric{} flexible.
    In particular, the statement holds for \CnSymmetric{}
    \Pframeworks{} and \TPframeworks{}.
\end{theorem}
\begin{proof}
    The equivalence of \ref{it:CnFlex} and \ref{it:CnTwoAPclasses}
    follows from \Cref{prop:CnFlexImpliesTwoCnAPclasses,prop:twoCnAPclassesImplyCnFlex}.
    \Cref{lem:CnMonochromaticAPclassIffCartesian} gives the equivalence
    of \ref{it:CnTwoAPclasses} and \ref{it:CnCartNAC}.
    The equivalence with infinitesimal flexibility is shown
    in \Cref{prop:CnInfFlexImpliesTwoCnAPclasses,prop:twoCnAPclassesImplyCnFlex}.
    \CnSymmetric{} \Pframeworks{} and \TPframeworks{} are walk-independent
    by \Cref{lem:zeroSumPframework,lem:zeroSumAPclass}.
\end{proof}

\section{Application to tessellations}\label{sec:tess}
As an illustration, we conclude the paper applying the theory developed above to a few infinite frameworks
constructed as 1-skeleta of tilings of the plane by regular polygons.
Obviously, the regular tiling by triangles is rigid, whereas the tiling by hexagons, resp.\ squares, is flexible.
More interesting examples are obtained by allowing more regular polygons to be used simultaneously.
We remark that polygons are regular only in the starting realization as flexing might destroy regularity.

Here we restrict to periodic, edge-to-edge tilings. Such a tiling is called \emph{$k$-uniform} if there are $k$
vertex orbits under the symmetry of the tiling.
A possible notation is listing for each vertex orbit the number of edges of the tiles to which the vertex is incident \cite{CundyRollett}.
For instance $[3^6;3^2.4.3.4;3^2.4.3.4]$ denotes that there are three orbits:
vertices of the first one are incident to six triangles,
while vertices of the other two orbits are incident to two triangles, a square, a triangle and another square (see \Cref{fig:333333-33434-33434}).
The following examples can also be found in the online library
Antwerp v3.0\footnote{\url{https://antwerp.hogg.io/}} described
in~\cite{GomJau-Hogg}.

When there are regular hexagons, we follow the approach of~\cite{Nagy2006}, namely,
we force opposite edges to be parallel.
This can be achieved by adding parallelograms, see \Cref{fig:hexagon}.
We do not display the added edges in the figures.
In the examples where we discuss possible \CnSymmetric{} flexes,
the parallelograms can be added preserving the symmetry.

Now pictures of some tessellations follow:
\Cref{fig:3636} is 1-uniform,
\Cref{fig:3464-33434} is 2-uniform and
\Cref{fig:333333-33434-33434,fig:33434-3464-3446,fig:3366-3636-666,fig:infinitelyMany} are 3-uniform.
Some flexes exhibit \emph{auxetic} behavior, compare~\cite{Mitschke2013,Mitschke2013a}.

\tikzexternalenable

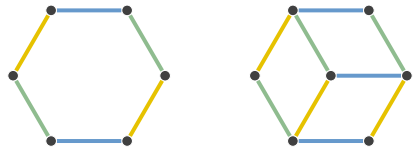
\begin{figure}[ht]
    \centering
    \tikzsetnextfilename{hexagon}
    \begin{tikzpicture}
        \foreach \w [count=\i] in {0,60,...,300}
        {
            \node[gvertex] (a\i) at (\w:1) {};
        }
        \draw[edge,col1] (a1)edge(a2) (a4)edge(a5);
        \draw[edge,col4] (a2)edge(a3) (a5)edge(a6);
        \draw[edge,col3] (a3)edge(a4) (a6)edge(a1);
    \end{tikzpicture}
    \qquad
    \tikzsetnextfilename{hexagonWithParallelograms}
    \begin{tikzpicture}
        \foreach \w [count=\i] in {0,60,...,300}
        {
            \node[gvertex] (a\i) at (\w:1) {};
        }
        \node[gvertex] (a0) at (0,0) {};
        \draw[edge,col1] (a1)edge(a2) (a4)edge(a5) (a3)edge(a0);
        \draw[edge,col4] (a2)edge(a3) (a5)edge(a6) (a1)edge(a0);
        \draw[edge,col3] (a3)edge(a4) (a6)edge(a1) (a5)edge(a0);
    \end{tikzpicture}
    \caption{A hexagon with opposite edges in the same \APclass{} and how it can be achieved by adding parallelograms.}
    \label{fig:hexagon}
\end{figure}

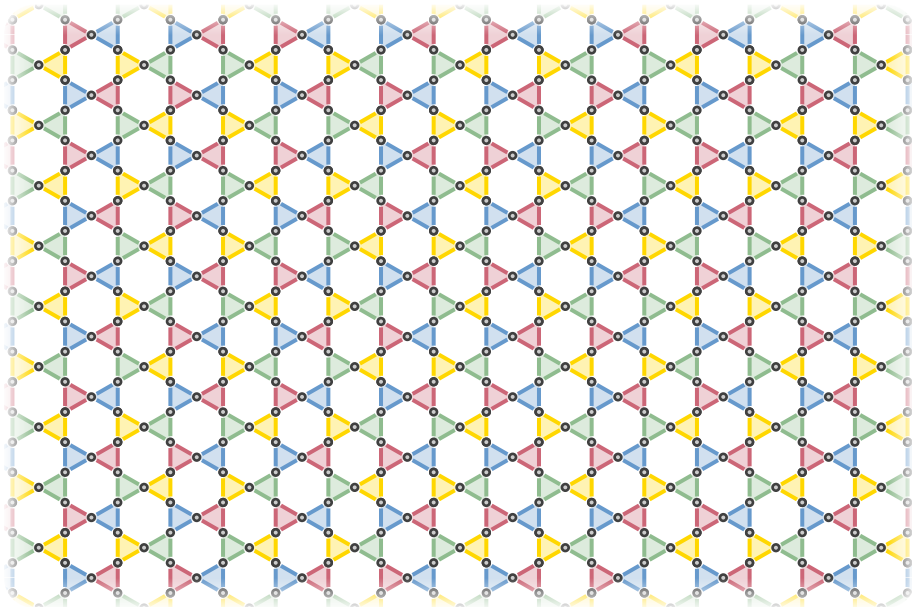
\begin{figure}[H]
    \centering
    \tikzsetnextfilename{3636}
    \begin{tikzpicture}[scale=0.4]
            \begin{pgfonlayer}{background}
                \clip[rounded corners=0.4cm] (-15,10)rectangle ++(30,-20);
            \end{pgfonlayer}

            \begin{scope}
            \clip[rounded corners=0.4cm] (-15,10)rectangle ++(30,-20);
            \begin{scope}
            \foreach \y [evaluate=\y as \xs using {Mod(\y,2)/2-2.25},evaluate=\xs as \xse using \xs+4] in {-5,...,5}
            {
                \foreach \x in {\xs,...,\xse}
                {
                    \coordinate (no) at ($\x*4*(30:1)+\x*4*(-30:1)+\y*2*(90:1)$);
                    \begin{scope}[shift={(no)}]
                        \foreach \r [evaluate=\r as \ra using 60*\r-30,evaluate=\r as \rb using 60*\r+30] in {1,2,...,6}
                        {
                            \node[fvertex] (a\r) at (\ra:1) {};
                            \node[fvertex] (b\r) at ($(a\r)+(\rb:1)$) {};
                        }
                        \draw[edge,colR] (a1)edge(a2) (a1)edge(b1) (b1)edge(a2) (a4)edge(a5) (a4)edge(b4) (b4)edge(a5);
                        \draw[edge,colB] (a2)edge(a3) (a2)edge(b2) (b2)edge(a3) (a5)edge(a6) (a5)edge(b5) (b5)edge(a6);
                        \draw[edge,colG] (a3)edge(a4) (a3)edge(b3) (b3)edge(a4) (a6)edge(a1) (a6)edge(b6) (b6)edge(a1);
                        \node[fvertex] (c6) at ($(b6)+(30:1)$) {};
                        \node[fvertex] (c6s) at ($(b6)+(-30:1)$) {};
                        \node[fvertex] (c3) at ($(b3)+(180+30:1)$) {};
                        \node[fvertex] (c3s) at ($(b3)+(180-30:1)$) {};
                        \draw[edge,colY] (b6)edge(c6) (c6)edge(c6s) (c6s)edge(b6);
                        \draw[edge,colY] (b3)edge(c3) (c3)edge(c3s) (c3s)edge(b3);
                        \begin{pgfonlayer}{background}
                            \fill[colR,face] (a1.center)--(b1.center)--(a2.center)--cycle;
                            \fill[colB,face] (a2.center)--(b2.center)--(a3.center)--cycle;
                            \fill[colG,face] (a3.center)--(b3.center)--(a4.center)--cycle;
                            \fill[colR,face] (a4.center)--(b4.center)--(a5.center)--cycle;
                            \fill[colB,face] (a5.center)--(b5.center)--(a6.center)--cycle;
                            \fill[colG,face] (a6.center)--(b6.center)--(a1.center)--cycle;

                            \fill[colY,face] (b3.center)--(c3.center)--(c3s.center)--cycle;
                            \fill[colY,face] (b6.center)--(c6.center)--(c6s.center)--cycle;
                        \end{pgfonlayer}

                    \end{scope}
                }
            }
            \end{scope}
            \end{scope}
            \fill[white,path fading=east] (-15.1,9)--(-14,9)--(-14,-9)--(-15.1,-9)--cycle;
            \fill[white,path fading=west] (15.1,9)--(14,9)--(14,-9)--(15.1,-9)--cycle;
            \fill[white,path fading=south] (-14,10.1)--(-14,9)--(14,9)--(14,10.1)--cycle;
            \fill[white,path fading=north] (-14,-10.1)--(-14,-9)--(14,-9)--(14,-10.1)--cycle;
            \begin{scope}
                \clip (13.99,9.003)rectangle(15.1,10.1);
                \fill[white,path fading=fade out] (14,9) circle[radius=1.1cm];
            \end{scope}
            \begin{scope}
                \clip (13.99,-9)rectangle(15.1,-10.1);
                \fill[white,path fading=fade out] (14,-9) circle[radius=1.1cm];
            \end{scope}
            \begin{scope}
                \clip (-14.004,9.003)rectangle(-15.1,10.1);
                \fill[white,path fading=fade out] (-14,9) circle[radius=1.1cm];
            \end{scope}
            \begin{scope}
                \clip (-14.004,-9)rectangle(-15.1,-10.1);
                \fill[white,path fading=fade out] (-14,-9) circle[radius=1.1cm];
            \end{scope}
        \end{tikzpicture}
    \caption{Tiling $[3636]$ has 4 \APclasses{}.}
    \label{fig:3636}
\end{figure}

\begin{figure}[H]
    \centering
    \tikzsetnextfilename{333333-33434-33434}
    \begin{tikzpicture}[scale=0.4]
            \newcommand{\w}{0}
            \begin{pgfonlayer}{background}
                \clip[rounded corners=0.4cm] (-15,10)rectangle ++(30,-20);
            \end{pgfonlayer}

            \begin{scope}
            \clip[rounded corners=0.4cm] (-15,10)rectangle ++(30,-20);
            \begin{scope}
            \foreach \y [evaluate=\y as \xs using {Mod(\y,2)/2-3},evaluate=\xs as \xse using \xs+5] in {-6,...,6}
            {
                \foreach \x in {\xs,...,\xse}
                {
                    \begin{scope}[shift={($\y*(90:1)+\y*0.5*(\w+60:1)+\y*0.5*(\w+120:1)+\x*2*(30:1)+\x*2*(-30:1)+\x*3*(\w:1)$)}]
                        \node[fvertex] (a0) at (0,0) {};
                        \foreach \r [evaluate=\r as \rr using 60*\r-30] in {1,2,...,6}
                        {
                            \node[fvertex] (a\r) at (\rr:1) {};
                            \draw[edge,colB]  (a0)edge(a\r);
                        }
                        \draw[edge,colB] (a1)edge(a2) (a2)edge(a3) (a3)edge(a4) (a4)edge(a5) (a5)edge(a6) (a6)edge(a1);
                        \node[fvertex] (b1) at ($(a1)+(\w:1)$) {};
                        \node[fvertex] (b2) at ($(a1)+(\w+60:1)$) {};
                        \node[fvertex] (b3) at ($(b1)+(b2)-(a1)$) {};
                        \draw[edge,colR] (a1)edge(b1) (a1)edge(b2) (b1)edge(b2) (b1)edge(b3) (b2)edge(b3);
                        \node[fvertex] (c1) at ($(a6)+(\w:1)$) {};
                        \node[fvertex] (c2) at ($(a6)+(\w-60:1)$) {};
                        \node[fvertex] (c3) at ($(c1)+(c2)-(a6)$) {};
                        \draw[edge,colR] (a6)edge(c1) (a6)edge(c2) (c1)edge(c2) (c1)edge(c3) (c2)edge(c3);
                        \node[fvertex] (d1) at ($(a5)+(\w-60:1)$) {};
                        \node[fvertex] (d2) at ($(a5)+(\w-120:1)$) {};
                        \node[fvertex] (d3) at ($(d1)+(d2)-(a5)$) {};
                        \draw[edge,colR] (a5)edge(d1) (a5)edge(d2) (d1)edge(d2) (d1)edge(d3) (d2)edge(d3);

                        \node[fvertex] (d2s) at ($(c1)+(30:1)$) {};
                        \draw[edge,colB] (b1)edge(c1) (b1)edge(d2s) (c1)edge(d2s);
                        \node[fvertex] (b2s) at ($(d1)+(-30:1)$) {};
                        \draw[edge,colB] (d1)edge(c2) (d1)edge(b2s) (c2)edge(b2s);
                        \begin{pgfonlayer}{background}
                            \fill[colB,face] (a1.center)--(a2.center)--(a3.center)--(a4.center)--(a5.center)--(a6.center)--cycle;
                            \fill[colB,face] (b1.center)--(c1.center)--(d2s.center)--cycle;
                            \fill[colB,face] (d1.center)--(c2.center)--(b2s.center)--cycle;
                            \foreach \bcd/\a in {b/1,c/6,d/5}
                            {
                                \fill[colR,face] (a\a.center)--(\bcd1.center)--(\bcd3.center)--(\bcd2.center)--cycle;
                            }
                        \end{pgfonlayer}

                    \end{scope}
                }
            }
            \end{scope}
            \end{scope}
            \fill[white,path fading=east] (-15.1,9)--(-14,9)--(-14,-9)--(-15.1,-9)--cycle;
            \fill[white,path fading=west] (15.1,9)--(14,9)--(14,-9)--(15.1,-9)--cycle;
            \fill[white,path fading=south] (-14,10.1)--(-14,9)--(14,9)--(14,10.1)--cycle;
            \fill[white,path fading=north] (-14,-10.1)--(-14,-9)--(14,-9)--(14,-10.1)--cycle;
            \begin{scope}
                \clip (13.99,9.003)rectangle(15.1,10.1);
                \fill[white,path fading=fade out] (14,9) circle[radius=1.1cm];
            \end{scope}
            \begin{scope}
                \clip (13.99,-9)rectangle(15.1,-10.1);
                \fill[white,path fading=fade out] (14,-9) circle[radius=1.1cm];
            \end{scope}
            \begin{scope}
                \clip (-14.004,9.003)rectangle(-15.1,10.1);
                \fill[white,path fading=fade out] (-14,9) circle[radius=1.1cm];
            \end{scope}
            \begin{scope}
                \clip (-14.004,-9)rectangle(-15.1,-10.1);
                \fill[white,path fading=fade out] (-14,-9) circle[radius=1.1cm];
            \end{scope}
        \end{tikzpicture}
  \caption{Tiling $[3^6;3^2.4.3.4;3^2.4.3.4]$ has two \APclasses{},
  both of them are \CnSymmetric{} for $\Cn$ corresponding to 2-fold  symmetry around the center of a red edge
  incident to two red triangles, 3-fold symmetry around the center of an ``isolated'' blue triangle,
  or 6-fold symmetry around a vertex of degree~6.
  Therefore, the flexes have respective $\Cn$-symmetries, when translated accordingly.}
  \label{fig:333333-33434-33434}
\end{figure}
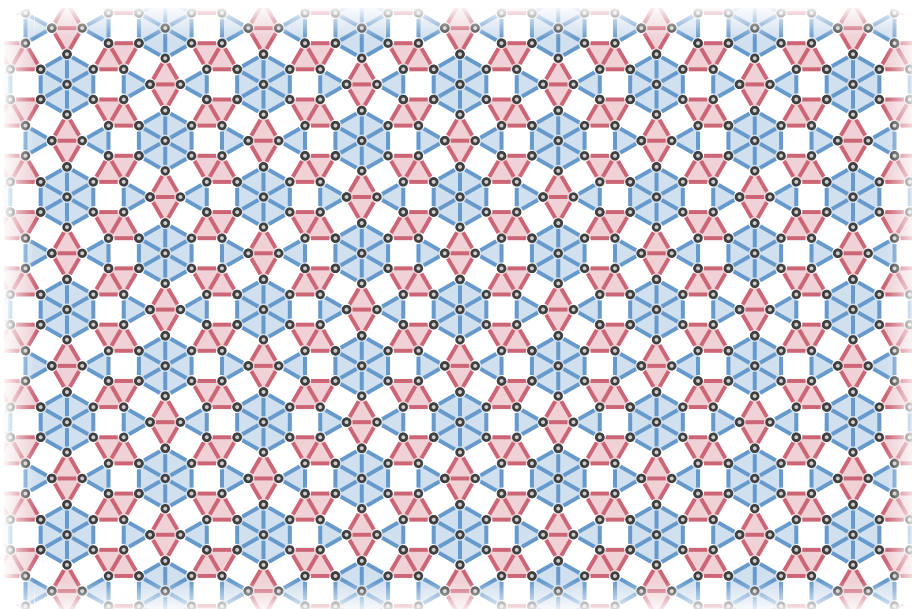

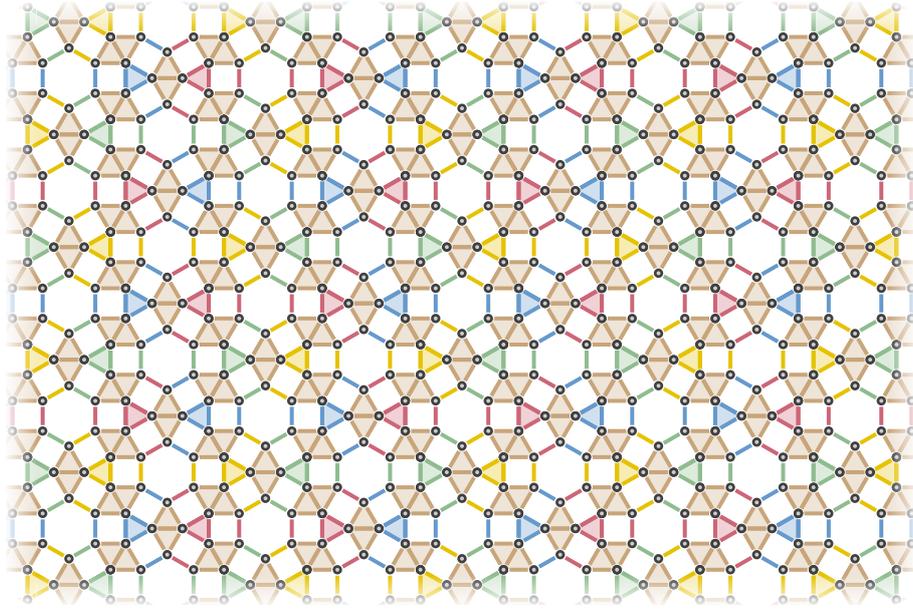
\begin{figure}[H]
    \centering
    \tikzsetnextfilename{3464-33434}
    \begin{tikzpicture}[scale=0.4]
            \begin{pgfonlayer}{background}
                \clip[rounded corners=0.4cm] (-15,10)rectangle ++(30,-20);
            \end{pgfonlayer}

            \begin{scope}
            \clip[rounded corners=0.4cm] (-15,10)rectangle ++(30,-20);
            \begin{scope}
            \foreach \y [evaluate=\y as \xs using {Mod(\y,2)/2-3},evaluate=\xs as \xse using \xs+5] in {-6,...,6}
            {
                \foreach \x [count=\xi,
                             evaluate=\x as \ci using {int(Mod(Mod(\xi+1,2)*2+\y,4)+1)},
                             evaluate=\x as \cj using {int(Mod(\ci+1,4)+1)},
                             evaluate=\x as \ck using {int(Mod(\ci+(-1+Mod(\y+1,2)*2)-1,4)+1)}
                             ] in {\xs,...,\xse}
                {
                    \begin{scope}[shift={($\x*2*(30:1)+\x*2*(-30:1)+\x*3*(0:1)+\y*(90:1)+\y*0.5*(60:1)+\y*0.5*(120:1)$)}]
                        \foreach \r [evaluate=\r as \rr using 60*\r-30] in {1,2,...,6}
                        {
                            \node[fvertex] (a\r) at (\rr:1) {};
                            \node[fvertex] (b\r) at ($(a\r)+(-60+\r*60:1)$) {};
                            \node[fvertex] (c\r) at ($(a\r)+(\r*60:1)$) {};
                            \draw[edge,col5]  (a\r)edge(b\r) (b\r)edge(c\r) (c\r)edge(a\r);
                        }
                        \draw[edge,col\ci] (a1)edge(a2) (a4)edge(a5);
                        \draw[edge,col\cj] (a2)edge(a3) (a5)edge(a6);
                        \draw[edge,col\ck] (a3)edge(a4) (a6)edge(a1);
                        \draw[edge,col\ci] (c1)edge(b2) (c4)edge(b5);
                        \draw[edge,col\cj] (c2)edge(b3) (c5)edge(b6);
                        \draw[edge,col\ck] (c3)edge(b4) (c6)edge(b1);
                        \node[fvertex] (d1) at ($(c1)+(90:1)$) {};
                        \draw[edge,col\ci] (c1)edge(d1) (d1)edge(b2);
                        \node[fvertex] (d2) at ($(c2)+(90+60:1)$) {};
                        \draw[edge,col\cj] (c2)edge(d2) (d2)edge(b3);
                        \begin{pgfonlayer}{background}
                            \foreach \i in {1,2,...,6}
                            {
                                \fill[col5,face] (a\i.center)--(b\i.center)--(c\i.center)--cycle;
                            }
                            \fill[col\ci,face] (b2.center)--(c1.center)--(d1.center)--cycle;
                            \fill[col\cj,face] (b3.center)--(c2.center)--(d2.center)--cycle;
                        \end{pgfonlayer}
                    \end{scope}
                }
            }
            \end{scope}
            \end{scope}
            \fill[white,path fading=east] (-15.1,9)--(-14,9)--(-14,-9)--(-15.1,-9)--cycle;
            \fill[white,path fading=west] (15.1,9)--(14,9)--(14,-9)--(15.1,-9)--cycle;
            \fill[white,path fading=south] (-14,10.1)--(-14,9)--(14,9)--(14,10.1)--cycle;
            \fill[white,path fading=north] (-14,-10.1)--(-14,-9)--(14,-9)--(14,-10.1)--cycle;
            \begin{scope}
                \clip (13.99,9.003)rectangle(15.1,10.1);
                \fill[white,path fading=fade out] (14,9) circle[radius=1.1cm];
            \end{scope}
            \begin{scope}
                \clip (13.99,-9)rectangle(15.1,-10.1);
                \fill[white,path fading=fade out] (14,-9) circle[radius=1.1cm];
            \end{scope}
            \begin{scope}
                \clip (-14.004,9.003)rectangle(-15.1,10.1);
                \fill[white,path fading=fade out] (-14,9) circle[radius=1.1cm];
            \end{scope}
            \begin{scope}
                \clip (-14.004,-9)rectangle(-15.1,-10.1);
                \fill[white,path fading=fade out] (-14,-9) circle[radius=1.1cm];
            \end{scope}
        \end{tikzpicture}
    \caption{Tiling $[3.4.6.4;3^2.4.3.4]$ is a subframework of \Cref{fig:333333-33434-33434}.
    It has 5 \APclasses{}.
    To obtain \CnSymmetric{} flexes with $\Cn$ given for instance
    by 3-fold rotation around the center of a yellow triangle,
    blue, green and red \APclasses{} are merged into a single \CnAPclass{}.
    }
    \label{fig:3464-33434}
\end{figure}

\begin{figure}[H]
    \centering
    \tikzsetnextfilename{33434-3464-3446}
    \begin{tikzpicture}[scale=0.4]
            \begin{pgfonlayer}{background}
                \clip[rounded corners=0.4cm] (-15,10)rectangle ++(30,-20);
            \end{pgfonlayer}

            \begin{scope}
            \clip[rounded corners=0.4cm] (-15,10)rectangle ++(30,-20);
            \begin{scope}
            \foreach \y [evaluate=\y as \xs using {Mod(\y,2)/2-4},evaluate=\xs as \xse using \xs+7] in {-6,...,6}
            {
                \foreach \x [count=\xi,
                             evaluate=\x as \ci using {int(Mod(Mod(\xi+1,2)*2+\y,4)+1)},
                             evaluate=\x as \cj using {int(Mod(\ci+1,4)+1)},
                             evaluate=\x as \ck using {int(Mod(\ci+(-1+Mod(\y+1,2)*2)-1,4)+1)}
                             ] in {\xs,...,\xse}
                {
                    \begin{scope}[shift={($\x*2*(30:1)+\x*2*(-30:1)+\x*(0:1)+\y*(90:1)+\y*0.5*(60:1)+\y*0.5*(120:1)$)}]
                        \node[fvertex] (a1) at (-0.5,0) {};
                        \node[fvertex] (a2) at (0.5,0) {};
                        \node[fvertex] (b1) at ($(a1)+(60:1)$) {};
                        \node[fvertex] (b2) at ($(a1)+(-60:1)$) {};
                        \draw[edge,col5]  (a1)edge(a2) (a1)edge(b1) (a1)edge(b2) (a2)edge(b1) (a2)edge(b2);
                        \node[fvertex] (ra1) at ($(a2)+(30:1)$) {};
                        \node[fvertex] (ra2) at ($(a2)+(-30:1)$) {};
                        \node[fvertex] (rb1) at ($(b1)+(30:1)$) {};
                        \node[fvertex] (rb2) at ($(b2)+(-30:1)$) {};
                        \draw[edge,col\ci] (a2)edge(ra1) (a2)edge(ra2) (ra1)edge(ra2) (b1)edge(rb1) (b2)edge(rb2);
                        \draw[edge,col5] (ra1)edge(rb1) (ra2)edge(rb2);
                        \node[fvertex] (la1) at ($(a1)+(150:1)$) {};
                        \node[fvertex] (la2) at ($(a1)+(210:1)$) {};
                        \node[fvertex] (lb1) at ($(b1)+(150:1)$) {};
                        \node[fvertex] (lb2) at ($(b2)+(210:1)$) {};
                        \draw[edge,col\cj] (a1)edge(la1) (a1)edge(la2) (la1)edge(la2) (b1)edge(lb1) (b2)edge(lb2);
                        \begin{pgfonlayer}{background}
                            \fill[col5,face] (a2.center)--(b1.center)--(a1.center)--(b2.center)--cycle;
                            \fill[col\ci,face] (a2.center)--(ra1.center)--(ra2.center)--cycle;
                            \fill[col\cj,face] (a1.center)--(la1.center)--(la2.center)--cycle;
                        \end{pgfonlayer}
                    \end{scope}
                }
            }
            \end{scope}
            \end{scope}
            \fill[white,path fading=east] (-15.1,9)--(-14,9)--(-14,-9)--(-15.1,-9)--cycle;
            \fill[white,path fading=west] (15.1,9)--(14,9)--(14,-9)--(15.1,-9)--cycle;
            \fill[white,path fading=south] (-14,10.1)--(-14,9)--(14,9)--(14,10.1)--cycle;
            \fill[white,path fading=north] (-14,-10.1)--(-14,-9)--(14,-9)--(14,-10.1)--cycle;
            \begin{scope}
                \clip (13.99,9.003)rectangle(15.1,10.1);
                \fill[white,path fading=fade out] (14,9) circle[radius=1.1cm];
            \end{scope}
            \begin{scope}
                \clip (13.99,-9)rectangle(15.1,-10.1);
                \fill[white,path fading=fade out] (14,-9) circle[radius=1.1cm];
            \end{scope}
            \begin{scope}
                \clip (-14.004,9.003)rectangle(-15.1,10.1);
                \fill[white,path fading=fade out] (-14,9) circle[radius=1.1cm];
            \end{scope}
            \begin{scope}
                \clip (-14.004,-9)rectangle(-15.1,-10.1);
                \fill[white,path fading=fade out] (-14,-9) circle[radius=1.1cm];
            \end{scope}
        \end{tikzpicture}
    \caption{Tiling $[3.3.4.3.4;3.4.6.4;3.4.4.6]$ has 5 \APclasses{}.
    If the origin is in the center of a brown edge incident to two brown triangles,
    then merging green \APclass{} with the red one and yellow with the blue one gives
    $\mathcal{C}_2$-symmetric \APclasses{} and flexes.}
    \label{fig:33434-3464-3446}
\end{figure}
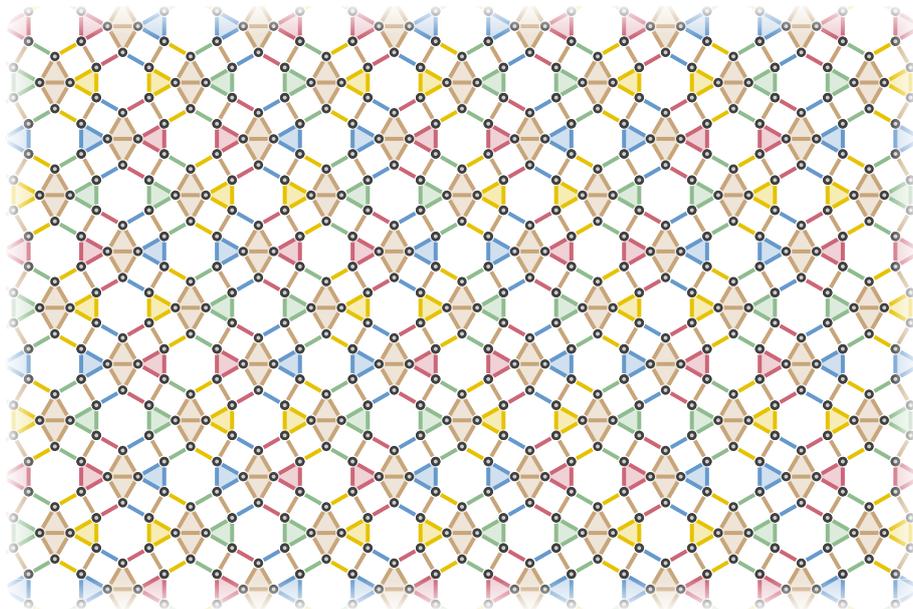

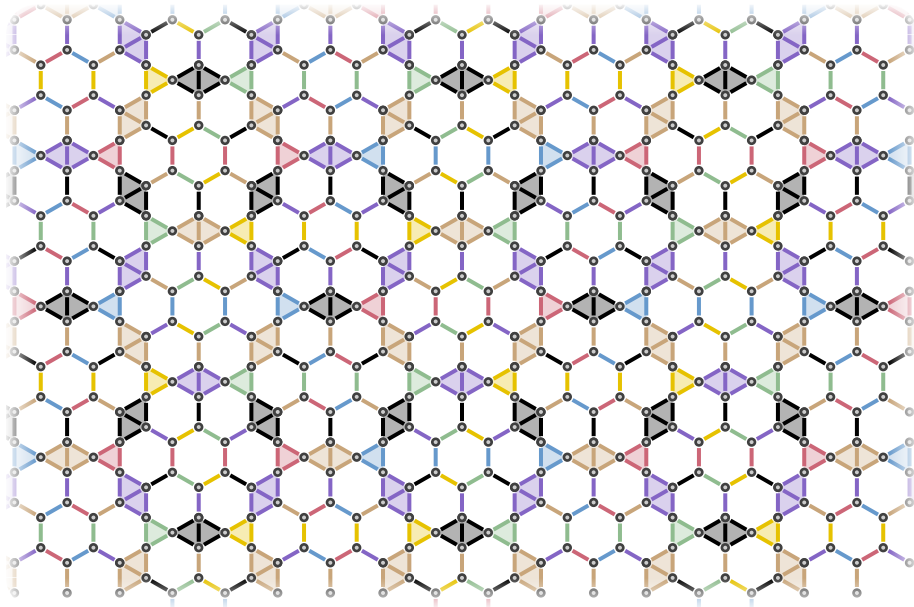
\begin{figure}[H]
    \centering
    \tikzsetnextfilename{3366-3636-666}
    \begin{tikzpicture}[scale=0.4]
            \begin{pgfonlayer}{background}
                \clip[rounded corners=0.4cm] (-15,10)rectangle ++(30,-20);
            \end{pgfonlayer}

            \begin{scope}
            \clip[rounded corners=0.4cm] (-15,10)rectangle ++(30,-20);
            \begin{scope}
            \foreach \y [evaluate=\y as \xs using {Mod(\y,2)/2-2},evaluate=\xs as \xse using \xs+3] in {-4,...,4}
            {
                \foreach \x [count=\xi,
                             evaluate=\x as \ci using {int(Mod(Mod(\xi+1,2)*2+\y,4)+1)},
                             evaluate=\x as \cj using {int(Mod(\ci+1,4)+1)},
                             evaluate=\x as \ck using {int(Mod(\ci+(-1+Mod(\y+1,2)*2)-1,4)+1)}
                             ] in {\xs,...,\xse}
                {
                    \begin{scope}[shift={($\x*5*(30:1)+\x*5*(-30:1)+\y*2.5*(90:1)$)}]
                        \foreach \r [evaluate=\r as \rr using 60*\r-30,evaluate=\r as \cl using {int(5+Mod(Mod(\y,3)+Mod(\r,2),3))},evaluate=\r as \cm using {int(5+Mod(Mod(\y,3)+Mod(\r+1,2),3))}] in {1,2,...,6}
                        {
                            \node[fvertex] (a\r) at (\rr:1) {};
                            \node[fvertex] (b\r) at ($(a\r)+(\rr:1)$) {};
                            \draw[edge,col\cm]  (a\r)edge(b\r);
                            \ifnum\r<4
                                \node[fvertex] (c\r) at ($(b\r)+(\rr:1)$) {};
                                \node[fvertex] (cr\r) at ($(b\r)+(\rr-60:1)$) {};
                                \node[fvertex] (cl\r) at ($(b\r)+(\rr+60:1)$) {};
                                \draw[edge,col\cl]  (b\r)edge(c\r) (b\r)edge(cl\r) (b\r)edge(cr\r) (c\r)edge(cl\r) (c\r)edge(cr\r);
                                \begin{pgfonlayer}{background}
                                    \fill[col\cl,face] (b\r.center)--(cl\r.center)--(c\r.center)--(cr\r.center)--cycle;
                                \end{pgfonlayer}

                            \fi

                        }
                        \draw[edge,col\ci] (a1)edge(a2) (a4)edge(a5);
                        \draw[edge,col\cj] (a2)edge(a3) (a5)edge(a6);
                        \draw[edge,col\ck] (a3)edge(a4) (a6)edge(a1);
                        \node[fvertex] (f1) at ($(cl1)+(90:1)$) {};
                        \draw[edge,col\ci] (cl1)edge(cr2) (cr2)edge(f1) (f1)edge(cl1);
                        \node[fvertex] (f2) at ($(cl2)+(90+60:1)$) {};
                        \draw[edge,col\cj] (cl2)edge(cr3) (cr3)edge(f2) (f2)edge(cl2);
                        \begin{pgfonlayer}{background}
                            \fill[col\ci,face] (cl1.center)--(cr2.center)--(f1.center)--cycle;
                            \fill[col\cj,face] (cl2.center)--(cr3.center)--(f2.center)--cycle;
                        \end{pgfonlayer}
                    \end{scope}
                }
            }
            \end{scope}
            \end{scope}
            \fill[white,path fading=east] (-15.1,9)--(-14,9)--(-14,-9)--(-15.1,-9)--cycle;
            \fill[white,path fading=west] (15.1,9)--(14,9)--(14,-9)--(15.1,-9)--cycle;
            \fill[white,path fading=south] (-14,10.1)--(-14,9)--(14,9)--(14,10.1)--cycle;
            \fill[white,path fading=north] (-14,-10.1)--(-14,-9)--(14,-9)--(14,-10.1)--cycle;
            \begin{scope}
                \clip (13.99,9.003)rectangle(15.1,10.1);
                \fill[white,path fading=fade out] (14,9) circle[radius=1.1cm];
            \end{scope}
            \begin{scope}
                \clip (13.99,-9)rectangle(15.1,-10.1);
                \fill[white,path fading=fade out] (14,-9) circle[radius=1.1cm];
            \end{scope}
            \begin{scope}
                \clip (-14.004,9.003)rectangle(-15.1,10.1);
                \fill[white,path fading=fade out] (-14,9) circle[radius=1.1cm];
            \end{scope}
            \begin{scope}
                \clip (-14.004,-9)rectangle(-15.1,-10.1);
                \fill[white,path fading=fade out] (-14,-9) circle[radius=1.1cm];
            \end{scope}
        \end{tikzpicture}
    \caption{Tiling $[3^2.6^2; 3.6.3.6; 6^3]$ has 7 \APclasses{}.}
    \label{fig:3366-3636-666}
\end{figure}

\begin{figure}[H]
    \centering
    \tikzsetnextfilename{3464-3464-3446}
    \begin{tikzpicture}[scale=0.4]
            \begin{pgfonlayer}{background}
                \clip[rounded corners=0.4cm] (-15,10)rectangle ++(30,-20);
            \end{pgfonlayer}

            \begin{scope}
            \clip[rounded corners=0.4cm] (-15,10)rectangle ++(30,-20);
            \begin{scope}
            \foreach \y [evaluate=\y as \xs using {Mod(\y,2)/2-2},evaluate=\xs as \xse using \xs+3] in {-4,...,4}
            {
                \foreach \x [count=\xi,
                             evaluate=\x as \ci using {int(Mod(Mod(\xi+1,2)*2+\y,4)+1)},
                             evaluate=\x as \cj using {int(Mod(\ci+1,4)+1)},
                             evaluate=\x as \ck using {int(Mod(\ci+(-1+Mod(\y+1,2)*2)-1,4)+1)}
                             ] in {\xs,...,\xse}
                {
                    \begin{scope}[shift={($\x*2*(30:1)+\x*2*(-30:1)+\x*6*(0:1)+\y*(90:1)+\y*(60:1)+\y*(120:1)$)}]
                        \foreach \r [evaluate=\r as \rr using 60*\r-30] in {1,2,...,6}
                        {
                            \node[fvertex] (a\r) at (\rr:1) {};
                            \node[fvertex] (b\r) at ($(a\r)+(-60+\r*60:1)$) {};
                            \node[fvertex] (c\r) at ($(a\r)+(\r*60:1)$) {};
                            \draw[edge,col5]  (a\r)edge(b\r) (b\r)edge(c\r) (c\r)edge(a\r);
                            \node[fvertex] (d\r) at ($(b\r)+(-60+\r*60:1)$) {};
                            \node[fvertex] (e\r) at ($(c\r)+(\r*60:1)$) {};
                            \draw[edge,densely dashed] (b\r)edge(d\r) (c\r)edge(e\r);

                        }
                        \draw[edge,col\ci] (a1)edge(a2) (a4)edge(a5);
                        \draw[edge,col\cj] (a2)edge(a3) (a5)edge(a6);
                        \draw[edge,col\ck] (a3)edge(a4) (a6)edge(a1);
                        \draw[edge,col\ci] (c1)edge(b2) (c4)edge(b5);
                        \draw[edge,col\cj] (c2)edge(b3) (c5)edge(b6);
                        \draw[edge,col\ck] (c3)edge(b4) (c6)edge(b1);
                        \node[fvertex] (f1) at ($(e1)+(90:1)$) {};
                        \draw[edge,col\ci] (e1)edge(d2) (e1)edge(f1) (d2)edge(f1);
                        \node[fvertex] (f2) at ($(e2)+(90+60:1)$) {};
                        \draw[edge,col\cj] (e2)edge(d3) (e2)edge(f2) (d3)edge(f2);
                        \begin{pgfonlayer}{background}
                            \foreach \i in {1,2,...,6}
                            {
                                \fill[col5,face] (a\i.center)--(b\i.center)--(c\i.center)--cycle;
                            }
                            \fill[col\ci,face] (e1.center)--(d2.center)--(f1.center)--cycle;
                            \fill[col\cj,face] (e2.center)--(d3.center)--(f2.center)--cycle;
                        \end{pgfonlayer}
                    \end{scope}
                }
            }
            \end{scope}
            \end{scope}
            \fill[white,path fading=east] (-15.1,9)--(-14,9)--(-14,-9)--(-15.1,-9)--cycle;
            \fill[white,path fading=west] (15.1,9)--(14,9)--(14,-9)--(15.1,-9)--cycle;
            \fill[white,path fading=south] (-14,10.1)--(-14,9)--(14,9)--(14,10.1)--cycle;
            \fill[white,path fading=north] (-14,-10.1)--(-14,-9)--(14,-9)--(14,-10.1)--cycle;
            \begin{scope}
                \clip (13.99,9.003)rectangle(15.1,10.1);
                \fill[white,path fading=fade out] (14,9) circle[radius=1.1cm];
            \end{scope}
            \begin{scope}
                \clip (13.99,-9)rectangle(15.1,-10.1);
                \fill[white,path fading=fade out] (14,-9) circle[radius=1.1cm];
            \end{scope}
            \begin{scope}
                \clip (-14.004,9.003)rectangle(-15.1,10.1);
                \fill[white,path fading=fade out] (-14,9) circle[radius=1.1cm];
            \end{scope}
            \begin{scope}
                \clip (-14.004,-9)rectangle(-15.1,-10.1);
                \fill[white,path fading=fade out] (-14,-9) circle[radius=1.1cm];
            \end{scope}
        \end{tikzpicture}
    \caption{Tiling $[3.4.6.4;3.4.6.4;3.4.4.6]$ has infinitely many \APclasses{} --- the dashed edges do not belong to the same class,
    but each \APclass{} is formed by edges intersected by a line orthogonal to them.}
    \label{fig:infinitelyMany}
\end{figure}
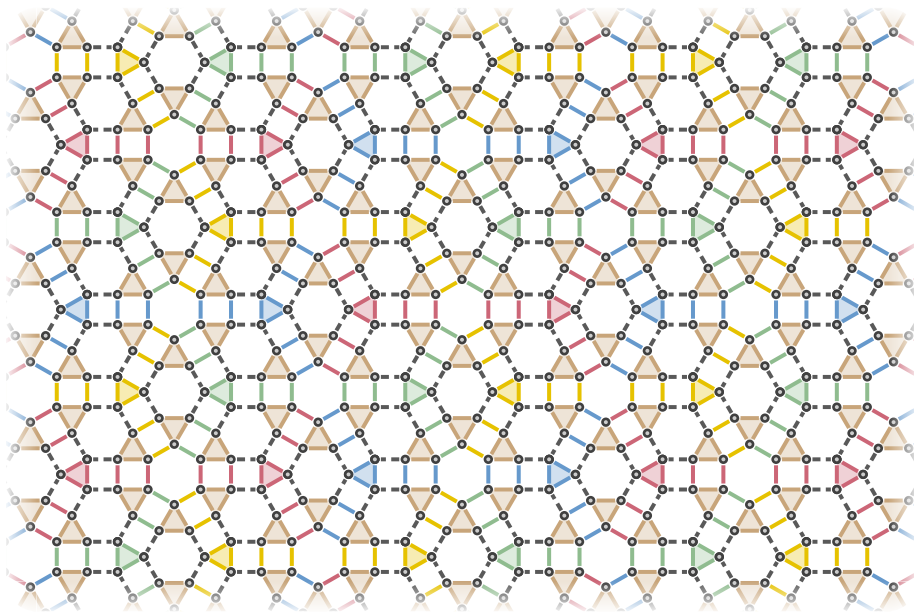

\tikzexternaldisable

\addcontentsline{toc}{section}{Acknowledgments}
\section*{Acknowledgments}
Georg Grasegger was supported by the Austrian Science Fund (FWF): I6233. For the purpose of open access, the authors have applied a CC BY public copyright license to any Author Accepted Manuscript version arising from this submission.
Jan Legerský was funded by the Czech Science Foundation (GAČR) project 22-04381L.
We would like to thank Walter Whiteley for pointing out the poster of his undergraduate students \cite{poster}
and for the discussions on the topic.

\phantomsection
\addcontentsline{toc}{section}{References}
\bibliographystyle{plainurl}
\bibliography{biblio.bib}

\appendix

\section{Verification of walk-independence}
In this section, we show how the walk-independence of a finite framework $(G,\rho)$
can be checked.
Let $T$ be a spanning tree of $G$.
For $e\in E_G\setminus E_T$, we define $C_e^T$ to be the unique cycle in $T+e$.
There are exactly $|E_G|-|V_G|+1$ such cycles.

\begin{proposition}
    \label{prop:zeroSumCheck}
    Let $(G,\rho)$ be a framework such that $\rho$ is a parallelogram placement.
    Let $T$ be a spanning tree of $G$.
    The framework is walk-independent if and only if \Cref{eq:zeroSumAPclass}
    holds for all cycles $C_e^T$.
\end{proposition}
In order to prove the proposition, we exploit the idea of a \emph{cycle basis}.
We recall the necessary notions as simple as possible.
See for instance~\cite{Kavitha2009} for more details on the topic.
\begin{definition}
    Let $\Goriented=(V_G,A)$ be an orientation of a finite graph $G$.
    A vector $\circulation{C} = (\circulation{C}(u,v))_{(u,v)\in A}\in \QQ^A$ is called a \emph{circulation}
    if for every $v\in V_G$ the \emph{conservation law} holds:
    \begin{equation*}
        \sum_{(u,v)\in A} \circulation{C}(u,v) = \sum_{(v,u)\in A} \circulation{C}(v,u)\,.
    \end{equation*}
    For a closed walk $C=(v_1,\ldots,v_\ell)$ in $G$, we define a circulation
    $\circulation{C}\in \QQ^A$ by
    \begin{equation*}
        \circulation{C}(w,w') = \sum_{\substack{(v_i,v_{i+1})\in C\\(v_i,v_{i+1})=(w,w')}}\!\!1\ - \sum_{\substack{(v_i,v_{i+1})\in C\\(v_i,v_{i+1})=(w',w)}}\!\!1\,.
    \end{equation*}
\end{definition}
The vector $\circulation{C}$ is indeed a circulation:
the conservation law at vertex $w$ is equivalent to
\begin{align*}
    &\sum_{(u,w)\in A} \sum_{\substack{(v_i,v_{i+1})\in C\\(v_i,v_{i+1})=(u,w)}}\!\!1\ + \sum_{(w,u)\in A} \sum_{\substack{(v_i,v_{i+1})\in C\\(v_i,v_{i+1})=(u,w)}}\!\!1\\
    = &\sum_{(u,w)\in A} \sum_{\substack{(v_i,v_{i+1})\in C\\(v_i,v_{i+1})=(w,u)}}\!\!1\ + \sum_{(w,u)\in A} \sum_{\substack{(v_i,v_{i+1})\in C\\(v_i,v_{i+1})=(w,u)}}\!\!1\,,
\end{align*}
where the left, resp.\ right, hand side gives the count on how many times the walk $C$
enters, resp.\ leaves, the vertex $w$.
Obviously, these two numbers are equal as $C$ is closed.

\begin{proof}[Proof of \Cref{prop:zeroSumCheck}]
    The ``only if'' part of the statement is trivial.

    Suppose that \Cref{eq:zeroSumAPclass} holds for all cycles $C_e^T$,
    we denote them by $C_1, \ldots, C_k$, where $k=|E_G|-|V_G|+1$.
    We fix an orientation $\Goriented=(V_G,A)$ of $G$.
    The circulations $\circulation{C}_1, \ldots, \circulation{C}_k$
    form a so called \emph{fundamental cycle basis} of the
    vector space of all circulations of $\Goriented$ with element-wise addition
    and scalar multiplication over $\QQ$ \cite{Kavitha2009}.
    Hence, for a closed walk $C$ in $G$,
    we have $\circulation{C}=\sum_{i=1}^k \alpha_i \circulation{C}_i$
    for some $\alpha_i\in \QQ$.
    Now, for an \APclass{} $r$, we have
    \begin{align*}
		\sum_{\oriented{u}{v} \in r\cap C} (\rho(v)-\rho(u))
		    &= \sum_{\oriented{u}{v} \in r\cap A} \circulation{C}(u,v)(\rho(v)-\rho(u))\\
		    &= \sum_{i=1}^k\alpha_i \sum_{\oriented{u}{v} \in r\cap A} \circulation{C}_i(u,v)(\rho(v)-\rho(u))\\
		    &= \sum_{i=1}^k\alpha_i \sum_{\oriented{u}{v} \in r\cap C_i} (\rho(v)-\rho(u)) = (0,0)\,.\qedhere
    \end{align*}
\end{proof}
Regarding the complexity of checking the walk-independence,
suppose we have $(G,\rho)$ with $n=|V_G|$, $m=|E_G|$, and with edges labeled by \APclasses{} they belong to.
Let $a$ denote the number of \APclasses{}.
A spanning tree $T$ can be found in ${O(m+n)=O(m)}$.
We fix a vertex $w$.
For every \APclass{} $r$ and vertex $w'$,
we compute the vector
\begin{align*}
	z_r(w') = \sum_{\oriented{u}{v} \in r\cap W} (\rho(v)-\rho(u))\,,
\end{align*}
where $W$ is the walk from $w$ to $w'$ in $T$.
These values can be obtained by traversing $T$,
namely, in $O(an)$ operations.
Now, the condition of the walk-independence holds for cycle $C_e^T$,
$e=uv$ belonging to the \APclass{} $r$,
if and only if $z_r(v)-z_r(u) = \rho(v) - \rho(u)$.
There are $O(m)$ cycles to check.
Hence, in total the complexity is $O(an + m)$.
Notice that $a\leq n-1$ if the walk-independence holds:
if there was an edge in $E_G\setminus E_T$ belonging to an \APclass{}
which dose not occur in $T$,
then the walk-independence is violated.
But $T$ has only $n-1$ edges.

\end{document}